\documentclass[aop,MSNbibl,seceqn,dvips]{arximspdf}
\usepackage{mathrsfs}
\usepackage{subenv}

\doi{10.1214/12-AOP777} %kopijuoti is PTS
\volume{41}
\issue{4}
\pubyear{2013}
\firstpage{2544}
\lastpage{2598}

\makeatletter
\renewcommand{\phi}{\varphi}
\newcommand{\rrvert}{\vert}
\newcommand{\llvert}{\vert}
\newcommand{\tr}{\operatorname{tr}}
\newcommand{\one}{\mathbf{1}}
\def\N{\mathbf{N}}
\def\Z{\mathbf{Z}}
\def\CQ{\mathcal{Q}}
\def\CW{\mathcal{W}}
\def\CR{\mathcal{R}}
\def\CI{\mathcal{I}}
\def\CO{\mathcal{O}}
\def\CP{\mathcal{P}}
\def\CA{\mathcal{A}}
\def\CS{\mathcal{S}}
\def\CD{\mathcal{D}}
\def\CF{\mathcal{F}}
\def\CG{\mathcal{G}}
\def\CH{\mathcal{H}}
\def\CJ{\mathcal{J}}
\def\CZ{\mathcal{Z}}
\def\CX{\mathcal{X}}
\def\CC{\mathcal{C}}
\def\CV{\mathcal{V}}
\def\CM{\mathcal{M}}

\def\loc{\operatorname{loc}}
\def\D{\mathbf{D}}
\def\E{\mathbb{E}}
\def\P{\mathbb{P}}
\def\bbR{\mathbb{R}}
\def\bbP{\mathbb{P}}
\def\PP{\mathrm{P}}

\newcommand{\eqdef}{\stackrel{\mathrm{def}}{=}}
\newcommand{\eqdefi}{\stackrel{\mathit{def}}{=}}
\def\TT{\mathrm{T}}
\def\FF{\mathrm{F}}

\def\C{\mathcal{C}}

\def\oc{\Omega\mathcal{C}}
\def\g{\gamma}
\def\b{\beta}
\def\bbX{\mathbb{X}}

\let\XX\bbX\let\YY\bbY
\def\DD{\mathrm{D}}
\def\A{\mathcal{A}}
\newcommand\td[1]{\tilde{#1}}
\newcommand{\tE}{{\tilde{\mathbb{E}}}} %macro for conditional
\def\R{\mathbb{R}}
\def\TV{\mathrm{TV}}
\def\A{\mathcal{A}}
\def\Q{\bar{\CQ}}
\def\CG{\mathcal{G}}
\def\eps{\varepsilon}
\newtheorem{theorem}{Theorem}[section]
\newproclaim{remark}{Remark}
\newtheorem{lemma}{Lemma}
\newproclaim{assumption}{Assumption}
\newproclaim{definition}{Definition}
\newtheorem{proposition}{Proposition}
\newtheorem{corollary}{Corollary}

\setattribute{abstract}{width}{280pt}
\makeatother

\begin{document}
\begin{frontmatter}

\title{Regularity of laws and ergodicity of hypoelliptic SDEs driven by rough paths}
\runtitle{Regularity of laws}

\begin{aug}
\author[A]{\fnms{Martin} \snm{Hairer}\corref{}\ead[label=e1]{Martin.Hairer@Warwick.ac.uk}\thanksref{t1}}
\and
\author[B]{\fnms{Natesh S.} \snm{Pillai}\ead[label=e2]{pillai@fas.harvard.edu}\thanksref{t2}}
\thankstext{t1}{Supported by EPSRC Grant EP/E002269/1, as
well as by the Royal Society through a Wolfson Research Merit Award and
by the Leverhulme
Trust through a Philip Leverhulme Prize.}
\thankstext{t2}{Supported by NSF Grant DMS-11-07070.}
\runauthor{M. Hairer and N. S. Pillai}
\affiliation{University of Warwick and Harvard University}
\address[A]{Mathematics Institute\\
University of Warwick\\
Coventry, CV4 7AL\\
United Kingdom\\
\printead{e1}}
%phantom{E-mail:\ }\printead*{e1}} %adresu isvedimo komanda gale!
\address[B]{Department of Statistics \\
Harvard University\\
Cambridge, Massachusetts 02138\\
USA\\
\printead{e2}}%\phantom{E-mail:\ }\printead*{e2}}
\end{aug}

% HISTORY:
\received{\smonth{4} \syear{2011}}
\revised{\smonth{3} \syear{2012}}

% ABSTRACT
%
\begin{abstract}
We consider differential equations driven by rough paths and study
the regularity of the laws and their long time behavior.
In particular, we focus on the case when the driving noise is a rough
path valued fractional
Brownian motion with Hurst parameter $H \in(\frac{1}{ 3},\frac{1}{
2}]$. Our contribution in this work is twofold.

First, when the
driving vector fields satisfy H\"ormander's celebrated ``Lie bracket
condition,'' we derive explicit quantitative bounds
on the inverse of the Malliavin matrix.
En route to this, we provide a novel ``deterministic'' version of
Norris's lemma
for differential equations driven by rough paths. This result, with the
added assumption that the linearized equation has
moments, will then yield that the transition laws have a smooth density
with respect to Lebesgue measure.

Our second main result states
that under H\"ormander's condition, the solutions to rough differential
equations driven by fractional Brownian
motion with $H \in(\frac{1}{ 3},\frac{1}{ 2}]$ enjoy a suitable version
of the strong Feller property. Under a standard controllability
condition, this implies that
they admit a unique stationary solution that is physical in the sense
that it does not ``look into the future.''
\end{abstract}

% KEYWORDS
% Pirmas kwd is didziosios raides
%
\begin{keyword}[class=AMS]
\kwd[Primary ]{60H07}
\kwd{60H10}
\kwd[; secondary ]{60G10}
\kwd{26A33}
\end{keyword}

\begin{keyword}
\kwd{H\"ormander's theorem}
\kwd{hypoellipticity}
\kwd{fractional Brownian motion}
\kwd{rough paths}
\end{keyword}

\end{frontmatter}

%s1 #&#
\section{Introduction} \label{secintro}
In this article, we consider stochastic differential equations of the form
%
%
%e1.1 #&#
\begin{equation}
\label{eqnSDErpi} dZ_t = V_0(Z_t) \,dt +
\sum_{i=1}^dV_i(Z_t)
\,dX^i_t,\qquad Z_0 = z \in\bbR^n ,
\end{equation}
where $X_t$ is a $d$-dimensional random rough path
\cite{TerryRough,TerryFlour,FrizVict10} and $V_0 , V_i \in\bbR^{n}$
are smooth vector fields. While
a large part of our work is deterministic and applies to a large class
of rough differential equations
driven by rough paths that are H\"older continuous with index greater
than $\frac{1}{ 3}$,
our probabilistic results focus
on the case when $X_t$ is a two-sided $d$-dimensional fractional
Brownian motion with Hurst
parameter $H \in(\frac{1}{3},\frac{1}{2}]$. Recall that the
fractional Brownian motion with Hurst parameter $H$ is the centered
Gaussian process such that $X_0 = 0$ and
\[
\E|X_t - X_s|^2 = |t-s|^{2H} .
\]
Differential equations driven by rough paths have been studied intensely
in the past decade, and this theory has now reached a certain level of maturity;
we refer to the monographs~\cite{MR2036784,TerryFlour,FrizVict10}
for an overview of the theory.
For driving signals that are rougher than Brownian motion, the theory
of rough paths has provided a systematic way of
constructing solutions to differential equations of the type (\ref
{eqnSDErpi}) in a way that is ``natural,'' in the sense that
solutions are limits of approximate solutions where $X$ is replaced by
a smoothened version.

When the noise $X_t$ is a replaced by a standard Brownian motion $B_t$,
it has been well known
since the groundbreaking work of H\"ormander~\cite{HorActa} that, for
the laws of the Markov process $Z_t$ to have a smooth transition
density, it is sufficient that
the Lie algebra formed by $\{\partial_t + V_0, V_1,\ldots,V_d \}$
spans $\bbR^{n+1}$ at every
point; see Assumption~\ref{assHormander} for a
precise formulation. The formalism
of Malliavin calculus was invented to give a probabilistic proof
of this result~\cite{Mal76,Mal97SA,Nual06,KSAMI,KSAMII,KSAMIII}. The
smoothness of transition densities coupled with some mild
controllability assumptions will then yield that
the system (\ref{eqnSDErpi}) has a unique invariant measure.

When the driving noise $X$ is a fractional Brownian motion with $H \neq
\frac{1 }{ 2}$, solutions to (\ref{eqnSDErpi}) are
neither a Markov process nor a semimartingale, so standard tools from
stochastic calculus break down.
%Even if the driving noise $X$ does not have the Markov property, it is
%clear that we can always turn \eref{eqn:SDErpi}
%into a Markov process by augmenting the state space to include the
%whole past of $X$. If the
%augmented state space is chosen judiciously, it is even often possible
%to ensure that the
%resulting Markov process has the Feller property (see for example
Inspired
by the results in the case of Brownian motion, two natural questions in
this context are to identify conditions under which:
\begin{longlist}[(1)]
\item[(1)] the ``transition densities'' of (\ref{eqnSDErpi}) are smooth;
\item[(2)] the system (\ref{eqnSDErpi}) has a ``unique invariant measure.''
\end{longlist}
Since $Z$ is not Markov in general, it does not really make sense to
speak of transition probabilities,
but the first question still makes sense by, for example, considering
the law of the solution at some
time $t>0$, given an initial condition $Z_0$, conditional on the
realization of $\{X_s \dvtx s \le0\}$.
Similarly, the notion of an ``invariant measure'' does not make
immediate sense for non-Markovian processes.
This problem has been discussed extensively in~\cite{Hair05,Hair09},
where a notion of an invariant measure
adapted to systems of the type (\ref{eqnSDErpi}) is introduced.
Essentially, these are stationary solutions to (\ref{eqnSDErpi})
that are ``physical'' in the sense that they are independent of the
innovation of~$X$.\looseness=1

In recent years, the SDE (\ref{eqnSDErpi}) was studied
when the driving noise $X$ is a fractional Brownian motion with Hurst
parameter $H > \frac{1}{ 2}$.
In this case, the answers to both of the above questions are completely
settled in
a series of papers~\cite{Hair05,BaudHair07,HairOhas07,NuaSau09,HairPill10}.
In particular, it was shown in~\cite{BaudHair07,HairPill10} that
the solutions to (\ref{eqnSDErpi}) have
smooth ``transition densities'' when the vector fields satisfy H\"ormander's
condition. It was also shown that if furthermore the control system
associated to (\ref{eqnSDErpi})
is approximately controllable then, under suitable
dissipativity and boundedness conditions on the vector fields~$V_i$,
(\ref{eqnSDErpi}) also admits a unique
invariant measure. However, the question of smoothness of laws in the
case of the driving
noise $X$ being a fractional Brownian motion with Hurst parameter $H <
\frac{1}{ 2}$ was completely
open until now, despite substantial recent progress in particular cases
\cite{Driscoll,HuTind11}.
The only general result in the context of rough paths theory was
obtained in~\cite{CassFriz}, where the authors
obtained the existence of densities with respect to Lebesgue measure
under H\"ormander's condition
for a large class of driving noises.

In this paper we largely settle the above two questions when $X$ is a
fractional Brownian motion with $H \in(\frac{1}{ 3}, \frac{1}{ 2}]$.
An important component underlying the probabilistic proofs of the
smoothness of transition densities is Norris's lemma \cite
{KSAMII,Norr86}, which roughly speaking states that if a
semimartingale is small and if one has a priori bounds on the
regularity of its components, then its bounded variation
part and the local martingale part are also small.
In this regard this lemma can be considered as a quantitative version
of the classical
Doob--Meyer decomposition theorem. A version of Norris's lemma for
fractional Brownian motion
with $H > \frac{1}{ 2}$ was proved in~\cite{BaudHair07}.
The recent work~\cite{HuTind11}, which appeared as the present
article was nearing completion,
contains a version of Norris's lemma for $H \in(\frac{1}{ 3}, \frac
{1}{ 2})$ that is similar in spirit to the
one in~\cite{BaudHair07}.

Our contribution in this work is twofold. First, we prove a \textit
{deterministic} version of Norris's lemma
for general integrals against rough paths. This may sound strange at
first since Norris's
seems to be the prototype of a probabilistic statement, and the whole
point of rough path theory
is to get rid of stochastic calculus and replace it by a deterministic theory.
We reconcile these conflicting perspectives by first showing an
estimate strongly resembling that
of Norris's lemma for processes of the form $Z_t = \int_0^t A_s \,dX_s
+ \int_0^t B_s \,ds$,
where $X$ is a rough path, and $A$ is a rough path ``controlled by~$X$;'' see
Section~\ref{secprelims} below for precise definitions. This estimate
makes use of a quantity that we
call the ``modulus of H\"older roughness'' of $X$, $L_\theta(X)$. See
Definition~\ref{defnsthc} below for the
precise definition of $L_\theta$.
In a second step, we then show that if $X$ is fractional Brownian
motion with
$H \leq\frac{1}{ 2}$, then $L_\theta(X)$ is almost surely positive
for $\theta> H$ and has inverse
moments of all orders. A loose formulation of our main result is as
follows (see Theorem~\ref{thmNlemmaRP}
and Lemma~\ref{lemboundRoughfBm} below for precise formulations that
include the exact dependency of
$M$ on $X$, $A$ and $B$):

%
%th1.1 #&#
\begin{theorem}
Let $X$ be a $\gamma$-H\"older continuous rough path in $\R^n$ with
$\gamma> \frac{1}{ 3}$, let $A$ be a rough path
in $\R^n$ controlled by $X$, let $B$ be a $\gamma$-H\"older
continuous function and set
%
%
%e1.2 #&#
\begin{equation}
\label{edefZ1} Z_t = \int_0^t
A_s \,dX_s + \int_0^t
B_s \,ds .
\end{equation}
Then if $X$ is $\theta$-H\"older rough for some $\theta< 2\gamma$,
there exist constants $r>0$ and $q>0$ such that one has the bound
\[
\|A\|_{\infty} + \|B\|_{\infty} \le M L_\theta(X)^{-q}
\|Z\|_{\infty}^r
\]
for a constant $M$ depending polynomially on the $\gamma$-H\"older
``norms'' of $X$, $A$ and~$B$.
Here, $\|\cdot\|_\infty$ denotes the supremum norm over the interval $[0,1]$.

Furthermore, if $X$ is the rough path canonically associated to
fractional Brownian motion
with $H \leq\frac{1}{ 2}$, then $\E L_\theta^{-p}(X) < \infty$ for
every $\theta> H$ and every $p > 0$.
\end{theorem}

%
%re1 #&#
\begin{remark}
Note that this immediately tells us that if $X$ is H\"older rough, then
it admits a kind of Doob--Meyer decomposition
in the sense that the processes $A$ and $B$ in (\ref{edefZ1}) are
uniquely determined by $Z$.
An interesting fact is that H\"older roughness is a \textit
{deterministic} property. In principle, one could
imagine being able to check that this property holds almost surely for
a number of driving noises, not
even necessarily Gaussian ones.
\end{remark}

Combined with standard arguments, this result yields quantitative
bounds on the
inverse of the Malliavin matrix, thus obtaining a quantitative version
of the result obtained in~\cite{CassFriz},
where the authors showed via a 0--1 law argument that the Malliavin
matrix is almost surely invertible.
If we use the additional assumption that the linearization of the RDE
(\ref{eqnSDErpi}) has moments of all orders,
our results also yield that~(\ref{eqnSDErpi}) has \textit{smooth
densities} thus
extending the work pioneered by Malliavin to the case in which the
driving noise
is a rough path.
In fact the very recent work~\cite{CLL11} obtains such moment bounds
for the linearization of (\ref{eqnSDErpi}) under certain boundedness
conditions, and thus our result immediately yields smoothness of
densities for a large class of RDEs of the form
(\ref{eqnSDErpi}).
Even in the case $H = \frac{1}{ 2}$, we believe that the bounds
derived in this work are new and pave the way
for more quantitative versions of H\"ormander's theorem.

When this work was nearing completion, we were notified of an
independent work~\cite{HuTind11} showing smoothness of densities to
solutions to (\ref{eqnSDErpi})
in the case when the driving vector fields exhibit a ``nilpotent''
structure, which allows one to obtain a priori
bounds on the Malliavin derivatives of the solutions. The work \cite
{HuTind11} also contains a version
of Norris's lemma in the context of SDEs driven by fractional Brownian
motion with Hurst parameter $H \in(\frac{1}{ 3},\frac{1}{ 2})$.

In the second half of the paper, we show that under an additional
controllability assumption,
the SDE (\ref{eqnSDErpi}) has
a unique invariant measure,\vadjust{\goodbreak} which follows from the strong Feller
property defined in~\cite{HairOhas07}.
Note that, thanks to a cutoff argument,
we do not need to require that the linearized equation has bounded
moments to obtain this uniqueness result.
If we denote by $A_{t,z}$ the closure of the set of all points that are
accessible at time $t$
for solutions to the control problem associated to (\ref{eqnSDErpi})
starting at $z$,
the second major result of this paper is the following:

%
%th1.2 #&#
\begin{theorem}\label{thmmain2}
Assume that the vector fields $\{V_i\}$ have derivatives with at most
polynomial growth, and that
(\ref{eqnSDErpi}) has global solutions.
Then, if H\"ormander's bracket condition holds at every point, (\ref
{eqnSDErpi}) is strong Feller.

In particular, if there exists $t>0$ such that $\bigcap_{z \in\R^n}
A_{t,z} \neq\varnothing$,
then (\ref{eqnSDErpi}) admits at most one invariant measure in the
sense of~\cite{Hair05}.
\end{theorem}

The above result gives us the uniqueness of the invariant measure for
system~(\ref{eqnSDErpi}). Theorem~\ref{thmmain2}, combined with the
assumption of the
existence of an invariant measure, will yield that system (\ref
{eqnSDErpi}) is ergodic.

The remainder of the article is structured as follows. In Section~\ref
{secprelims} we review the framework of
controlled rough paths from~\cite{Gubi04} and set up the notation and
derive some preliminary estimates.
In Section~\ref{secNorris}, we then prove a general deterministic
version of Norris's lemma for SDEs driven by
rough paths. Furthermore we show that our assumptions are almost surely
satisfied by the sample paths of fractional Brownian motion.
Section~\ref{secLpflow} is a rather technical section in which we
show that solutions to (\ref{eqnSDErpi})
are smooth in the sense of Malliavin calculus and obtain a priori
bounds on their Malliavin derivatives.
We then obtain quantitative bounds on the lowest eigenvalue of the
Malliavin matrix in Section~\ref{secinvmmatrix}. Using the results in
that section, we show that
the existence of moments of the derivative of the flow implies the
smoothness of
the transition densities. In Section~\ref{secergodic}, we show
the ergodicity of SDEs driven by fractional Brownian motion under H\"
ormander's condition and a standard
controllability assumption.
In Section~\ref{secexamples}, we finally give a few examples where
our results are applicable.
%The proofs of few technical lemmas are given in the Appendix.

%s2 #&#
\section{Preliminaries} \label{secprelims}

%s2.1 #&#
\subsection{Notation}

Throughout this article, we will make use of the following notation.
For quantities
$E$ and $R$, we write $E \leq\mathrm{K}(R)$
as a shorthand to mean that there exists a continuous increasing
function $b\dvtx\mathbb{R}_+ \mapsto\mathbb{R}_+$
such that the bound $E \leq b(R)$ holds. Note that the function $b$ in
question is unspecified
and may change from one line to the next. We also use the letter
$M$ to denote an arbitrary (possibly problem-dependent) constant
whose precise value might vary from one line to the next.

%s2.2 #&#
\subsection{Introduction to the theory of rough paths}

In this work, we adopt the framework of~\cite{Gubi04} which offers a
slightly different
perspective on the pioneering work of Terry Lyons~\cite{TerryRough}.\vadjust{\goodbreak}

Denote by $\oc$ the set of
continuous functions from $\bbR^{2}$ to $\bbR$ which are $0$ on the diagonal
and define the ``increment'' operator $\delta\dvtx \C\mapsto\oc$ by
%
%
%e2.1 #&#
\begin{equation}\label{eqndel} \delta A_{st} \eqdef A_t - A_s .
\end{equation}
%
%This operator is going to be crucial later on.
For a fixed final time $T > 0$ and a continuous function $f\dvtx[0,T]
\mapsto\bbR^n$, set
%
%
%e2.2 #&#
\begin{equation}
\|f\|_{\infty} = \sup_{t \in[0,T]} \bigl|f(t)\bigr| ,\qquad\|f\|_{\gamma} =
\sup_{s,t \in[0,T]} \frac{|\delta
f_{st}|}{|t-s|^\gamma} . \label{eqnHoldnorm}
\end{equation}
We also define the norm
\[
\|f \|_{\mathcal{C}^\g} = \|f\|_{\infty} + \|f\|_{\g} .
\]

With these notation, a \textit{rough path} on the interval $[0,T]$
consists of two parts, a continuous function $X \dvtx[0,T] \mapsto
\mathbb
{R}^d$ and a continuous ``area process'' $\bbX\dvtx [0,T]^2 \mapsto
\bbR^{d\times d}, \bbX\in\oc$ satisfying the
algebraic identity
%
%
%e2.3 #&#
\begin{equation}
\label{eqnrpalgid} \bbX^{ij}_{st} - \bbX^{ij}_{ut}-
\bbX^{ij}_{su} = \delta X^i_{su}
\delta X^j_{ut}
\end{equation}
for all $\{s,u,t\} \in[0,T]$ and $1 \leq i,j \leq d$. For $\bbX\in
\oc$, define
%
%
%e2.4 #&#
\begin{equation}
\label{eqnlareadnm} \|\bbX\|_{2\g} \eqdef\mathop{\sup_{s,t \in
[0,T] }}_{ s \neq t}
\frac{|\bbX_{st}| }{ |t-s|^{2\g}} .
\end{equation}

For $\g\in(\frac{1}{ 3}, \frac{1}{ 2}]$, we denote by $\mathcal
{D}^\gamma([0,T],\R^d)$ the space of rough paths,
consisting of those pairs $(X, \bbX)$ satisfying (\ref{eqnrpalgid})
and such that
\[
\bigl\|(X,\bbX)\bigr\|_{\g} \eqdef\|X\|_{\g} + \|\bbX\|_{2\g}
< \infty.
\]
Notice that $\|(X,\bbX)\|_{\g}$ is only a semi-norm and that
$\mathcal{D}^\gamma$ actually is not
a vector space, due to the nonlinear constraint (\ref{eqnrpalgid}).

For every smooth function $X\colon[0,T]\to\R^d$, there exists a
canonical representative
in $\mathcal{D}^\gamma$ by choosing
\[
\XX_{s,t} = \int_s^t \delta
X_{sr} \otimes\,dX_r .
\]
We then denote by $\mathcal{D}_g^\gamma$ the closure of the set of
smooth functions
in $\mathcal{D}^\gamma$. (Here, $g$ stands for ``geometric.'') The
space $\mathcal{D}_g^\gamma$ has
the nice feature of being a Polish space~\cite{FrizVict10},
Proposition~8.27, which will be useful in the sequel.

%s2.3 #&#
\subsection{Controlled rough paths}

For defining integrals with respect to rough paths, a useful notion
introduced first in~\cite{Gubi04} is that of
``controlled'' paths:
%
%
%de1 #&#
\begin{definition}
Let $(X, \bbX) \in\mathcal{D}^\g([0,T],\R^d)$ for some $\gamma\in
(\frac{1}{ 3},\frac{1}{ 2}]$. A pair $(Z,Z')$ is said to be
controlled by $X$ if
$Z \in\C^\gamma([0,T],\bbR^n)$, $Z' \in\C^\gamma([0,T], \bbR
^{n\times d})$, and the ``remainder'' term $R^Z \in\oc$
implicitly defined by
%
%
%e2.5 #&#
\begin{equation}
\label{eqnweakcon} \delta Z_{st} = Z'_s
\delta X_{st} + R^Z_{st} ,
\end{equation}
satisfies $\|R^Z\|_{2\g} < \infty$.
\end{definition}

Denoting by $\mathcal{C}_{X}^{\g}$ the set of paths controlled by $X$,
we endow it with the norm
\[
\bigl\|\bigl(Z,Z'\bigr)\bigr\|_{X,\g} = \bigl|Z(0)\bigr| + \bigl\|Z'
\bigr\|_{\mathcal{C}^\g} + \bigl\|R^Z\bigr\|_{2\g
} .
\]
As noticed in~\cite{Gubi04}, now we may \textit{define} the integral
of a weakly controlled path $(Z,Z') \in\mathcal{C}_{X}^{\g}$ with respect
to a rough path $(X,\bbX)$ by taking a limit of modified Riemann sums:
%
%
%e2.6 #&#
\begin{eqnarray}
\label{eqnintrp} \int_{0}^T Z_t
\otimes dX_t = \lim_{|\PP| \rightarrow0} \sum_{[s,t] \in\PP}
\bigl(Z_s \otimes\delta X_{st} + Z'_s
\bbX_{st} \bigr) ,
\end{eqnarray}
where $\PP$ is a finite partition of the interval $[0,T]$ into
subintervals and $|\PP|$ denotes the length of the largest
subinterval. The following result, adapted from~\cite{Gubi04},
Proposition 1,
gives the continuity of the integral with respect to its integrand:

%
%th2.1 #&#
\begin{theorem}\label{thmGubi}
Let $(X,\bbX) \in\mathcal{D}^\g([0,T],\R^d)$ for some $\g> \frac
{1}{ 3}$ and $(Y,Y') \in\mathcal{C}_{X}^{\g}$ be
a weakly controlled rough path. Then the map
\[
\bigl(Y,Y'\bigr) \mapsto\bigl(Z,Z'\bigr) \eqdefi
\biggl(\int_{0}^\cdot Y_t \otimes
dX_t, Y \otimes I \biggr) ,
\]
where the integral is as defined in (\ref{eqnintrp}), is continuous
from $\mathcal{C}_{X}^{\g}$ to $\mathcal{C}_{X}^{\g}$, and
furthermore we have the bound
%
%
%e2.7 #&#
\begin{eqnarray}
\label{eqnrpgubibd} \bigl\|R^Z\bigr\|_{2\g} \leq M \bigl(\|X
\|_{\g} \bigl\|R^Y\bigr\|_{2\g} + \|\bbX\|_{2\g}
\bigl\|Y'\bigr\|_{\mathcal{C}^\g} \bigr)
\end{eqnarray}
for a constant $M$ independent of $X,Y$.
\end{theorem}

%
%re2 #&#
\begin{remark}\label{remZZprimebd}
Notice that from (\ref{eqnrpgubibd}) we deduce that
%
%
%e2.8 #&#
\begin{equation}
\label{eqnZZprimebd} \bigl\|\bigl(Z,Z'\bigr)\bigr\|_{X,\g} \leq M
\bigl(\|X\|_{\g} \bigl(\bigl\|R^Y\bigr\|_{2\g} + \|Y
\|_{\mathcal{C}^\g}\bigr) + \|\bbX\|_{2\g}\bigl\|Y'
\bigr\|_{\mathcal{C}^\g} \bigr) .
\end{equation}
\end{remark}

For $(Y,Y') \in\mathcal{C}_{X}^{\g}$ and a $C^2$ function $\psi
\dvtx
\bbR^n
\mapsto\bbR^m$,
we may define a new weakly controlled rough path $(\psi(Y), \psi(Y)')
\in\mathcal{C}_{X}^{\g}$ as
%
%
%e2.9 #&#
\begin{equation}
\label{eqnphirp} \psi(Y)_t = \psi(Y_t),\qquad
\psi(Y)'_t = D\psi(Y_t) Y'_t
.
\end{equation}
Then we have the following bound from~\cite{Hair10}, Lemma 2.2:
%
%
%le1 #&#
\begin{lemma} \label{lemphiYlem}
Let $(Y,Y') \in\mathcal{C}_{X}^{\g}$ and $(\psi(Y), \psi(Y)')$ be
as defined
in (\ref{eqnphirp}).
Then we have
\begin{eqnarray*}
&&\bigl\|\bigl(\psi(Y), \psi(Y)'\bigr)\bigr\|_{X,\g}\\
&&\qquad \leq M\bigl
(1 +
\|\psi\|_\infty+ \bigl\|D^2 \psi\bigr\|_\infty\bigr) \bigl(1 +
\bigl\|(X,\bbX)\bigr\|_{\g}\bigr) \bigl(1 + \bigl\|\bigl
(Y,Y'\bigr)
\bigr\|_{X,\g
}\bigr)^2 ,
\end{eqnarray*}
where the supremum norms of $\psi$ and $D^2\psi$ are taken over the
ball of radius $\|Y\|_\infty$, and the constant $M$ is independent of
$X,Y, \psi$.
\end{lemma}

%s2.4 #&#
\subsection{Notion of solution}

With all of these notation at hand, we give the following definition
of a solution to (\ref{eqnSDErpi}):

%
%de2 #&#
\begin{definition}
Let $\gamma> \frac{1}{ 3}$, and let $(X,\bbX) \in\CD^\gamma$.
Then $Z \in\CC^\gamma$ is
a solution to~(\ref{eqnSDErpi}) if $(Z,Z') = (Z,V(Z)) \in\CC^\gamma
_X$, and the integral version of (\ref{eqnSDErpi})
holds, where the composition of a controlled rough path with a
nonlinear function is interpreted as in (\ref{eqnphirp})
and the integral of a controlled rough path against $X$ is defined by
(\ref{eqnintrp}).
Here, we denoted by $V$ the collection $(V_1,\ldots,V_d)$.
\end{definition}

A standard fixed point argument, as given, for example, in \cite
{TerryRough,Gubi04}, then yields:

%
%th2.2 #&#
\begin{theorem} \label{thmlocsol}
For $V \in\CC^3$, there exists a unique local solution to (\ref{eqnSDErpi}).
\end{theorem}

From now on, we will refer to this notion of a solution to (\ref{eqnSDErpi}).
% without further specifying it.

%
%re3 #&#
\begin{remark}
In Theorem~\ref{thmlocsol}, if the vector fields $V$ are bounded with
bounded derivatives,
then there exists a global solution~\cite{TerryRough}.
\end{remark}

%s3 #&#
\section{A version of Norris's lemma} \label{secNorris}

One of the main ingredients of the proof of H\"ormander's theorem using
Malliavin calculus is Norris's lemma, which
is essentially a quantitative version of the Doob--Meyer decomposition
theorem. Loosely speaking, it states
that under certain additional regularity assumptions, if a
semimartingale is ``small,'' then both its bounded variation
part and its martingale part have to be ``small'' separately. In other
terms, if we have some a priori knowledge
of the regularity of a semimartingale, then there is a limit to the
amount of cancellations that can occur between
the two terms in its Doob--Meyer decomposition. The intuitive reason
for this is that a continuous martingale is
nothing but a time-changed Brownian motion, and so it has to be very
rough at every single scale.

Results of this type are usually considered to be the archetype of a
probabilistic result. The aim of this
section is to argue that while the probabilistic intuition described
above is certainly correct, one can have a much more pathwise perspective
on Norris's lemma. This was already apparent in~\cite{HairMatt09},
where the authors obtain a
result that is similar in flavor to Norris's lemma, but where this
lack of cancellations is formulated as a
\textit{deterministic} property that occurs on a \textit{universal}
``large'' subset of Wiener space.
Here, we take this viewpoint one step further by exhibiting a universal
set on which
a quantitative version of Norris's lemma holds as a deterministic property.

The main ingredient in our pathwise perspective is the following
definition that makes precise what we mean by
a path that is ``rough at every scale'':

%
%de3 #&#
\begin{definition} \label{defnsthc}
A path $X_t$ with values in $\R^n$ is said to be \textit{$\theta$-H\"
{o}lder rough}
in the interval $[0,T]$ for $\theta\in(0,1)$, if there exists a
constant $L_{\theta}(X)$ such that
for every $s \in[0,T]$, every $\eps\in(0,T/2]$ and every $\phi\in
\bbR^n$ with $\|\phi\| = 1$, there exists $t \in[0,T]$ such that
%
%
%e3.1 #&#
\begin{equation}
\label{eHolderRough} |t - s| \le\eps\quad\mbox{and}\quad\bigl
|\langle\phi, \delta
X_{s,t} \rangle\bigr| > L_{\theta}(X) \eps^{\theta} .
\end{equation}
We denote the largest such $L_\theta$ the ``modulus of $\theta$-H\"
older roughness of $X$.''
\end{definition}

%
%re4 #&#
\begin{remark}
We emphasize that the choice of quantifiers in the above definition
ensures that such H\"older rough paths
actually do exist. In particular, as soon as $n \ge2$, it is essential
to allow the precise location of $t$ such that (\ref{eHolderRough})
holds to depend on the vector $\phi$.
\end{remark}

A first, rather straightforward consequence of this definition is that
if a rough path $(X,\bbX)$ happens to be
H\"older rough, then the ``derivative process'' $Z'$ in the
decomposition (\ref{eqnweakcon}) of a controlled rough path
is uniquely determined by $Z$. This can be made quantitative in the
following way:

%
%pr1 #&#
\begin{proposition} \label{propNlemma1}
Let $(X, \bbX) \in\CD^\g([0,T],\R^n )$ be such that $X$
is a $\theta$-H\"older rough path.
Then there exists a constant $M$ depending only on $n$ and $m$
such that the bound
\[
\bigl\|Z'\bigr\|_{\infty} \le\frac{M \|Z\|_\infty}{ L_{\theta
}(X)} \bigl(\bigl\|
R^Z\bigr\|_{2\g}^{{\theta}/ { (2\g)}} \|Z
\|_\infty^{- {\theta
}/{( 2\g)}}\vee T^{-\theta} \bigr) ,
\]
holds for every controlled rough path $(Z,Z') \in\mathcal{C}_{X}^{\g
}
([0,T],\R^m )$.
%Here $\|Z'_t\|$ denotes the Frobenius norm of the matrix $Z_t$.
\end{proposition}
\begin{pf}
Fix $s \in[0,T]$ and $\eps\in(0,T/2]$,
From the definition of the remainder $R^Z$ in (\ref{eqnweakcon}), it
then follows that
%
%
%e3.2 #&#
\begin{equation}
\sup_{|t-s| \le\eps} \bigl| Z'_s \delta X_{s,t}
\bigr| \leq\sup_{|t-s| \le\eps} \bigl(| \delta Z_{s,t}| +
\bigl|R^Z_{s,t}\bigr| \bigr) \le2 \|Z\|_\infty+
\bigl\|R^Z\bigr\|_{2\g} \eps^{2\g} . \label{eqnBReta}
\end{equation}
Let now $Z'_s(j)$ denote the $j$th row of the matrix $Z'_s$.
Since $X$ is $\theta$-H\"older rough by assumption, for every $j \leq
d$, there exists
$v = v(j)$ with $|v - s| \le\eps$ such that
%
%
%e3.3 #&#
\begin{equation}
\label{eqnLgammarh} \bigl|\bigl\langle Z'_s(j) ,\delta
X_{s,v}\bigr\rangle\bigr| > L_{\theta}(X) \eps^{\theta}\bigl
|Z'_s(j)\bigr|
.
\end{equation}
Combining both (\ref{eqnBReta}) and (\ref{eqnLgammarh}), we thus
obtain that
\[
L_{\theta}(X) \eps^{\theta} \bigl|Z'_s(j)\bigr| \le
2 \|Z\|_\infty+ \bigl\|R^Z\bigr\|_{2\g} \eps^{2\g} .
\]
Summing over the rows of $Z_s'$ yields a universal constant $C$ such that
\[
L_{\theta}(X) \eps^{\theta} \bigl|Z'_s \bigr| \le C
\bigl( \|Z\|_\infty+ \bigl\|R^Z\bigr\|_{2\g}
\eps^{2\g} \bigr) .
\]
Optimizing over $\eps$, we choose $\eps= \|Z\|_\infty^{1/2\g} \|
R^Z\|_{2\g}^{-1/ 2\g} \wedge(T/2)$, thus deducing that
\[
\bigl|Z'_s\bigr| \le\frac{M \|Z\|_\infty}{ L_{\theta}(X)} \bigl(\bigl\|R^Z
\bigr\|_{2\g
}^{{\theta} /{ (2\g)}} \|Z\|_\infty^{- {\theta}/{ (2\g)}}
\vee T^{-\theta}\bigr).
 \]
Since $s$ was arbitrary, the stated bound follows at once.
\end{pf}

One way of reading Proposition~\ref{propNlemma1} is to say that if $\|
Z \|_{\infty}$ is small,
then $\|Z'\|_{\infty}$ must also be small, provided that $(Z,Z') \in
\mathcal{C}_{X}^{\g}([0,T],\R^m)$ and that $X$
is H\"older rough.
In the following theorem, we apply Proposition~\ref{propNlemma1} to
obtain a quantitative
version of a ``Doob--Meyer type decomposition'' for SDEs driven by a
rough path~$X$.
This is the main new technical result of this article.

%
%th3.1 #&#
\begin{theorem}\label{thmNlemmaRP}
Let $(X, \bbX) \in\CD^\g([0,T],\R^n)$ with $\gamma> \frac{1}{
3}$ be such that $X$ is $\theta$-H\"older rough
with $2\gamma> \theta$. Let $(A,A') \in\mathcal{C}_{X}^{\g
}([0,T],\R^{mn})$ and $B \in\CC^\gamma([0,T],\break\R^m)$, and set
%
%
%e3.4 #&#
\begin{equation}
\label{edefZ} Z_t = \int_0^t
A_s \,dX_s + \int_0^t
B_s \,ds .
\end{equation}
Then, there exist constants $r>0$ and $q>0$ such that, setting
\[
\CR\eqdefi 1 + {L_{\theta}(X)}^{-1} + \bigl\|(X,\bbX)\bigr\|_{\g}
+ \bigl\|\bigl(A,A'\bigr)\bigr\|_{X,\g} + \|B\|_{\mathcal{C}^\g} ,
\]
one has the bound
\[
\|A\|_{\infty} + \|B\|_{\infty} \le M \CR^q \|Z
\|_{\infty}^r
\]
for a constant $M$ depending only on $T$, $m$ and $n$.
\end{theorem}

%
%re5 #&#
\begin{remark}
The proof provides the explicit value $q = 6$ and shows that $r$ can be
taken arbitrarily close to $(2\gamma- \theta)^2(3\gamma-1)/(4\gamma
^2(1+\gamma))$, but these values are certainly not optimal.
\end{remark}

\begin{pf}
Note first that the definition of $Z$ does not change if we add a
constant to $X$. We will therefore
assume without loss of generality that $X_0 = 0$, so that $\|X\|_\infty
\le T^\gamma\|X\|_\gamma\le T^\gamma\CR$.
By Theorem~\ref{thmGubi} we deduce that the pair $(Z,A)$ is a weakly
controlled
rough path, $(Z,A) \in\mathcal{C}_{X}^{\g}([0,T],\R^m)$, with
%
%
%e3.5 #&#
\begin{equation}
\label{eqnnrslemti1} \delta Z = A \delta X + R^Z
\end{equation}
and
\[
\bigl\|R^Z\bigr\|_{2\g} \le M \bigl(\|X\|_\g
\bigl\|R^A\bigr\|_{2\g} + \|\bbX\|_{2\g}\| A
\|_{\mathcal{C}^\g} + \|B\|_{\mathcal{C}^\g} \bigr) \le M\CR^2 .
\]
We deduce from the above that in particular, we have the a priori bound
$\|Z\|_\infty\le M \CR^2$.

It then follows from Proposition~\ref{propNlemma1} that
%
%
%e3.6 #&#
\begin{eqnarray}
\label{eqnAinftynlbd} \|A\|_\infty&\le& M{ L_{\theta}(X)}^{-1}
\|Z\|_\infty^{1-{\theta}/{ (2\g)}} \bigl(\bigl\|R^Z\bigr\|_{2\g
}^{{\theta} /{ (2\g)}} + \|Z\|_{\infty}^{{\theta}/ { (2\g)}} \bigr
)
\nonumber
\\[-8pt]
\\[-8pt]
\nonumber
 &\le& M
\CR^3 \|Z\|_\infty^{1-{\theta}/{ (2\g)}} .
\end{eqnarray}
This is already the requested bound on $A$. The bound on $B$ is
slightly more difficult to obtain.

At this stage, we would like to make use of the information that $\|A\|
_\infty$ is ``small''
to get a bound on the integral of $A$ against $X$. In order to do
so, it turns out to be convenient to choose a $\b\in(\frac{1}{ 3},\g
)$ with $2\b> \theta$, so that
we can interpret $(X, \bbX)$ as an element of $\CD^\b([0,T],\R^n)$
with $\|(X,\bbX)\|_\b\leq M\|(X,\bbX)\|_\g$.
This will allow us to make use of interpolation inequalities to combine
our a priori knowledge about the boundedness of
$(A,A')$ in $\CC^\gamma_X$ norm with (\ref{eqnAinftynlbd}) to
conclude that $(A,A')$ is small in $\CC^\beta_X$.

We first obtain a bound on $A'$.
Since $(A,A') \in\mathcal{C}_{X}^{\g}([0,T], \R^{mn})$, we infer
from~(\ref{eqnAinftynlbd}) and Proposition~\ref{propNlemma1} that
\[
\bigl\|A'\bigr\|_{\infty} \le M{ L_{\theta}(X)}^{-1}
\bigl\|R^A\bigr\|_{2\b}^ {{\theta} /{ (2\g)}} \|A
\|_\infty^{1-{\theta}/{ (2\g)}} \le M \CR^4 \|Z
\|_\infty^{(1-{\theta}/{( 2\g)})^2} .
\]
Using the inequality
%
%
%e3.7 #&#
\begin{equation}
\label{eqnHoldgbipbd} \bigl\|A'\bigr\|_{\b} \le2
\bigl\|A'\bigr\|^{{\b}/{ \g}}_{\g} \bigl\|A'
\bigr\|^{1 -
{\b}/{ \g}}_{\infty} ,
\end{equation}
which follows immediately from the definition of the H\"older norm,
we obtain the bound
\[
\bigl\|A'\bigr\|_\b\le M\bigl\|A'\bigr\|^{1- {\b}/{ \g}}_\infty
\bigl\|A'\bigr\|^{{\b} /{ \g}}_\g\le M
\CR^4 \|Z\|_\infty^{(1-{\theta}/{ (2\g)})^2(1 - {\b}/{{ \g}})} ,
\]
where we used the fact that $\b< \g$. Similarly, we would like to
obtain a bound on $\|R^A\|_{2\beta}$.
Combining the definition of $R^A$ with (\ref{eqnAinftynlbd}), we
deduce that
\[
\bigl\|R^A\bigr\|_{\infty} \le2\bigl(\|A\|_{\infty} +
\bigl\|A'\bigr\|_{\infty}\|X\|_\infty\bigr) \le M
\CR^5 \|Z\|_\infty^{(1-{\theta}/{ {(2\g)}})^2} .
\]
Using the obvious equivalent to (\ref{eqnHoldgbipbd}), we conclude that
\[
\bigl\|R^A\bigr\|_{2\b} \le M \bigl\|R^A
\bigr\|_{2\g}^{{\b} /{\g}} \bigl\|R^A\bigr\|_{\infty
}^{1-{\b}/{\g}}
\le M \CR^5 \|Z\|_\infty^{(1-{\theta}/{{(2\g)}})^2(1-{\b}/{\g})} .
\]
We are now in a position to use Theorem~\ref{thmGubi} to bound the
integral $\int_0^\cdot A_s \,dX_s$. Indeed, we obtain
from (\ref{eqnrpgubibd}) the bound
\[
\biggl\|\int_0^\cdot A_s
\,dX_s \biggr\|_{\infty} \leq M\bigl(|A_0| \|X
\|_{\infty} + \| X \|_{\b} \bigl\|R^A\bigr\|_{2\b} +
\|\bbX\|_{2\b} \bigl\|A'\bigr\|_\b\bigr) .
\]
Inserting into this bound all of the above estimates, we conclude that
%
%
%e3.8 #&#
\begin{equation}
\label{eqnAintegbd} \biggl\|\int_0^\cdot
A_s \,dX_s \biggr\|_{\infty} \le M \CR^6
\|Z\|_\infty^{(1-{\theta}/{( 2\g)})^2(1 - {\b}/{ \g})} .
\end{equation}
This estimate, together with the definition (\ref{edefZ}) of $Z$
immediately implies that we also have the bound
\[
\biggl\|\int_0^\cdot B_s \,ds
\biggr\|_{\infty} \le M \CR^6 \|Z\|_\infty^{(1-{\theta}/{ (2\g
)})^2(1 - {\b}/{ \g})} .
\]
Once again we use an interpolation inequality to strengthen this bound.
Applying the interpolation
inequality
\[
\|\partial_t f\|_{\infty} \le M \|f\|_\infty\max
\biggl(\frac{1
}{ T}, \|f\|_{\infty}^{-{1}/{ (\g+1)}} \|
\partial_t f\|_{\g
}^{{1 } /{ \g+1}} \biggr)
\]
(see~\cite{HairMatt09}, Lemma 6.14) with $f(t) = \int_0^t B_s \,ds$,
it follows that
%
%
%e3.9 #&#
\begin{equation}
\label{eqnBintegbd} \|B\|_\infty\le M \CR^6 \|Z
\|_\infty^{(1-{\theta}/{ {2\g}})^2(1
- {\b}/{ \g}){\g}/{ (1+ \g)}} .
\end{equation}
The claim now follows from (\ref{eqnAinftynlbd}) and
(\ref{eqnBintegbd}), and the remark following the statement follows
by choosing $\beta\approx\frac{1}{ 3}$.
\end{pf}

%s3.1 #&#
\subsection{H\"older roughness of sample paths of fBm}
Our aim now is to show that the sample paths of a fractional Brownian
motion $X$ (with $H \leq1/2$ throughout this section) are indeed
almost surely H\"older rough and
to provide quantitative bounds on the tail behavior of $L_\theta(X)$
for a suitable $\theta$.
Let $\{\mathcal{F}^X_s, s \in\bbR\}$ be the natural filtration
generated by the fBm $X$, namely $\CF^X_s$ is
the $\sigma$-algebra generated by $\{X_r\}_{r \in(-\infty,s]}$. We
start with the following lemma on the small ball
probability of the conditioned fBm:

%
%le2 #&#
\begin{lemma} \label{lemslep}
Let $\phi\in\R^n$ with $\|\phi\|= 1$ and $\delta\le1$. Then there
exist constants $M$ and $c$ such that the bound
%
%
%e3.10 #&#
\begin{equation}
\label{eqnfbmslebd} \mathbb{P} \Bigl(\inf_{\|\phi\| = 1}\sup_{s,t
\in[0,\delta]} \bigl|
\langle\phi, \delta X_{st} \rangle\bigr| \leq\eps\big| \mathcal
{F}^X_0 \Bigr) \leq Me^{- c \delta^{2H} \eps^{-2}}
\end{equation}
holds almost surely, for every $0 < \eps\le1$ and $H \leq1/2$.
\end{lemma}

\begin{pf}
By the scaling properties of (conditioned) fractional Brownian motion,
and since the bound is trivial for $\eps> \delta^H$, we can restrict
ourselves to the case $\delta= 1$ and $\eps\le1$.
For the moment, let us fix an arbitrary $\phi$ with $\|\phi\| = 1$.\vadjust{\goodbreak}

Since $X_0 = 0$, we obtain
\[
\mathbb{P} \Bigl(\sup_{s,t \in[0,1]} \bigl|\langle\phi, \delta{X}_{st}
\rangle\bigr| \leq\eps\big| \mathcal{F}^X_0 \Bigr) \leq
\mathbb{P} \Bigl(\sup_{t \in[0,1]} \bigl|\langle\phi,
{X}_{t}\rangle\bigr|
\leq\eps\big| \mathcal{F}^X_0 \Bigr) .
\]
At this stage, we note that there exists a one-dimensional Wiener
process $W$ (depending on $\phi$) independent of $\mathcal{F}^X_0$,
a stochastic process $Y^\phi= \langle\phi, Y\rangle$ such that
$Y_t$ is
$\mathcal{F}^X_0$-measurable for every $t\ge0$,
and a constant $c$ such that
%
%
%e3.11 #&#
\begin{equation}
\label{eexprW} \langle\phi, {X}_{t}\rangle= Y_{t}^\phi+
c \int_0^t (t-s)^{H-{1}/{2}} \,dW(s) \eqdef
Y_t^\phi+ \hat X_t .
\end{equation}
(See, e.g.,~\cite{MandVann68,Hair05}, as well as (\ref{erepr})
below.) Furthermore, $Y$ has almost surely bounded sample
paths. Any such sample path then induces a seminorm $\|\cdot\|_Y$ on
$\R^n$ by
\[
\|\phi\|_Y \eqdef\sup_{t \in[0,1]} |Y_t^\phi|
.
\]
Furthermore, this is almost surely nondegenerate, so that $\|\cdot\|
_Y$ is actually a norm.
For the rest of the proof, we use $c$ for a generic universal constant
that will change from expression to expression.

%If we denote by $\Sigma_Y$ the $\sigma$-algebra generated by the
%process $Y$,
It then follows from~\cite{shifted}, Theorem~2 (set $\alpha= 1$), which
is a refinement of Anderson's inequality~\cite{Ande55}, that we have
the bound
\begin{eqnarray*}
\mathbb{P} \Bigl(\sup_{t \in[0,1]} \bigl|\langle\phi,
{X}_{t}\rangle\bigr|
\leq\eps\big| \mathcal{F}^X_0 \Bigr) &=& \P\Bigl(
\sup_{t \in[0,1]} \bigl|\hat X_t + Y_t^\phi\bigr| \leq
\eps\big| \mathcal{F}^X_0 \Bigr)
\\
&\le&\exp\biggl(- \inf_{\|Y^\phi- h\|_\infty\le\eps} \frac{\|
h\|_\CH^2 }{ 2} \biggr) \P\Bigl(
\sup_{t \in[0,1]} |\hat X_t| \leq\eps\Bigr) ,
\end{eqnarray*}
where $\|h\|_\CH$ denotes the norm of $h$ in the Cameron--Martin space
of $\hat X$. Since this norm is always stronger than the
supremum norm, there exists a constant $c$ such that $\|h\|_{\CH}\ge c
\|h\|_\infty\ge c (\|\phi\|_Y - \eps)$, so that
%
%
%e3.12 #&#
\begin{equation}\quad
\label{einterBd} \mathbb{P} \Bigl(\sup_{t \in[0,1]} \bigl|\langle
\phi,
{X}_{t}\rangle\bigr| \leq\eps\big| \mathcal{F}^X_0
\Bigr) \le C\exp\bigl(- c\|\phi\|_Y^2 \bigr) \P\Bigl(
\sup_{t \in[0,1]} |\hat X_t| \leq\eps\Bigr)
\end{equation}
for some positive constants $c$ and $C$ independent of $\eps\le1$.

On the other hand, we can invert the expression (\ref{eexprW}) for
$\hat X$, yielding
\[
W_{t} = c\int_0^t
(t-s)^{{1}/{ 2} - H} \,d\hat X_s = c\int_0^t
(t-s)^{-{1}/{ 2} - H} \hat X_s \,ds .
\]
In particular, provided that $H< \frac{1}{ 2}$, we have the bound
\[
\sup_{t \in[0,1]} | W_{t}| \le c \sup_{t \in[0,1]} |\hat
X_t| ,
\]
so that
\[
\P\Bigl(\sup_{t \in[0,1]} |\hat X_t| \leq\eps\Bigr) \le
\mathbb{P} \Bigl(\sup_{t \in[0,\delta]} |{W}_{t}| \leq c
\delta^{{1}/{ 2}-H}\eps\Bigr) \le M e^{- c\eps^{-2}} ,\vadjust{\goodbreak}
\]
where the last inequality is the well-known small ball probability for
the standard
Brownian motion~\cite{LiShao01}.
Combining this with (\ref{einterBd}), we conclude that
%
%
%e3.13 #&#
\begin{equation}
\label{eboundY} \mathbb{P} \Bigl(\sup_{t \in[0,\delta]} \bigl
|\langle\phi,
{X}_{t}\rangle\bigr| \leq\eps\big| \mathcal{F}^X_0
\Bigr) \le M \exp\biggl(-\frac{c }{ \eps^{2}} - c \|\phi\|_Y^2
\biggr) .
\end{equation}

Up to now, the calculation was performed with a fixed instance of $\phi$.
In order to conclude, we use a covering argument similar to \cite
{Norr86}, page~127. The main problem difference is that
such an argument requires a priori bounds on the process $X$, and these
are of course not uniform in $Y$.
It turns out that, thanks to the exponentially decaying factor in (\ref
{eboundY}), it is still possible to obtain
a uniform bound, but one needs to be a little bit more careful.
Our main tool is the fact that, as a consequence of John's theorem
\cite{John,Keith}, it is possible to perform an orthogonal
change of coordinates for $\phi$ (that depends on $Y$) and to find
constants $Y^{(1)},\ldots,Y^{(n)}$ such that
%
%
%e3.14 #&#
\begin{equation}
\label{eboundNormY} \sup_{j} Y^{(j)} |\phi_j|
\le\|\phi\|_Y \le C\sup_{j} Y^{(j)} |
\phi_j| ,
\end{equation}
were the constant $C$ is universal and depends only on $n$. By
equivalence of norms in $\R^n$ we can furthermore,
for any given realization of $Y$, replace the Euclidean norm $\|\phi\|
$ in (\ref{eqnfbmslebd}) by the $\ell^\infty$ norm $\|\phi\|_\infty$.
If we can find a finite collection $\Phi\subset\R^n$
such that, for every $\phi$ with $\|\phi\|_\infty= 1$,
one has the bound
%
%
%e3.15 #&#
\begin{equation}
\label{econdPhi1} \sup_{\|\phi\|_\infty= 1}\inf_{\tilde\phi\in
\Phi} \|\phi-\tilde\phi
\|_Y \leq\frac{\eps}{ 4} ,
\end{equation}
then we obtain the inequality
%
%
%e3.16 #&#
\begin{eqnarray}\label{esecondTerm}
&&\mathbb{P} \Bigl(\inf_{\|\phi\|_\infty= 1}\sup_{s,t \in[0,1]}
\bigl|\langle\phi, \delta
X_{st} \rangle\bigr| \leq\eps\big| Y \Bigr)\nonumber\\
&&\qquad \le M\sum
_{\tilde\phi\in\Phi} \exp\biggl(-\frac{c }{ \eps^{2}} - c \|
\tilde\phi
\|_Y^2 \biggr)
\\
&&\qquad\quad{} + \mathbb{P} \biggl(\sup_{\|\phi\|_\infty= 1}\inf_{\tilde
\phi\in
\Phi}
\sup_{s,t \in[0,\delta]} \bigl|\langle\phi-\Phi, \delta\hat
X_{st} \rangle\bigr| \leq
\frac{\eps}{ 4} \biggr) .\nonumber
\end{eqnarray}
If we can furthermore choose the collection $\Phi$ so that
%
%
%e3.17 #&#
\begin{equation}
\label{econdPhi2} \sup_{\|\phi\|_\infty= 1}\inf_{\tilde\phi\in
\Phi} \|\phi-\tilde\phi
\|_\infty\leq\eps^2 ,
\end{equation}
then, due to the Gaussian tails of $\hat X$, the second term in (\ref
{esecondTerm}) is bounded
by $M \exp(-c/\eps^2)$ as desired. It thus remains to show that, for
every norm $\|\cdot\|_Y$, it is possible to
choose $\Phi$ satisfying (\ref{econdPhi1}) and (\ref{econdPhi2}) and
such that
%
%
%e3.18 #&#
\begin{equation}
\label{econdPhifinal} \sum_{\phi\in\Phi} \exp\bigl(-c \|
\phi\|_Y^2 \bigr)\le\frac
{C}{ \eps^\kappa}
\end{equation}
for some constants $C>0$ and $\kappa> 0$, uniformly over $\|\cdot\|
_Y$. We choose $\Phi\subset\{\phi\dvtx \|\phi\|_\infty= 1\}$ in
the following way. Let
\[
A_j = \eps^2 \wedge\frac{\eps}{ 4Y^{(j)}} ,
\]
and, for every $k \in\Z^d$, denote by $Ak$ the element in $\R^n$
given by $\sum_j A_j k_j e_j$, where $e_j$ is the $j$th unit vector.
We also write $\Z^d_j$ for the elements in $k \in\Z^d$ with $k_j = 0$.
We then set $\Phi= \bigcup_{j =1}^n (\Phi_j^+ \cup\Phi_j^-)$, where
\[
\Phi_j^\pm= \bigl\{\pm e_j + Ak \dvtx k \in
\Z^d_j \bigr\} \cap\bigl\{ \phi\dvtx \|\phi\|_\infty= 1\bigr\}
.
\]
It is clear that this choice of $\Phi$ satisfies both (\ref
{econdPhi1}) and (\ref{econdPhi2}), so that it remains to show
that (\ref{econdPhifinal}) is satisfied, uniformly over the choices of
$\{Y^{(j)}\}$.
For $k \in\Z^d_j$, denote $\phi_{j;k}^\pm= \pm e_j + Ak$. With this
notation at hand, it follows from (\ref{eboundNormY}) that there
exists a constant $c$ such that
\[
\|\phi_{j;k}\|_Y^2 \ge c\eps^2
\sum_{i \neq j} |k_i|^2 \bigl(1
\wedge\eps^2 \bigl|Y^{(i)}\bigr|^2 \bigr) .
\]
It follows that
\begin{eqnarray*}
\sum_{\phi\in\Phi_j^\pm} \exp\bigl(-c \|\phi
\|_Y^2 \bigr)
&\le& \prod_{i \neq j}
\sum_{|k| \le A_i^{-1}} \exp\bigl(-c\eps^2|k|^2
\bigl(1 \wedge\eps^2 \bigl|Y^{(i)}\bigr|^2 \bigr) \bigr)
\\
&\le&\prod_{i \neq j} \biggl(2\eps^{-2} +
\one_{|4Y^{(i)}| > \eps
^{-1}} \sum_{|k| \le4Y^{(j)}/\eps} \exp
\bigl(-c|k|^2 \eps^4 \bigl|Y^{(i)}\bigr|^2
\bigr) \biggr)
\\
&\le& c\prod_{i \neq j} \biggl(\eps^{-2} +
\one_{|4Y^{(i)}| > \eps
^{-1}} \int_{\R} e^{-c|x|^2 \eps^4 |Y^{(i)}|^2} \,dx \biggr)
\\
&\le& c\prod_{i \neq j} \biggl(\eps^{-2} +
\one_{|4Y^{(i)}| > \eps
^{-1}} \frac{1}{ \eps^2 |Y^{(i)}|} \biggr) \le c \eps^{-2n} ,
\end{eqnarray*}
where all the constants $c$ are independent of the choice of
coefficients $Y^{(i)}$, so that the bound (\ref{econdPhifinal}) does
indeed hold, which concludes the proof.

Note that, for $H = 1/2$, the exact same argument goes through, but it
is simplified due to the Markov property, which implies
that $Y = 0$.
\end{pf}

We have the following corollary of Lemma~\ref{lemslep}:
%
%
%co1 #&#
\begin{corollary} \label{corslep}
For any interval $I_\delta\eqdefi[u_\ell, u_\ell+ \delta] \subset
\bbR$ of length $\delta$ and any $u \leq u_\ell$,
there exist constants $M$ and $c$ such that the bound
\[
\mathbb{P} \Bigl(\inf_{\|\phi\| = 1}\sup_{s,t \in I_\delta} \bigl
|\langle\phi, \delta
X_{st} \rangle\bigr| \leq\eps\big| \mathcal{F}^X_u
\Bigr) \leq Me^{- c \delta^{2H} \eps^{-2}}
\]
holds for every $0 < \eps\le1$ and $H \leq1/2$.\vadjust{\goodbreak}
\end{corollary}

\begin{pf}
Define the event $G \eqdef\{\inf_{\|\phi\| = 1}\sup_{s,t \in
I_\delta} |\langle\phi, \delta X_{st} \rangle| \leq\eps\}$.
Since the increments of
the fBm are stationary, by Lemma~\ref{lemslep} we obtain the bound
%
%
%e3.19 #&#
\begin{eqnarray}
\label{eqnfbmslebd1} \mathbb{E} \bigl( 1_{G} | \mathcal
{F}^X_{u_\ell}
\bigr) \leq Me^{- c
\delta^{2H} \eps^{-2}} .
\end{eqnarray}
Now notice that for any $G \in\mathcal{F}^X_{u_\ell}$ and $u \leq
u_\ell$,
$\mathbb{E}(G|\mathcal{F}^X_{u}) = \mathbb{E} (\mathbb
{E}(G|\mathcal{F}^X_{u_\ell})|\mathcal{F}^X_{u} )$. Since the
right-hand side of equation (\ref{eqnfbmslebd1}) does not depend on
${u_\ell}$, it immediately follows that
%
%
%e3.20 #&#
\begin{eqnarray}
\label{eqnfbmslebd2}\quad  \mathbb{P} \Bigl(\inf_{\|\phi\| = 1}\sup
_{s,t \in I_\delta} \bigl|
\langle\phi, \delta X_{st} \rangle\bigr| \leq\eps\big| \mathcal
{F}^X_{u} \Bigr) = \mathbb{E} \bigl( 1_{G} |
\mathcal{F}^X_{u} \bigr) \leq Me^{- c \delta^{2H} \eps^{-2}}
\end{eqnarray}
and the corollary follows.
\end{pf}
%
%
%re6 #&#
\begin{remark}\label{remHtHrem}
Although the above proof requires that $H \leq\frac{1}{ 2}$, one
would expect that a result
similar to that of Lemma~\ref{lemslep}
holds for any Gaussian process $X$ such that
\[
M^2_{\ell}|t-s|^{2H} \leq\E|X_t -
X_s|^2 \leq M^2_{u}|t-s|^{2H},\qquad
s,t \in[0,T]
\]
for some constants $M_{\ell}, M_u$, even if $H>\frac{1}{ 2}$. For
instance, a result similar to
Lemma~\ref{lemslep}
can be shown to hold in the case of
fBm with index $H > \frac{1}{ 2}$; see, for example, \cite
{BaudHair07}, Proposition~3.4.
\end{remark}

Using this estimate we are now in the position to obtain bounds on the
modulus of H\"older roughness
for fractional Brownian motion with $H \leq\frac{1}{ 2}$ by a type of
chaining argument.

%
%le3 #&#
\begin{lemma}\label{lemboundRoughfBm}
Let $X$ be a fBm with Hurst parameter $H \leq\frac{1}{ 2}$. Then, for
every $ \theta> H$, the
sample paths of $X$ are almost surely $\theta$-H\"older rough. Moreover,
there exist constants $M$ and $c$ independent of $X$ such
that
\[
\bbP\bigl(L_{\theta}(X) \le\eps| \mathcal{F}^X_0
\bigr) \leq M \exp\bigl(-c \eps^{-2} \bigr)
\]
for all $\eps\in(0,1)$. In particular, $\E(L_\theta^{-p}(X) |
\mathcal{F}^X_0) < \infty$ for every $p>0$.
\end{lemma}

\begin{pf}
A different way of formulating Definition~\ref{defnsthc} is given by
\[
L_{\theta}({X}) = \inf_{\|\phi\|=1}\inf_{t \in[0,T]}
\inf_{r \in
[0,T/2]} \sup_{ |t-s| \le r} \frac{|\langle\phi,\delta X_{st}
\rangle| }{ r^{\theta}} .
\]
We then define the ``discrete analog'' $D_{\theta}({X})$ of $L_{\theta
}({X})$ to be
\[
D_{\theta}({X}) \eqdef\inf_{\|\phi\|=1}\inf_{n \geq1}
\inf_{k
\leq2^{n}} \sup_{s,t \in I_{k,n}} \frac{|\langle\phi,\delta
X_{st}\rangle| }{ (2^{-n}T)^\theta} ,
\]
where $I_{k,n} = [\frac{k-1 }{ 2^n}T, \frac{k }{ 2^n}T]$.
We first claim that
%
%
%e3.21 #&#
\begin{eqnarray}
\label{eqnLtH} L_{\theta}({X}) \ge\frac{1}{ 2\cdot8^{\theta}}
D_{\theta}({X})
.
\end{eqnarray}
Indeed, given $t\in[0,T]$ and $r\in[0,T/2]$, pick $n \in\mathbb{N}$
such that
$r/8 \le2^{-n}T < r/4$. It follows that there exists some $k$ such
that $I_{k,n}$ is included in
the set $\{s \dvtx r/2 \le|t-s| \le r\}$. Then, by definition of
$D_\theta$, for any unit vector $\phi$
there exist two points $t_1, t_2 \in I_n$
such that
\[
\bigl|\langle\phi, \delta X_{t_2 t_1} \rangle\bigr| \ge
2^{-n\theta}
D_{\theta}(X) .
\]
Therefore by the triangle inequality, we conclude that the magnitude of
the difference between $\langle\phi, X_t \rangle$ and one of the two
terms $\langle\phi, X_{t_i} \rangle, i = 1,2$ (say $t_1$) is at least
\[
\bigl|\langle\phi, \delta X_{t_1 t} \rangle\bigr| \ge\tfrac{1 }{
2}\cdot
\bigl(2^{-n}T\bigr)^\theta D_{\theta}(X)
\]
and therefore
\[
\frac{|\langle\phi, \delta X_{t_1t} \rangle| }{ r^{\theta}} \ge
\frac{1 }{ 2}\cdot\frac{2^{-n\theta} }{ r^{\theta}}
D_{\theta}(X) \ge\frac{1}{ 2\cdot8^{\theta}} D_{\theta}({X}) .
\]
Since $t$, $r$ and $\phi$ were chosen arbitrarily, claim (\ref
{eqnLtH}) follows.

It follows that it is sufficient to obtain the requested bound on $\bbP
(D_{\theta}(X) \le\eps| \mathcal{F}^X_0)$.
We have the straightforward bound
\begin{eqnarray*}
\bbP\bigl(D_{\theta}(X) \le\eps| \mathcal{F}^X_0
\bigr) &\leq& \bbP\biggl( \inf_{\|\phi\|=1}\inf_{n \geq1}
\inf_{k \leq2^{n}} \sup_{s,t \in I_{k,n}} \frac{|\langle\phi,
\delta X_{st} \rangle|
}{ 2^{-n\theta}} \le\eps\Big|
\mathcal{F}^X_0 \biggr)
\\
&\leq&\sum_{n =1}^\infty\sum
_{k=1}^{2^n}\bbP\biggl( \inf_{\|\phi\|
=1}
\sup_{s,t \in I_{k,n}} \frac{|\langle\phi, \delta X_{st}
\rangle| }{ 2^{-n\theta}} \le\eps\Big| \mathcal{F}^X_0
\biggr) .
\end{eqnarray*}
Applying Lemma~\ref{lemslep} and noting that the bound obtained in
this way is independent of $k$, we conclude that
\[
\bbP\bigl(D_{\theta}(X) \le\eps| \mathcal{F}^X_0
\bigr) \leq M \sum_{n =1}^\infty2^n
\exp\bigl(-c {2^{2n(\theta- H)} \eps^{-2}} \bigr)\leq\tilde M
\sum
_{n =1}^\infty\exp\bigl(- \tilde c n
\eps^{-2} \bigr) .
\]
Here, we used the fact that we can find constants $K$ and $\tilde c$
such that
\[
n \log2 - c {2^{2n(\theta- H)} \eps^{-2}} \le K - \tilde c n
\eps^{-2} ,
\]
uniformly over all $\eps\le1$ and all $n \ge1$. We deduce from this
the bound
\[
\bbP\bigl(D_{\theta}(X) \le\eps| \mathcal{F}^X_0
\bigr) \leq M \biggl(e^{-\tilde c\eps^{-2}} + \int_1^\infty
\exp\bigl(-\tilde c \eps^{-2}x \bigr) \,dx \biggr) ,
\]
which immediately implies the result.
\end{pf}

%s4 #&#
\section{Malliavin derivatives}\label{secLpflow}

In this section, we derive formulas for the Malliavin derivatives of
solutions to (\ref{eqnSDErpi}), when conditioned
on the past of the driving noise. In order to clarify the meaning of
this statement, we will reduce this
conditioned solution to a functional of an underlying Wiener process.
With this notation, the Malliavin
derivative will simply be the ``usual'' Malliavin derivative of a
random variable on Wiener space.

Before proceeding further, let us make a digression that clarifies this
construction.
For $\alpha\in(0,1)$, we define the fractional
integration operator $\CI^{\alpha}$ and the
corresponding fractional differentiation operator
$\CD^\alpha$ by
%
%
%e4.1 #&#
\begin{eqnarray}
\label{efrac} \CI^{\alpha}f(t) &\equiv&\frac{1 }{ \Gamma(\alpha
)}\int
_0^t (t-s)^{\alpha-1}f(s) \,ds ,
\nonumber
\\[-8pt]
\\[-8pt]
\nonumber
\CD^{\alpha}f(t) &\equiv&\frac{1 }{ \Gamma(1-\alpha)}\frac{d }{
dt} \int
_0^t (t-s)^{-\alpha}f(s) \,ds ,
\end{eqnarray}
with the convention that $\CI^0$ and $\CD^0$ denote the identity operator.
The operators $\CI^{\alpha}$ and $\CD^{\alpha}$ are inverses of
each other;
see, for example,~\cite{SamKilMar93} for a survey of fractional
integral operators.

It turns out that the operator $\CI^{{1}/{ 2} - H}$ is an isometry
between the Cameron--Martin space of the conditioned fBm and
that of the underlying Wiener process mentioned at the beginning
of this section.
More precisely, given a typical instance $w_{-} \in\mathcal{C}(\R_-,
\R^d)$ of
the ``past'' of the fBm, it follows from the Mandelbrot--van Nesse
representation of the fractional Brownian motion
\cite{MandVann68,Hair05} that there exists a constant $\alpha_H$
and a
(one-sided) Wiener process $W$ on $\R_+$ independent of $w_{-}$ such that
the future $w_+ \in\CC(\R_+, \R^d)$
of the fBm conditioned on the past
$w_{-}$ may be expressed as
%
%
%e4.2 #&#
\begin{equation}
\label{erepr} w_+ = \CG w_{-} + \alpha_H
\CD^{{1}/{ 2}-H} W ,
\end{equation}
where $\CD^{{1}/{ 2}-H}$ is as defined in (\ref{efrac}),
and the operator $\CG\colon\mathcal{C}(\mathbb{R}_{-}, \R^d)
\mapsto\CC(\R_+,\R^d)$ is given by
%
%
%e4.3 #&#
\begin{equation}
\label{eqnopera} (\CG w_-) (t) \eqdef\gamma_H\int
_0^{\infty} \frac{1}{ r }g \biggl(
\frac{t}{
r} \biggr) w_-(-r) \,dr .
\end{equation}
Here, the kernel $g$ is given by
%
%
%e4.4 #&#
\begin{eqnarray}
\label{eqngfun} g(v) \eqdef v^{H-{1}/{ 2}} + (H- 3/2) v \int
_0^1 \frac
{(u+v)^{H-{5}/{2}}}{
(1-u)^{H-{1}/{ 2}}} \,du ,
\end{eqnarray}
and the constant $\gamma_H$ is given by $\gamma_H = (H-\frac{1}{
2})\alpha_H
\alpha_{1-H}$, where
$\alpha_H$ is
the constant appearing in (\ref{erepr}). The interpretation
of the operator $\CG$ is that $ (\CG w_- )(t)$ is the
conditional expectation at time $t$
of a two-sided fractional Brownian motion with Hurst parameter $H$,
conditioned on coinciding with $w_-$ for
negative times.

Henceforth we will use the notation (\ref{erepr}); namely we
denote the past of the fBm by $w_{-} \in\C(\R_{-},\R^d)$
and the future by $w_+ \in\C(\R,\R^d)$. At this stage, we will use
a slight abuse of notation,
and we will also sometimes interpret $w_+$ as an element in the space
$\CC^\gamma_g(\R_+,\R^d)$
of geometric rough paths that are $\gamma$-H\"older continuous,
although we then usually denote it by $(X,\bbX)$.

In view of (\ref{erepr}), it will be useful to clarify how to
interpret this identity when the future is considered as
an element in the space $\CC^\gamma_g(\R_+,\R^d)$, and for which
instances of $w_-$ the decomposition
(\ref{erepr}) makes sense.
Recall that, for $(X,\XX)\in\CC^\gamma_g([0,T],\R^d)$ and $h \in
\CC([0,T],\R^d)$ a path with bounded variation,
we can define a translated path $(Y,\YY) = \tau_{h}(X,\XX)$ in a
natural way by
%
%
%e4.5 #&#
\begin{eqnarray}
\label{eqnmaptauh} Y_t &=& X_t + h_t ,
\nonumber
\\[-8pt]
\\[-8pt]
\nonumber
\YY_{s,t} &=& \XX_{s,t} + \int_s^t
\delta X_{s,r} \otimes dh_r + \int_s^t
\delta h_{s,r} \otimes dX_r + \int_s^t
\delta h_{s,r} \otimes dh_r .
\end{eqnarray}
Since we assumed $h$ to be of bounded variation, the integrals
appearing in this expression
should be interpreted as usual Riemann--Stieltjes integrals. Assume
furthermore that $h$ is such that there exists a constant
$\|h\|_{1;\gamma}$ such that, for every $s \le t$ in $[0,T]$, the
variation of $h$ over the interval $[s,t]$ is bounded by
$\|h\|_{1;\gamma}|t-s|^\gamma$. In this case, it follows immediately
that there exists $M$ (depending on $T$) such that
%
%
%e4.6 #&#
\begin{equation}
\label{eboundTR1}\qquad \|Y - X\|_\gamma\le\|h\|_{1;\gamma} ,\qquad\|\YY
- \XX
\|_{2\gamma
} \le M\|h\|_{1;\gamma} \bigl(\|h\|_{1;\gamma} + \|X
\|_{\gamma
} \bigr) .
\end{equation}
Similarly, we check that there exists a constant $M$ such that
%
%
%e4.7 #&#
\begin{equation}
\label{eboundTR2} \bigl\|\tau_h(X,\XX) - \tau_h(\tilde X,
\tilde\XX)\bigr\|_\gamma\le M \|h\|_{1;\gamma} \bigl\|(X,\XX) -
(\tilde X,
\tilde\XX)\bigr\|_\gamma,
\end{equation}
so that $\tau_h$ is Lipschitz continuous as a map from $\CC^\gamma
_g([0,T],\R^d)$ to itself.

Denote now by $\CW_\gamma$ the completion of
$\CC^{\infty}_0 (\mathbb{R}_-;\bbR^d)$
with respect to the norm
%
%
%e4.8 #&#
\begin{equation}
\label{noisespace} |\!|\!|\omega|\!|\!|_{\gamma} \equiv\mathop{
\sup_{s,t \in\bbR_- }}_{ s
\neq t} \frac{|\omega(t) - \omega(s)|
}{ |t-s|^\gamma(1 + |t| + |s|)^{{1}/{ 2}}} .
\end{equation}
For $H \in(0,1)$ and $\gamma\in(0, H)$, it can be shown that there
exists a probability measure $\mathbb{P}_-$ on $\CW_{\gamma}$
such that the canonical process associated to $\mathbb{P}_-$ is a
fractional Brownian motion with Hurst parameter $H$~\cite{Hair05}.

Notice now that the operator $\CG$ given by (\ref{eqnopera}) is actually
defined on all of~$\CW_\gamma$.
Indeed, similar to~\cite{HairOhas07}, Proposition~A.2, it can be
checked that
the kernel $g$ defined in equation (\ref{eqngfun}) is smooth away
from $0$ and that its derivative satisfies
%
%
%e4.9 #&#
\begin{equation}
\label{ebehaveg} g'(t) = \CO(1) ,\qquad t \ll1 ,\qquad g'(t) =
\CO\bigl(t^{H-{3}/{2}}\bigr) ,\qquad t \gg1 .
\end{equation}
It follows that, for $w_{-} \in\CW_\gamma$, one has the bound
\[
\bigl|(\CG w_{-})'(t)\bigr| \le C |\!|\!|w_{-}
|\!|\!|_\gamma\bigl(t^{\gamma-1} + t^{\gamma-{1}/{ 2}} \bigr) ,
\]
where $|\!|\!|\cdot|\!|\!|_\gamma$ is as in (\ref{noisespace}).
In particular, over every finite time interval there exists a constant
$M$ such that
%
%
%e4.10 #&#
\begin{equation}
\label{eboundCG} \|\CG w_{-}\|_{1;\gamma} \le M
|\!|\!|w_{-}|\!|\!|_\gamma.
\end{equation}
As a consequence of this discussion, (\ref{erepr}) makes sense in
$\CC_g^\gamma(\R_+,\R^d)$
for every $w_- \in\CW_\gamma$, and this is how we will interpret
this identity from now on.

%s4.1 #&#
\subsection{Derivatives of the solutions}

We now derive expressions for the derivatives of solutions to (\ref
{eqnSDErpi}), both with respect
to its initial condition and with respect to the driving noise. For
this, we make the following
assumption, which will be enforced throughout the whole article:

%
%as1 #&#
\begin{assumption}\label{assgrowth}
The vector fields $V_j$ are $\CC^\infty$ and all of their derivatives
grow at most polynomially fast.
Furthermore, for every initial condition $z \in\R^n$, every final
time $T$ and every $(X,\bbX) \in\CD_g^\gamma([0,T], \R^d)$
with $\gamma> \frac{1}{ 3}$, (\ref{eqnSDErpi}) has a solution up to
time $T$.
\end{assumption}

%
%re7 #&#
\begin{remark}
We assume polynomial growth so that we can bound the Malliavin
derivatives in terms of moments
of the Jacobian (see Theorem~\ref{lemMallbdrSDE} below), a condition
which is typically not too hard to verify. Otherwise, our conditions
would be very awkward to state, for a rather minor gain in generality.
\end{remark}

%
%re8 #&#
\begin{remark}\label{remunifbound}
Since the solution to (\ref{eqnSDErpi}) depends continuously on both
its initial condition and the rough path $(X,\bbX)$
\cite{FrizVict10}, it follows from a simple compactness argument
that, for every $R>0$ and every final time $T>0$,
there exists a constant~$M$ such that, if we denote by $(Z,Z')$ the
solution to (\ref{eqnSDErpi}), the bound
\[
\bigl\|\bigl(Z,Z'\bigr)\bigr\|_{X,\gamma} \le M
\]
holds uniformly over all initial conditions $|z| \le R$ and all driving
noises  $\|(X,\bbX)\|_\gamma\le R$.
Here, we use the fact that, over finite time intervals, the embedding
$\CD_g^\gamma\hookrightarrow\CD_g^\beta$ is compact
for $\beta< \gamma$~\cite{FrizVict10}, Proposition~8.17, and that
the continuous dependence on the driving path also holds in $\CD
_g^\beta$.
\end{remark}

For an initial condition $z$ and an instance of the driving noise $w =
(w_{-}, w_+)$, let $ \Phi_t(z,w_+)$ denote the solution map of (\ref
{eqnSDErpi}),
\[
Z_t = \Phi_t(z,w_+) .
\]
Note that for defining the solution, we only use $w_+$ and do not use
$w_{-}$, the past of the driving noise.
Define the Jacobian
\[
J_{0,t} \eqdef\frac{\partial{\Phi_t(z,w_+)}}{\partial z } ,
\]
and, for notational convenience, set $V = (V_1,V_2, \ldots, V_d)$.
Then the Jacobian $J_{0,t}$ and its inverse satisfy the (rough)
evolution equations
%
%
%e4.11 #&#
\begin{subequation}\label{esdejac}
%
%
%e4.11.a #&#
%e4.11.b #&#
\begin{eqnarray}
dJ_{0,t} &=& DV_0(Z_t) J_{0,t} \,dt +
DV(Z_t) J_{0,t} \,dX_t ,\label{eqnsflow}
\\
dJ^{-1}_{0,t} &=& -J^{-1}_{0,t}
DV_0(Z_t) \,dt - J_{0,t}^{-1}
DV(Z_t) \,dX_t . \label{eqnsflowinv}
\end{eqnarray}
\end{subequation}
Here, both $J$ and $J^{-1}$ are $n\times n$ matrices, and $J_{0,0} =
J^{-1}_{0,0} = 1$.
In order to deduce~(\ref{eqnsflowinv}) from (\ref{eqnsflow}), we used
the chain rule,
which holds provided that $(X,\bbX) \in\CC_g^\gamma$.

%Of course as seen from \eqref{eqnsflow} and \eqref{eqnsflowinv}, both
%$J_{0,t}$ and
%$J_{0,t}^{-1}$ satisfy SDEs driven by rough paths and hence by Theorem
%weakly controlled rough paths themselves.
%If the vector fields $V_i$ were linear ($V_i(z) = A_i z + b_i$),
%we have the following pathwise bounds for $J_{0,T}, J_{0,T}^{-1}$ (
%To show the smoothness of the laws of $Z_t$, a key ingredient is the
%moment bounds of
%the form
% \forall p \geq1.
%Notice that the above does not follow from the pathwise bounds
%therefore $1/\g> 2$) is a Gaussian process it is well known the
%expectation of the right-hand side of
%$\eqref{eqnpfrizjac}$ is finite if and only ${1\frac}{ \g} \leq2$.
%At the moment
%we are not aware of any result which gives the moment bounds

We now consider the effect on the solution of a variation, not of the
initial condition, but of the driving noise
itself. For this, we define the operators $\A_T \dvtx
L^2([0,T],\mathbb
{R}^d) \mapsto\mathbb{R}^n$ by
%
%
%e4.12 #&#
\begin{equation}
\A_T v = \int_{0}^{T}
J_{0,s}^{-1} V(Z_s) v(s) \,ds .\label{eqnopA}
\end{equation}
A particular role will be played by $\A_T^*\dvtx\mathbb{R}^n \mapsto
L^2([0,T], \R^d)$, the adjoint of $\A_T$,
which is given by
%
%
%e4.13 #&#
\begin{equation}
\label{eqnopastar} \bigl(\A^*_T \xi\bigr) (s) = {V(Z_s)}^*
\bigl(J_{0,s}^{-1}\bigr)^*\xi,\qquad\xi\in\mathbb{R}^n
.
\end{equation}
It is known~\cite{FrizVict10} that for every sample path $w_+$ of
fractional Brownian motion
in $\CD^\gamma_g$ and for any
fixed $T$, the map\setcounter{footnote}{2}\footnote{Since $\CD_g^\gamma$ is not a linear
space, the ``addition'' of the paths $w_+$ and $h$
should be interpreted in the sense of
(\ref{eqnmaptauh}) below.}
%
%
%e4.14 #&#
\begin{equation}
\label{eHder} h \in\mathcal{H}_{H,+} \mapsto\Phi_T(z, w_+ +
h)
\end{equation}
is Fr\'echet differentiable, where
$\mathcal{H}_{H,+}$ denotes the Cameron--Martin space of the Gaussian
process $w_+$.
In fact this Fr\'echet differentiability in Cameron--Martin directions
holds in great generality for RDE solutions driven by rough paths \cite
{FrizVict10}. Furthermore,
setting $h(s) = \int_0^s v(r) \,dr$ for some $v$, one has the identity
%
%
%e4.15 #&#
\begin{equation}
\label{ederNoise} D_h \Phi_T(z, w_+ + h)
|_{h = 0} = J_{0,T}\CA_T v ,
\end{equation}
whenever $v \in L^2$.
Note that, by (\ref{erepr}), the space $\CH_{H,+}$ consists of those
paths $h$ such that
$h = \CD^{{1}/{ 2}-H} \tilde h$ for some $\tilde h$ in the
Cameron--Martin space of $W$, which in turn
is equal to $H^1$, the space of square integrable functions with square
integrable weak derivative.
If $H \neq\frac{1}{ 2}$, the corresponding element $v$ does not
necessarily belong to $L^2$,
so that one may wonder what the meaning of (\ref{ederNoise}) is in general.
Writing $F_s = J_{0,s}^{-1} V(Z_s)$
as a shorthand, a calculation shows that, for $1/3 < H < 1/2$, there is
a constant $c$ such that
%
%
%e4.16 #&#
\begin{eqnarray}
\label{eADv}\qquad &&\CA_T \CD^{{1}/{ 2}-H} v
\nonumber
\\[-8pt]
\\[-8pt]
\nonumber
&&\qquad= c\int_{0}^{T}
\biggl( \int_s^T (r-s)^{H-{3}/{ 2}}
(F_s - F_r) \,dr + \frac{(T-s)^{H-{1}/{
2}} }{ {1}/{ 2}-H}F_s
\biggr) v(s) \,ds ,
\end{eqnarray}
so that $|\CA_T \CD^{{1}/{ 2}-H} v| \le M \|F\|_{\CC^\gamma} \|
v\|_{L^2}$ for some constant $M$, provided that
$\gamma> \frac{1}{ 2} - H$. Since, in our particular case, $F$ is a
rough path controlled by $(X,\bbX)$,
the condition $\|F\|_{\CC^\gamma} < \infty$ for $\gamma> \frac{1}{
2} - H$ can always be satisfied when
$\frac{1}{ 2} - H < H$, namely when $H > \frac{1}{ 4}$.
See~\cite{MallRough} for more discussion on why a bound like this is
true in general. The reason for deriving
the explicit expression (\ref{eADv}) in our case is that it will be
useful in the next subsection.

In the sequel, we will write $\DD_v Z^z_T$ as a shorthand for the
derivative of the solution map
in the direction $h = \int_0^\cdot v(s) \,ds$, that is,
%
%
%e4.17 #&#
\begin{eqnarray}
\label{eqncontbas} \DD_v Z^z_T =
D_h \Phi_T(z, w_+) = J_{0,T}\A_T
v .
\end{eqnarray}
We also set
%
%
%e4.18 #&#
\begin{equation}
\label{edefDsZ} \DD_s Z^z_t =
J_{s,t} V\bigl(Z_s^z\bigr) \eqdef
J_{0,t} J_{0,s}^{-1} V\bigl(Z_s^z
\bigr) ,
\end{equation}
so that $\DD_v Z^z_T$ is the $L^2$-scalar product of $\DD Z_T^z$ with $v$.

%s4.2 #&#
\subsection{Malliavin differentiability of the solutions}

Using representation $(\ref{erepr})$, the solution map $\Phi
_t(z,w_+)$ conditioned
on the past $w_-$ of the driving noise
may be viewed as a functional of an underlying Wiener process on
$[0,\infty)$ which then
allows us to define the Malliavin derivative of the solution map $\D
Z^z_t = \D\Phi_t(z,w_+)$
in the usual way. For $H = 1/2$, we have $\D_s Z^z_t = \DD_s Z^z_t$ where
$\DD$ is as defined in~(\ref{edefDsZ}). Thus we focus on the case $H
< 1/2$ below.
As shown in~\cite{MallRough}, the Malliavin derivative is related to the
Fr\'echet derivative given by (\ref{eHder}) in the following way.
For any $\int_0^\cdot v(s) \,ds \in\mathcal{H}_{H,+}$, define
$\tilde v = \CI^{{1}/{ 2} - H} v$.
Then we have the identity
%
%
%e4.19 #&#
\begin{equation}
\D_{\tilde{v}} Z^z_T = \frac{1}{ \alpha_H}
\DD_v Z^z_T = \frac{1}{
\alpha_H}J_{0,T}
\CA_T v = \frac{1}{ \alpha_H}J_{0,T} \CA_T
\CD^{{1}/{2}-H}\tilde v, \label{eqnDVsDv}
\end{equation}
where $\alpha_H$ is as in (\ref{erepr}). In line with the notation
from~\cite{Nual06},
we define $s \mapsto\D_s Z^z_t$ to be the stochastic process such
that the relation
\[
\D_{\tilde{v}} Z^z_t = \int_0^t
\tilde v(s) \D_s Z^z_t \,ds
\]
holds for every $\tilde v \in L^2$.
Comparing this with (\ref{eADv}), we see that one has the identity
%
%
%e4.20 #&#
\begin{eqnarray}
\label{eexprMalDer} \D_s Z^z_t &=& c\int
_s^t (r-s)^{H-{3}/{ 2}} \bigl(J_{s,t}V
\bigl(Z_s^z\bigr) - J_{r,t}V
\bigl(Z_r^z\bigr) \bigr) \,dr
\nonumber
\\[-8pt]
\\[-8pt]
\nonumber
&&{} + \frac{2c}{ 1-2H} (T-s)^{H-{1}/{ 2}}J_{s,t}V
\bigl(Z_s^z\bigr)
\end{eqnarray}
for some fixed constant $c$. (Furthermore, $\D_s Z^z_t = 0$ for $s \ge t$.)
In general, we can rewrite this as
%
%
%e4.21 #&#
\begin{equation}
\label{erelDD} \D Z^z_t = \CD_{+}^{{1}/{ 2}-H}
\DD Z^z_t ,
\end{equation}
where $\DD_s Z^z_t$ is as in (\ref{edefDsZ}), with $\DD_s Z^z_t = 0$
for $s \ge t$,
and $\CD_{+}^{{1}/{ 2}-H}$ is the linear operator given by
\[
\bigl(\CD_{+}^{{1}/{ 2}-H} f \bigr) (s) = c\int
_s^\infty(r-s)^{H-{3}/{ 2}} \bigl(f(r) - f(s)
\bigr) \,dr .
\]

The aim of this section is to obtain a priori bounds on the
higher-order Malliavin derivatives of the solution.
As a first step, we obtain pointwise bounds on multiple derivatives of
the solution map. In view of (\ref{edefDsZ}),
we will need to compute $\DD_s J_{0,t}$ in order to obtain such
bounds. At this stage, let us put indices back into the
various expressions in order to clarify the precise meaning of the
various expressions that appear. We
will use Einstein's convention of summation over repeated indices, and
we write $\DD_s^i$ for the derivative with
respect to the $i$th component of the driving noise $(X,\bbX)$.
It is clear that $\DD_s J_{0,t} = 0$ for $t < s$. Furthermore, we see
from (\ref{eqnsflow}) that
%
%
%e4.22 #&#
\begin{equation}
\label{einitialDerJ} \DD_s^i J_{0,s}^{k\ell}
= D_mV_i^k(Z_s)
J_{0,s}^{m\ell} .
\end{equation}
For $t > s$, we formally differentiate (\ref{eqnsflow}), and we use
identity (\ref{edefDsZ}) to obtain the rough evolution equation
\begin{eqnarray*}
d \DD_s^i J_{0,t}^{k\ell} &=&
D_mV_j^k(Z_t)
\DD_s^i J_{0,t}^{m\ell
}
\,dX^j(t) + D_mV_0^k(Z_t)
\DD_s^i J_{0,t}^{m\ell} \,dt
\\
&&{} + D^2_{mn} V_j^k(Z_t)
J_{0,t}^{m\ell} J_{s,t}^{no}
V_i^o(Z_t) \,dX^j(t)\\
&&{} +
D^2_{mn} V_0^k(Z_t)
J_{0,t}^{m\ell} J_{s,t}^{no}
V_i^o(Z_t) \,dt .
\end{eqnarray*}
Note now that this is a linear inhomogeneous equation for $\DD_s^i
J_{0,t}^{k\ell}$, where the linear part has exactly
the same structure as (\ref{eqnsflow}). As a consequence, we can solve
it using the variation of constants formula
which, when combined with (\ref{einitialDerJ}), yields the expression
%
%
%e4.23 #&#
\begin{eqnarray}
\label{ederJac} \DD_s^i J_{0,t}^{k\ell}
&=& J_{s,t}^{kj} D_mV_i^j(Z_s)
J_{0,s}^{m\ell
} \nonumber\\
&&{}+ \int_s^t
J_{r,t}^{kp} D^2_{mn}
V_j^p(Z_r) J_{0,r}^{m\ell}
J_{s,r}^{nq} V_i^q(Z_r)
\,dX^j(r)
\\
&&{}+ \int_s^t J_{r,t}^{kp}
D^2_{mn} V_0^p(Z_r)
J_{0,r}^{m\ell} J_{s,r}^{nq}
V_i^q(Z_r) \,dr .\nonumber
\end{eqnarray}
A similar identity also holds for $\DD_s J^{-1}_{0,t}$, but
the precise form of this expression is unimportant as will be seen presently.

We now introduce the following notation in order to keep track of the
terms appearing in
the expressions for higher order Malliavin derivatives.
Denote by $\TT$ the space of finite rooted trees, and by $\FF$ the
space of finite forests (unordered finite collections
of trees, allowing for repetitions). Formally, we denote by $(\tau
_1,\ldots,\tau_k)$ the forest consisting of the trees
$\tau_1,\ldots,\tau_k$. For $F = (\tau_1,\ldots,\tau_k) \in\FF
$, we also write
$[F] = [\tau_1,\ldots,\tau_k] \in\TT$ for the tree obtained by
gluing the roots of the trees of $F$ to a common new root.

For any forest $F \in\FF$, we then build a sequence of subsets $\CV
_F^k \subset\CC_X^\gamma([0,T],\R)$ with $k \ge1$ in the following
way. For the empty forest $(\cdot)$, we set $\CV_{(\cdot)}^k = \{1\}$ for every~$k$.
For $F = (\bullet)$, the forest consisting of one single tree, which itself
consists only of a root, we set
\[
\CV_{(\bullet)}^k = \bigl\{J_{0,\cdot}^{ij},
\bigl(J_{0,\cdot
}^{-1} \bigr)^{ij}, D^{(\ell)}_{i_1,\ldots,i_\ell}
V_m^j(Z_\cdot) \dvtx\ell\in\{0,\ldots, k\}
\bigr\} ,
\]
where $m \in\{0,\ldots,d\}$, and the indices $i$, $j$ and $i_1,\ldots,
i_\ell$ belong to $\{1,\ldots,n\}$.
For forests $F = (\tau_1,\ldots,\tau_m)$ consisting of more than one
tree, we set
%
%
%e4.24 #&#
\begin{equation}
\label{edefProdForest} \CV_{(\tau_1,\ldots,\tau_m)}^k = \bigl
\{Y_1\cdots Y_m \dvtx Y_j \in
\CV_{(\tau_j)}^k ,\forall j\bigr\} .
\end{equation}
In other words, the processes contained in $\CV_{(\tau_1,\ldots,\tau
_m)}^k$ are obtained by multiplying
together the processes contained in $\CV_{(\tau_j)}^k$. Finally, if
$F$ consists of a single tree consisting of more than
just one root, so that $F = [G]$ for some forest $G$, we set
%
%
%e4.25 #&#
\begin{equation}
\label{edefRecForest} \CV_{[G]}^k = \biggl\{\int
_0^\cdot Y_s \,dX^\ell(s)
, \int_0^\cdot Y_s \,ds \dvtx Y \in
\CV_G^k , \ell\in\{1,\ldots,d\} \biggr\} .
\end{equation}
This construction has the following feature:

%
%le4 #&#
\begin{lemma}\label{leminducConstr}
There exists a map $F \mapsto T_F$ from $\FF$ to $2^\FF$, the set of
subsets of $\FF$, with the following properties:
\begin{itemize}
\item The set $T_F$ is finite for every $F \in\FF$.
\item For every $F \in\FF$, $k \ge1$, there exist coefficients
$c_{Y,U,\bar U}^i$ taking values in $\{0,1,-1\}$ such that
the identity
\[
\DD_s^i Y_t = \sum
_{G, \bar G \in T_F} \sum_{U \in\CV_G^{k+1}}\sum
_{\bar U\in\CV_{\bar G}^{k+1}} c_{Y,U,\bar U}^i U_s \bar
U_t
\]
holds for every $Y \in\CV_F^k$ and every $0 \le s < t \le T$.
\end{itemize}
\end{lemma}

%
%re9 #&#
\begin{remark}
Objects indexed by trees arise naturally when considering higher-order
expansions (in time) of solutions
to differential equations~\cite{MR0305608},
but this is completely unrelated to the construction of this section.
In our case, the bound on the higher order Malliavin derivatives
proceeds via an
inductive argument, and the set of trees used here is simply one
relatively easy
combinatorial object having the same recursive structure as our bounds.
\end{remark}

\begin{pf*}{Proof of Lemma~\ref{leminducConstr}}
Note first that, by writing $J_{r,t} = J_{0,t}J_{0,r}^{-1}$ and
similarly for
$J_{s,r}$, we see from (\ref{ederJac}) that $\DD_s J_{0,t}$ can
indeed be written as
\[
\DD_s^i J_{0,t}^{k\ell} = \sum
_{G, \bar G \in T_J} \sum_{U\in\CV
_{G}^{2}} \sum
_{\bar U\in\CV_{\bar G}^{2}} c^{ik\ell}_{U,\bar U}
U_s \bar U_t
\]
for some coefficients $c^{ik\ell}_{U,\bar U} \in\{0,1,-1\}$, and for
\[
T_J = \bigl\{(\bullet),(\bullet,\bullet,\bullet),\bigl(\bullet,[
\bullet,\bullet,\bullet,\bullet,\bullet]\bigr)\bigr\} .
\]
The same statement holds true for $\DD_s^i (J_{0,t}^{-1})^{k\ell}$.
Furthermore, we have from (\ref{edefDsZ}) and the chain rule the identity
\[
\DD_s^i D^{(\ell)}_{i_1,\ldots,i_\ell}
V_m^j(Z_t) = D^{(\ell
+1)}_{i_1,\ldots,i_\ell,k}
V_m^j(Z_t) J_{s,t}^{kp}
V_i^p(Z_s) ,
\]
so that we have
\[
\DD_s^i D^{(\ell)}_{i_1,\ldots,i_\ell}
V_m^j(Z_t) = \sum
_{G, \bar
G \in T_D} \sum_{U\in\CV_{G}^{2}} \sum
_{\bar U\in\CV_{\bar
G}^{2}} \bar c_{U,\bar U} U_s \bar
U_t
\]
for some coefficients $\bar c_{U,\bar U} \in\{0,1,-1\}$ (depending
also on all the indices appearing on the left-hand side),
and for
\[
T_D = \bigl\{(\bullet,\bullet)\bigr\} .
\]
It follows that we can indeed find a set $T_{(\bullet)} = T_J \cup
T_D$ with the two properties stated in the lemma.

For more complicated forests, the claim follows by building $T$
recursively in the following way. If $F = (\tau_1,\ldots,\tau_m)$
for trees $\tau_j$ such that $T_{(\tau_j)}$ is known, we observe that
one has the identity
%
%
%e4.26 #&#
\begin{equation}\quad
\label{eprodY} \DD_s^j Y_1(t)\cdots
Y_m(t) = \sum_{i=1}^m
Y_1(t)\cdots Y_{i-1}(t) \DD_s^jY_{i}(t)
Y_{i+1}(t)\cdots Y_{m}(t) .
\end{equation}
As a consequence, if we write $F \oplus G$ for the union of two forests
and $F \ominus G$ for the
forest obtained by removing from $F$ its subforest $G$, we can set
%
%
%e4.27 #&#
\begin{equation}
\label{edefT1} T_{(\tau_1,\ldots,\tau_m)} = \bigcup_{i=1}^m
\bigl\{T_{(\tau_i)}, F \ominus(\tau_i) \oplus T_{(\tau_i)}
\bigr\} .
\end{equation}
It follows from (\ref{eprodY}) and (\ref{edefProdForest}) that this
definition does indeed ensure that the requested properties
are satisfied.\vadjust{\goodbreak}

It remains to consider the case $F = (\tau)$ for some nontrivial tree
$\tau$. In this case, there exists
a forest $G$ such that $\tau= [G]$ and elements in $\CV_F^k$ are
given by (\ref{edefRecForest}).
Note now that one has the identities
\begin{eqnarray*}
\DD_s^i\int_0^t
Y_r \,dX^\ell(r) &=& \delta_{i\ell}
Y_s + \int_0^t
\DD_s^i Y_r \,dX^\ell(r) ,
\\
\DD_s^i \int_0^\cdot
Y_r \,dr &=& \int_0^t
\DD_s^iY_r \,dr .
\end{eqnarray*}
As a consequence, if we set
%
%
%e4.28 #&#
\begin{equation}
\label{edefT2} T_{[G]} = T_G \cup\bigl\{G, (\cdot)\bigr\}
\cup\bigl\{\bigl([H]\bigr) \dvtx H \in T_G\bigr\} ,
\end{equation}
the requested properties are again satisfied by induction. Since every
forest can be built from elementary trees
by the two operations considered in (\ref{edefT1}) and (\ref
{edefT2}), this concludes the proof.
\end{pf*}

For our purpose, this has the following useful consequence. For a fixed
final time $T$, define the controlled rough path
$\CJ^z \in\CD_g^\gamma([0,T], \R^{2n^2+n})$ by
%
%
%e4.29 #&#
\begin{equation}
\label{edefCJ} \CJ_t^z = \bigl(Z_t,
J_{0,t}, J_{0,t}^{-1}\bigr) .
\end{equation}
(Note that both $J$ and $J^{-1}$ also implicitly depend on the starting
point~$z$.)
We then have an a priori bound on the derivatives of the solution with
respect to the driving noise in terms of $\CJ^z$:

%
%pr2 #&#
\begin{proposition}\label{propreprDerSol}
Let $A_t$ denote any component of the vector $\CJ^z_t$.
Under Assumption~\ref{assgrowth}, for every multiindex $\alpha=
(\alpha_1,\ldots,\alpha_\ell)$
there exists a finite index set $T_\alpha$ and elements $F^k_j \in\CC
_X^\gamma([0,T],\R)$
with $k \in T_\alpha$ and $j \in\{1,\ldots,|\alpha|+1\}$, such that
the identity
%
%
%e4.30 #&#
\begin{equation}
\label{ereprMultDer} \DD_{s_1}^{\alpha_1}\cdots
\DD_{s_\ell}^{\alpha_\ell} A_t = \sum
_{k \in T_\alpha} F^k_1(s_1)\cdots
F^k_\ell(s_\ell)F^k_{\ell+1}(t)
,
\end{equation}
holds for every $0\le s_1 < \cdots< s_\ell< t \le T$.

Furthermore, there exist constants $M$ and $p$ depending only on
$\alpha$ and $T$ such that the bound
\[
\bigl\|F^k_j\bigr\|_{X,\gamma} \le M \bigl(1+ \bigl\|\bigl(
\CJ^z, {\CJ^z}'\bigr)\bigr\|_{X,\gamma
}
\bigr)^p ,
\]
holds for every $k$ and $j$.
\end{proposition}

\begin{pf}
The first claim follows immediately from Lemma~\ref{leminducConstr}
by induction on~$|\alpha|$.
The second claim follows from the construction of the sets $\CV_F^k$,
combined with Lemma~\ref{lemphiYlem}
and Theorem~\ref{thmGubi}.
\end{pf}

In the particular case when $X$ is fractional Brownian motion with
Hurst parameter $H$,
it follows from Proposition~\ref{propreprDerSol} that, if we consider
the Malliavin derivatives $\D_s$
with respect to the underlying Wiener process as at the beginning of
this section, we have the following bound:\vadjust{\goodbreak}

%
%th4.1 #&#
\begin{theorem}\label{lemMallbdrSDE}
As above, let $A_t$ denote any component of the vector $\CJ^z_t$, and
let $(X,\bbX)$ be fractional Brownian motion
with Hurst parameter $H \in(\frac{1}{ 3},\frac{1}{ 2}]$.
Under Assumption~\ref{assgrowth}, for every multiindex $\alpha=
(\alpha_1,\ldots,\alpha_\ell)$,
every $\gamma\in(\frac{1}{ 3}, H)$, every $\delta> 0$ and every
$T>0$, there exist constants $M$ and $p$ such that the
bound
\[
\Biggl(\bigl|\D_{s_1}^{\alpha_1}\cdots\D_{s_\ell}^{\alpha_\ell}
A_{s_{\ell+1}} \bigr|\prod_{j = 1}^{\ell}
|s_{j+1}-s_j|^{1-2H + \delta
} \Biggr) \le M \bigl(1+ \bigl\|
\bigl(\CJ^z, {\CJ^z}'\bigr)
\bigr\|_{X,\gamma} \bigr)^p ,
\]
holds uniformly for all $0 \le s_1<\cdots< s_{\ell+1} \le T$.
Furthermore, the exponent $p$ can be chosen to depend only on $|\alpha|$.
\end{theorem}

%
%re10 #&#
\begin{remark}
Since the function $t \mapsto t^{2H-1-\delta}$ is square integrable
near the origin for $\delta$
sufficiently small,
the random variable $A_t$ belongs to the stochastic Sobolev space $\D
^{\infty}_{\loc}$.
If furthermore $\E\|(\CJ, \CJ')\|_{X,\gamma}^p < \infty$ for
every~$p$, then $A_t$ belongs to the
stochastic Sobolev space $\D^{\infty}$.
See~\cite{Nual06}, page~49, for the definitions of $\D^{\infty
}_{\loc}$ and $\D^{\infty}$.
\end{remark}

\begin{pf*}{Proof of Theorem~\ref{lemMallbdrSDE}}
The proof for $H = 1/2$ follows trivially from Proposition \ref
{propreprDerSol}. Consider now $s_1<\cdots<s_\ell$ to be fixed and
consider, for $j=0,\ldots,\ell$, the sequence of functions
\[
F^{(j)}(r_{j+1},\ldots,r_\ell) =
\D_{s_1}\cdots\D_{s_{j}} \DD_{r_{j+1}} \cdots
\DD_{r_\ell} A_{t} ,
\]
so that our aim is to obtain a bound on $F^{(\ell)}$. Note that, by
(\ref{erelDD}), the $F^{(j)}$ satisfy the recursive formula
%
%
%e4.31 #&#
\begin{eqnarray}\label{eFRec}
\qquad F^{(j)}(r_{j+1},\ldots,r_\ell) &=& c\int
_{s_j}^\infty(r-s_j)^{H-{3}/{ 2}}
\bigl(F^{(j-1)}(r,r_{j+1},\ldots,r_\ell)
\nonumber
\\[-8pt]
\\[-8pt]
\nonumber
&&\hspace*{90pt}{}- F^{(j-1)}(s_j,r_{j+1},\ldots,r_\ell)
\bigr) \,dr .
\end{eqnarray}
We claim now that, for every $j$, there exists an index set $T_j$ and a
family of functions $F^{(j,k)}_i$ such that,
for $s_j < r_{j+1}<\cdots< r_\ell< t$, one has the identity
%
%
%e4.32 #&#
\begin{equation}
\label{eexprF} F^{(j)}(r_{j+1},\ldots,r_\ell) =
\sum_{k \in T_j} \prod_{i=j+1}^\ell
F^{(j,k)}_i(r_i) .
\end{equation}
Furthermore, for every $\beta\in(\frac{1}{ 2} - H,H)$, there exists
a constant $M$ independent of $s_1,\ldots,s_\ell$ such that
these functions satisfy the bound
%
%
%e4.33 #&#
\begin{equation}
\label{eboundProdF} \prod_{i=j+1}^\ell
\bigl\|F^{(j,k)}_i(r_i)\bigr\|_{\beta,j} \le M\prod
_{i =
1}^{j} |s_{j+1}-s_j|^{H-{1}/{ 2}-\beta}
\bigl(1+ \bigl\|\bigl(\CJ^z, {\CJ^z}'\bigr)
\bigr\|_{X,\gamma} \bigr)^p\hspace*{-35pt}
\end{equation}
for some fixed $p>0$.
Here, where we denote by $\|F\|_{\beta,j}$ the $\CC^\beta$-norm (not
seminorm!) of $F$, restricted to the interval\vadjust{\goodbreak}
$[s_{j+1},t]$. In the special case $j = \ell$, this is just $|F(t)|$.
Once we show that (\ref{eexprF}) and (\ref{eboundProdF})
hold, the proof is complete since the special case $j = \ell$ and the
choice $\beta= \frac{1}{ 2} - H+\delta$ yields the stated claim
for~$\delta$ sufficiently small. For larger values of $\delta$, the
claim can easily be reduced to that for small $\delta$.

Note furthermore that $F^{(j)}(r_{j+1},\ldots,r_\ell) = 0$ if there
exists $i > j$ such that $r_i > t$ and that
the function $F^{(j)}$ is symmetric under permutations of its
arguments. As a consequence, (\ref{eexprF}) is sufficient
to determine $F^{(j)}$.

The proof now goes by induction over $j$. For $j=0$, we have
\[
F^{(0)}(r_1,\ldots,r_\ell) =
\DD_{r_1}\cdots\DD_{r_\ell} A_{t} ,
\]
which is indeed of the form (\ref{eexprF}) by Proposition~\ref
{propreprDerSol}. In this case, the bound (\ref{eboundProdF})
reduces to the statement that the $F^{(0,k)}_i$ are $\beta$-H\"older
continuous, which is
also a consequence of Proposition~\ref{propreprDerSol}. In order to
make use of the recursion (\ref{eFRec}), we have to
rewrite it in such a way that the arguments of $F^{(j-1)}$ are always
ordered. Using the recursion hypothesis, we then have
the identity
\begin{eqnarray*}
&&F^{(j)}(r_{j+1},\ldots,r_\ell)\\[-2pt]
 &&\qquad= c\sum
_{k \in T_{j-1}} \int_{s_j}^{r_{j+1}}
(r-s_j)^{H-{3}/{ 2}} \bigl(F^{(j-1,k)}_j(r)-
F^{(j-1,k)}_j(s_j) \bigr) \,dr\\[-2pt]
&&\hspace*{32pt}\qquad\quad{}\times \prod_{i>j}F^{(j-1,k)}_i(r_i)
\\[-2pt]
&&\qquad\quad{} + c\sum_{k \in T_{j-1}}\sum_{i > j}
\int_{r_i}^{r_{i+1}} (r-s_j)^{H-{3}/{ 2}}
F^{(j-1,k)}_i(r) \,dr
\\[-2pt]
&&\hspace*{62pt}\qquad\quad{} \times\Biggl(\prod_{q=j}^{i-1}F^{(j-1,k)}_q(r_{q+1})
\Biggr) \Biggl(\prod_{q=i+1}^\ell
F^{(j-1,k)}_q(r_q) \Biggr)
\\[-2pt]
&&\qquad\quad{} - \frac{2c }{ 1-2H} (r_{j+1} - s_j)^{H-{1}/{ 2}}
F^{(j-1,k)}_j(s_j) \prod
_{i>j}F^{(j-1,k)}_i(r_i)
\\[-2pt]
&&\qquad\eqdef T_1 + T_2 + T_3 .
\end{eqnarray*}
Rewriting the integral from $r_i$ to $r_{i+1}$ appearing in $T_2$ as
\[
\int_{r_i}^{t} (r-s_j)^{H-{3}/{ 2}}
F^{(j-1,k)}_i(r) \,dr - \int_{r_{i+1}}^{t}
(r-s_j)^{H-{3}/{ 2}} F^{(j-1,k)}_i(r) \,dr ,
\]
we see that $F^{(j)}$ is indeed again of the form~(\ref{eexprF}). It
remains to show that bound~(\ref{eboundProdF}) holds. To show that it holds for $T_1$, write
\[
G_k(s) = \int_{s_j}^{s}
(r-s_j)^{H-{3}/{ 2}} \bigl(F^{(j-1,k)}_j(r)-
F^{(j-1,k)}_j(s_j) \bigr) \,dr ,\vadjust{\goodbreak}
\]
so that one has, for $s > s_j$, the bound
\[
\bigl|\partial_s G_k(s)\bigr| \le(s-s_j)^{H-{3}/{ 2}+\beta}
\bigl\| F^{(j-1,k)}_j\bigr\|_{\beta,j-1} .
\]
In particular, one has for $s > s_{j+1}$ the bound
\[
\bigl|\partial_s G_k(s)\bigr| \le(s_{j+1} -
s_j)^{H-{1}/{ 2}} (s-s_j)^{\beta-1}
\bigl\|F^{(j-1,k)}_j\bigr\|_{\beta,j-1} .
\]
Furthermore, we obtain in a similar way the bound
\[
\bigl|G_k(s_{j+1})\bigr| \le M(s_{j+1} -
s_j)^{H-{1}/{ 2}+\beta} \bigl\| F^{(j-1,k)}_j
\bigr\|_{\beta,j-1}
\]
for some constant $M$,
so that a straightforward calculation yields
\[
\|G_k\|_{\beta,j} \le M (s_{j+1} -
s_j)^{H-{1}/{ 2}} \bigl\| F^{(j-1,k)}_j
\bigr\|_{\beta,j-1}
\]
for some constant $M$. The requested bound on $T_1$ (actually a bound
that it better than requested)
then follows at once.

To bound $T_2$, we proceed similarly by setting
\[
G_k(s) = \int_{s}^{t}
(r-s_j)^{H-{3}/{ 2}} F^{(j-1,k)}_i(r) \,dr ,
\]
and noting that $G_k(t) = 0$ and
\[
\bigl|\partial_s G_k(s)\bigr| \le(s_{j+1} -
s_j)^{H-{1}/{ 2}-\beta} (s-s_j)^{\beta-1}
\bigl\|F^{(j-1,k)}_i\bigr\|_{\beta,j-1} .
\]
It follows as above that
\[
\|G_k\|_{\beta,j} \le M (s_{j+1} -
s_j)^{H-{1}/{ 2}-\beta} \bigl\| F^{(j-1,k)}_i
\bigr\|_{\beta,j-1}
\]
as requested. Finally, the bound on $T_3$ follows in the same way.
\end{pf*}

%s5 #&#
\section{Regularity of laws} \label{secinvmmatrix}

Our aim in this section is to show that if the vector fields $V$
satisfy H\"ormander's celebrated Lie bracket condition (see below),
then the Malliavin matrix
of the process $Z_t$ is almost surely invertible, and to obtain quantitative
bounds on its lowest eigenvalue.

In order to state H\"ormander's condition, we define recursively the
families of vector fields
%
%
%e5.1 #&#
\[
\mathcal{V}_0 = \{V_k \dvtx k \ge1 \} ,\qquad
\mathcal{V}_{n+1} = \CV_n \cup\bigl\{[U,V_k]
\dvtx U \in\CV_n , k \ge0 \bigr\} ,
\]
where $[U,V]$ denotes the Lie bracket between the vector fields $U$ and
$V$. Note that under
Assumption~\ref{assgrowth}, the elements in $\CV_n$ also have
derivatives of all orders that grow at most polynomially.
We now formulate H\"ormander's bracket condition~\cite{HorActa}:

%
%as2 #&#
\begin{assumption}\label{assHormander}
For every $z_0 \in\mathbb{R}^n$, there exists $N \in\mathbb{N}$
such that the identity
%
%
%e5.2 #&#
\begin{equation}
\operatorname{span}\bigl\{U(z_0) \dvtx U \in{\mathcal{V}}_N
\bigr\} = \mathbb{R}^n \label{eqnHorm}
\end{equation}
holds.
\end{assumption}

%For the proof of the strong Feller property later on, we need our
%results in terms
%of the conditioned fractional Brownian motion. As explained in
%$X_t = \tilde{X}_t + m_t$ and thus obtain
%dZ_t = V_0(Z_t) dt + V(Z_t) dm_t + V(Z_t) d\tilde{X}_t ,
%where $\tilde{X}_t$ is an independent one sided fBm on $\mathbb{R}_+$.
%The difficulty comes from the time derivative
%of the conditional mean $m_t$, which diverges at $t=0$ like $
%Our treatment of this singularity is similar to~\cite{HairPill10}
%which was inspired by~\cite{HairMatt09}.
%Using an iterative argument, we will show the invertibility of the
%reduced Malliavin matrix:

It is well known from the works of Malliavin, Bismut, Kusuoka, Stroock
and others~\cite{Mal76,Bismut81,KSAMI,KSAMII,KSAMIII,Norr86,Mal97SA,Nual06}
that when the driving noise $X$ is Brownian motion, one way of proving
the smoothness of the law
of $Z_T$ under H\"ormander's condition is to first show
the invertibility of the ``reduced Malliavin matrix''\footnote{This
is a slight misnomer since our SDE is driven by fractional Brownian motion,
rather than Brownian motion. One can actually rewrite the solution as a
function of an underlying Brownian motion by making use of
representation (\ref{erepr}), but the associated Malliavin matrix has a
slightly more complicated relation to $C_T$ than usual.
Still, it will be useful to first obtain a bound on the inverse of~$C_T$.}
%This will be done in Theorem~\ref{theosmooth} below.}:
%
%
%e5.3 #&#
\begin{equation}
C_T \eqdef\A_T \A_T^* = \int
_{0}^T J_{0,s}^{-1}
V(Z_s) V(Z_s)^*\bigl(J_{0,s}^{-1}
\bigr)^*~\,ds . \label{eqnmallmat}
\end{equation}
%
%where we defined the (random) operator: $\A_T : L^2([0,T],
%and $\A_T^*: \mathbb{R}^n \mapsto L^2([0,T], \R^d)$, the adjoint of $
%it $L_2([0,T], \R^d)$ or $L_2[\mathbb{R}_+, \R^d]$ ?}
% by:
%(\A^*_T \xi)(s) = {V(Z_s)}^* (J_{0,s}^{-1})^*\xi, \xi\in

Recall that the matrix norm of a symmetric matrix is equal to its
largest eigenvalue.
Since $C_T$ is a symmetric matrix, one can write the norm of its
inverse as
%
%
%e5.4 #&#
\begin{eqnarray}
\label{eqnCTinvst} \bigl\| C_T^{-1}\bigr\|^{-1} =
\inf_{ \|\phi\| = 1} \langle v, C_T v \rangle,\qquad\phi\in
\mathbb{R}^n .
\end{eqnarray}
%
%s5.1 #&#
\subsection{\texorpdfstring{Deterministic bounds on $\|C_T^{-1}\|$}
{Deterministic bounds on ||C T -1||}}
In this subsection we only use the fact that $(X,\bbX) \in\CD^\gamma
_g$ and the
fact that $X$ is $\theta$-rough for some $\theta> H$. Thus the bounds
obtained are purely deterministic.

Before we turn to our bound on the inverse of $C_T$, let us introduce
some notation.
For any smooth vector field $U$,
define the process $\CZ_U(t) = J_{0,t}^{-1} U(Z_t)$, and set
%
%
%e5.5 #&#
\begin{equation}
\label{eqncalR} \CR_z \eqdef1 + L_\theta(X)^{-1}
+ \bigl\|(X,\bbX)\bigr\|_{\g} + \bigl\|\bigl(\CJ^z, {
\CJ^z}'\bigr)\bigr\|_{X,\gamma} + |z| ,
\end{equation}
where $\CJ^z$ is as in (\ref{edefCJ}). Here, we fix a ``roughness
exponent'' $\theta> H$ which will
appear in subsequent statements.

%
%le5 #&#
\begin{lemma} \label{lemhormit1}
Fix a final time $T>0$. Under Assumption~\ref{assgrowth}, there exist
constants $c, a > 0$ such that
the bound
%
%
%e5.6 #&#
\[
\bigl\|\bigl\langle\phi, \CZ_U(\cdot)\bigr\rangle\bigr\|
_{\infty} \le
M \CR_z^c \bigl|\langle\phi, C_T \phi
\rangle\bigr|^a ,
\]
holds for all $U \in\mathcal{V}_1$, all $\phi\in\R^n$ such that $\|
\phi\| = 1$, all initial conditions $z$, all $(X,\bbX) \in\CD
^\gamma_g([0,T],\R^d)$ and the constant $M>0$ is independent of
$X,\phi, z$.
\end{lemma}
\begin{pf} By definition we have
%
%
%e5.7 #&#
\begin{equation}
\langle\phi, C_T \phi\rangle= \sum_{i=1}^d
\int_{0}^T \bigl\langle\phi,
J_{0,s}^{-1} V_i(Z_s)\bigr
\rangle^2 \,ds = \sum_{i=1}^d
\bigl\|\bigl\langle\phi,\CZ_{V_i}(\cdot)\bigr\rangle\bigr\|_{L^2[0,T]}^2
.\label{eqninterpref} %\ge
% \sum_{i=1}^d\int_{0}^{T/2} \langle\phi, J_{0,s}^{-1} V_i(X_s)\rangle
%^2 ds .
\end{equation}
To obtain an upper bound of order $| \langle\phi, C_T \phi\rangle|^a$
on the supremum norm, our main tool is the interpolation inequality
%
%
%e5.8 #&#
\begin{equation}
\label{eqnl2linfip} \|f\|_{\infty} \leq2\max\bigl(T^{-{1}/{ 2}}\|f
\|_{L^2[0,T]} , \|f\|_{L^2[0,T]}^{{2\gamma}/{( 2\gamma+1) }} \|f
\|_{\gamma}^{{1 }/{ (2\gamma+1)}} \bigr) ,
\end{equation}
which holds for every $\gamma$-H\"older continuous function $f \dvtx
[0, T
] \mapsto\R$; see, for example,
\cite{HairPill10}, Lemma A.3. Since in our case the final time $T$
is fixed, the $L^2$ norm is controlled by the $\gamma$-H\"older norm,
so that
%
%
%e5.9 #&#
\[
\bigl\|\bigl\langle\phi,\CZ_{V_i}(\cdot)\bigr\rangle\bigr\|
_{\infty} \le M
\bigl\|\bigl\langle\phi,\CZ_{V_i}(\cdot)\bigr\rangle\bigr\|
_{L^2[0,T]}^{{2
\gamma}/ { (2\gamma+1)}} \bigl\|\bigl\langle\phi,\CZ_{V_i}(\cdot
)\bigr\rangle
\bigr\|_{\CC^\gamma
}^{{1}/ { (2\gamma+1)}}.
\]

Since the vector fields $V_i$ have derivatives with at most polynomial
growth by assumption,
we obtain immediately from Lemma~\ref{lemphiYlem} the bound
%
%
%e5.10 #&#
\begin{equation}
\label{eqnsimplej0gib} \bigl\|\bigl\langle\phi, \CZ_{V_i}(\cdot
)\bigr\rangle
\bigr\|_{\CC^\gamma} \le M \CR_z^a
\end{equation}
for some exponent $a$. Combining this with (\ref{eqninterpref}),
the claim follows at once.
\end{pf}
The next lemma involves an iterative argument (similar in spirit to
\cite{KSAMII,Norr86,BaudHair07}) to show that a similar bound holds
with $V_j$ replaced by any
vector field obtained by taking finitely many Lie brackets between the $V_j$'s.

%
%le6 #&#
\begin{lemma}\label{lemhormit2}
Fix a final time $T>0$. Under Assumption~\ref{assgrowth}, for every
$i \ge1$, there exist constants $c_i, a_i > 0$ such that
%
%
%e5.11 #&#
\[
\bigl\|\bigl\langle\phi, \CZ_U(\cdot)\bigr\rangle\bigr\|
_{\infty} \le
M \CR_z^{c_i} \bigl|\langle\phi, C_T \phi
\rangle\bigr|^{a_i}
\]
for every $U \in\mathcal{V}_i$, every $\phi\in\R^n$ such that $\|
\phi\| = 1$, every initial condition~$z$, every $(X,\bbX) \in\CD
^\gamma_g([0,T],\R^d)$ and the constant $M>0$ is independent of
$X,\phi,z$.
\end{lemma}

\begin{pf}
The proof goes by induction over $i$. We already know from Lemma~\ref
{lemhormit1} that the statement
holds for $i = 1$.
Assume now that it holds for some $i \ge1$, and let us show that it
holds for
$i+1$. For any $t \leq T$ and $U \in\mathcal{V}_i$, a simple
application of the chain rule [which holds since $(X,\bbX)$ is assumed
to be a geometric rough path] yields
%
%
%e5.12 #&#
\begin{equation}
\bigl\langle\phi, \CZ_U(t) \bigr\rangle= \int_{0}^t
\bigl\langle\phi, \CZ_{[V_0,U]}(s)\bigr\rangle\,ds +\sum
_{j=1}^d\int_{0}^t
\bigl\langle\phi, \CZ_{[V_j,U]}(s)\bigr\rangle\,d{X}_s^j
, \label{eqnLiebrackcal}
\end{equation}
where the second integral is a rough integral as in Theorem~\ref{thmGubi}.

First we derive a priori bounds on the two integrands of (\ref
{eqnLiebrackcal})
and then apply Theorem~\ref{thmNlemmaRP}.
It follows from Lemma~\ref{lemphiYlem} and Assumption~\ref
{assgrowth} that
%
%
%e5.13 #&#
\begin{equation}
\label{eqnapbd1} \bigl\|\bigl\langle\phi,\CZ_{[V_j,U]}(\cdot
)\bigr\rangle
\bigr\|_{X,\gamma} \le M \CR_z^a ,\qquad j = 0,\ldots,d ,
\end{equation}
so that a similar bound holds on $\|\langle\phi,\CZ_{[V_j,U]}(\cdot
)\rangle\|_{\gamma}$.\vadjust{\goodbreak}

By the induction hypothesis, for every $U \in\mathcal{V}_i$ we have
the bound  $\|\langle\phi,\break \CZ_U(\cdot)\rangle\|_{\infty} \le M
\CR_z^{c_i} |\langle\phi, C_T \phi\rangle|^{a_i}$ for
some constants $a_i, c_i$.
Applying Theorem~\ref{thmNlemmaRP} to (\ref{eqnLiebrackcal}) and
using the a priori bound
(\ref{eqnapbd1}), we conclude that there exist constants $\alpha
_{i+1}, c_{i+1}$ such that
\[
\bigl\|\bigl\langle\phi, \CZ_{[U,V_\ell]}(\cdot)\bigr\rangle
\bigr\|_{\infty} \le
M \CR_z^{c_{i+1}} \bigl|\langle\phi, C_T \phi
\rangle\bigr|^{\alpha_{i+1}}
\]
for $\ell= 0,\ldots,d$. Since $\CV_{i+1}$ contains precisely the
vector fields $[U,V_\ell]$,
this concludes the proof.
\end{pf}

%Now we iterate the previous argument.
Now we combine the above two lemmas and H\"ormander's hypothesis,
Assumption~\ref{assHormander}, to
obtain lower bounds on the smallest eigenvalue of $C_T$.

%
%pr3 #&#
\begin{proposition}\label{propmallmatest}
Assume that Assumptions~\ref{assgrowth} and~\ref{assHormander}
hold. Fix $T>0$, and let the matrix $C_T$ and the quantity $\CR$ be as
defined in (\ref{eqnmallmat}) and (\ref{eqncalR}), respectively.
Then there exists a constant $c > 0$ such that the bound
%
%
%e5.14 #&#
\begin{equation}
\label{eboundMall} \inf_{\|\phi\| = 1} \bigl|\langle\phi, C_T
\phi
\rangle\bigr| > M \CR_z^{-c}
\end{equation}
holds uniformly over every driving path $(X,\bbX) \in\CD^\gamma_g$
and every initial condition $z$.
The constant $M >0$ is independent of $X,\phi,z$.
\end{proposition}

%
%re11 #&#
\begin{remark}
We emphasize again that (\ref{eboundMall}) yields a lower bound on the
eigenvalues of $C_T$ that is not probabilistic in nature. All the probabilistic
cancellations that take place in the classical probabilistic proofs of
H\"ormander's
theorem are ``hidden'' in the strict positivity of $L_\theta(X)$ and
the boundedness
of $\|(X,\bbX)\|_{\g}$.
\end{remark}

\begin{pf*}{Proof of Proposition~\ref{propmallmatest}}
Let $N \in\mathbb{N}$ be such that
%
%
%e5.15 #&#
\begin{equation}
\label{eboundbelow} K \eqdef\inf_{|z| \le\CR} \inf_{\|\phi\| =
1} \sum
_{U \in\bar
\CV_N} \bigl|\bigl\langle\phi, U(z)\bigr
\rangle\bigr|^2 > 0 .
\end{equation}
The existence of such an $N$ follows from Assumption~\ref
{assHormander} and the smoothness of the vector fields $V$.

Note now that, considering the right-hand side at time $0$, we see that
%
%
%e5.16 #&#
\[
\bigl|\bigl\langle\phi, U(z_0)\bigr\rangle\bigr|^2 \le\bigl\|
\bigl
\langle\phi, \CZ_U(\cdot)\bigr\rangle\bigr\|_{\infty}^2
.
\]
From Lemmas~\ref{lemhormit1} and~\ref{lemhormit2}, there then exist
constants $c_N, \alpha_N$
such that
%
%
%e5.17 #&#
\begin{eqnarray}
K \le\inf_{\|\phi\| = 1} \sup_{U \in\bar{\mathcal{V}}_N} \bigl
\| \bigl\langle\phi,
\CZ_U(\cdot)\bigr\rangle\bigr\|_{\infty}^2 \le M
\CR_z^{c_N} \inf_{\|\phi\| = 1} \bigl|\langle\phi,
C_T \phi\rangle\bigr|^{\alpha_N} , \label{eqnnorrisfinalest}
\end{eqnarray}
which is precisely the required bound.
\end{pf*}

Now let $\mathcal{M}_T$ be the Malliavin matrix of the map $W \mapsto Z_T^z$
where $W$ is the underlying Wiener process from representation (\ref{erepr}).
Then we have the following pathwise bound on $\mathcal{M}_T$:

%
%th5.1 #&#
\begin{theorem} \label{ThmorigMallmat}
Under the assumptions of Proposition~\ref{propmallmatest}, there
exists a constant $c_1 > 0$ such that the bound
%
%
%e5.18 #&#
\begin{equation}
\label{eboundMallorg} \inf_{\|\phi\| = 1} \bigl|\langle\phi, \mathcal{M}_T
\phi\rangle\bigr| > M \CR_z^{-c_1}
\end{equation}
holds for \textit{every} driving path $(X,\bbX) \in\CD^\gamma$, every
initial condition $z$, and the constant $M>0$ is independent of $X,\phi,z$.
\end{theorem}

\begin{pf}
By virtue of (\ref{eqnDVsDv}), we have the identity
%
%
%e5.19 #&#
\begin{equation}
\label{eqnMallCTreln} \bigl|\langle\phi, \mathcal{M}_T \phi
\rangle\bigr| = \bigl\|
\bigl(\CD^{{1}/{ 2}
- H}\bigr)^* \A^*_T J^*_{0,T} \phi
\bigr\|^2_{L_2[0,T]},
\end{equation}
where $(\CD^{{1}/{ 2}-H})^*$ is the $L_2[0,T]$ adjoint of the
operator $\CD^{{1}/{ 2}-H}$ defined in
(\ref{efrac}). Notice that $\CI^{{1}/{ 2}-H}\dvtx L_2[0,T] \mapsto
L_2[0,T]$ is a bounded operator,
and since $\CI^{{1}/{ 2}-H}$ and $\CD^{{1}/{ 2}-H}$ are
inverses of each other, we conclude
that operator $(\CD^{{1}/{ 2}-H})^*$ has a bounded inverse in
$L_2[0,T]$. Thus
%
%
%e5.20 #&#
\begin{eqnarray*}
\bigl\| \bigl(\CD^{{1}/{ 2} - H}\bigr)^* \A^*_T J^*_{0,T} \phi
\bigr\|_{L_2[0,T]} &\geq& M \bigl\|\A^*_T J^*_{0,T} \phi
\bigr\|^2_{L_2[0,T]}\\
& = &M\bigl\langle J^*_{0,T}\phi,
C_T J^*_{0,T}\phi\bigr\rangle,
\end{eqnarray*}
which, from Proposition~\ref{propmallmatest}, is bounded from below by
%
%
%e5.21 #&#
\[
M \CR_z^{-c} \bigl\|J^*_{0,T}\phi\bigr\|^2
\ge M \CR_z^{-c_1}\|\phi\|^2 ,
\]
where the last bound is a consequence of the fact that $\|
J^{-1}_{0,T}\| \leq M \CR_z$.
\end{pf}

%s5.2 #&#
\subsection{Probabilistic bounds and smoothness of laws}

Recall from (\ref{erepr}) that the ``future'' evolution of
the fBm conditional on the past $w_{-}$ may be expressed as
%
%
%e5.22 #&#
\[
w_+ = \CG w_{-} + \alpha_H \CD^{{1}/{ 2}-H} W ,
\]
where $\CG w_{-}$ is the conditional expectation with the operator $\CG
$ given by (\ref{eqnopera}).
As in the previous section, we will mostly be interested in the
situation when $w_-$ is fixed, and the conditional
law of the solution is considered. If $H = 1/2$, then all the
statements are simple since in this case $\CG= 0$ and $\CD^0$ is the
identity operator.

One problem is that it is in general quite difficult to obtain moment
bounds on the Jacobian (and its inverse)
for equations of the type (\ref{eqnSDErpi}) when the driving noise is
only $\gamma$-H\"older for some $\gamma> \frac{1}{ 3}$ (rather than
$\gamma> \frac{1}{ 2}$). The best bounds obtained in
\cite{FrizVict10} rule out a downright explosion of the Jacobian,
but only yield logarithmic moments in general.
The very recent article~\cite{CLL11} obtains such moment bounds, but
under boundedness conditions that are stronger
than Assumption~\ref{assgrowth}. See also~\cite{PeterExp} for a
related result.
We therefore state the moment bounds on the solution and its Jacobian
as an additional assumption.
We will use $\tE$ and $\tilde{\mathbb P}$ to denote the expectation
and probability, respectively, conditioned on the past of the driving
noise $w_{-}$.

%
%as3 #&#
\begin{assumption} \label{assjcfmom}
There exists an exponent $\zeta< 2$ and a seminorm $|\!|\!|\cdot|\!|\!
|$ on
$\CC(\R_-,\R^d)$ such that $|\!|\!|w_-|\!|\!|$ is
almost surely finite and such that, for every $R>0$ and every $p\ge1$,
the bound
%
%
%e5.23 #&#
\begin{equation}
\label{eqncondJbd} \tE\bigl\|\bigl(\CJ^z,{\CJ^z}'
\bigr)\bigr\|^p_{X,\gamma} \le M \exp\bigl(M |\!|\!|w_{-}
|\!|\!|^\zeta\bigr) ,
\end{equation}
holds for some constant $M$ independent of $X$, uniformly over all
initial conditions with $|z| \le R$.
\end{assumption}
%
%
%re12 #&#
\begin{remark} \label{remuncondJbd}
Combining (\ref{eqncondJbd}) with Fernique's theorem immediately
yields the unconditioned bound
%
%
%e5.24 #&#
\begin{equation}
\label{eqnuncondJbd} \mathbb{E} \bigl\|\bigl(\CJ^z,{\CJ^z}'
\bigr)\bigr\|^p_{X,\gamma} \leq M
\end{equation}
for any $p \geq1$.
\end{remark}
Now we combine the results above with the results from the previous
section to obtain probabilistic bounds on the inverse of the Malliavin
matrix, under the additional hypothesis that
Assumption~\ref{assjcfmom} holds. %Therefore in the following
%sections we only use the pathwise bounds \eqref{eqnpfrizjac}.
%Our bounds will be purely pathwise bounds, so the fact that the
%$B^i$'s are sample paths
%of fractional Brownian motion is irrelevant, except to get a bound on
%their H\"older regularity.

%All that we will assume is that the driving process is $\gamma$-H
%The main result in this section are a priori bounds not only on
%with respect to the initial condition. This will be a slight
%generalisation of the results in~\cite{HuNual07}, which required the
%drift term $V_0$ to be bounded.

%One crucial ingredient in order to obtain the bounds on the Malliavin
%matrix required to show the strong Feller property
%is control on the moments of both the solution and its Jacobian. These
%bounds
%do not quite follow from standard results
%since most of them require that $V_0$ is also bounded.
%However we cannot assume this since we need condition \eref{e:dissip}
%for showing ergodicity.

%Set $Z_t = \Phi(z_0,t)$
%where $\Phi$ Define the Jacobian
%
%
%pr4 #&#
\begin{proposition}\label{propmallmatestprob}
Let (\ref{eqnSDErpi}) be such that Assumptions~\ref{assgrowth},
\ref
{assHormander} and~\ref{assjcfmom} are satisfied.
Fix $T>0$, and let $\mathcal{M}_T$ be the Malliavin matrix as in (\ref
{eboundMallorg}).

Then, there exists a norm $|\!|\!|\cdot|\!|\!|$ such that $|\!|\!
|w_-|\!|\!| < \infty$
almost surely and,
for any $R>0$ and any $p \geq1$, there exists a constant $M$ such that
the bound
%
%
%e5.25 #&#
\begin{eqnarray}
\label{eqninvmcond} \td{\bbP} \Bigl(\inf_{\|\phi\| =1}\langle
\phi,
\mathcal{M}_T \phi\rangle\leq\eps\Bigr) \leq M e^{M |\!|\!
|w_{-}|\!|\!|^\zeta}
\eps^p ,
\end{eqnarray}
holds for all $\eps\in(0,1]$ and all initial conditions $z$ with $|z|
\le R$.
Here, the constant~$\zeta$ is as in (\ref{eqncondJbd}).

Similarly, the unconditional bound
%
%
%e5.26 #&#
\begin{eqnarray}
\label{eqninvuncond} {\bbP} \Bigl(\inf_{\|\phi\| =1}\langle\phi,
\mathcal{M}_T \phi\rangle\leq\eps\Bigr) \leq M \eps^p
\end{eqnarray}
holds.
\end{proposition}

\begin{pf}
From Theorem~\ref{ThmorigMallmat} we deduce that for small enough
$\eps$,
%
%
%e5.27 #&#
\[
\td{\bbP} \Bigl(\inf_{\|\phi\| = 1} \bigl|\langle\phi, \mathcal{M}_T
\phi\rangle\bigr| \leq\eps\Bigr) \leq\td{\bbP} \bigl( \CR\geq\eps^{-c_1}
\bigr)
\]
for some constant $c_1 >0$. By Markov's inequality, for any $p \geq1$,
this expression is bounded by
%
%
%e5.28 #&#
\[
M\eps^{pc_1} \td{\mathbb{E}} \CR^p .
\]
Now for any $p \geq1$, from Lemma~\ref{lemboundRoughfBm} and
Assumption~\ref{assjcfmom} it follows that
\[
\td{\mathbb{E}} \bigl( L_\theta(X)^{-p} + \bigl\|\bigl(
\CJ^z, {\CJ^z}'\bigr)\bigr\|^p_{X,\gamma}
\bigr) \le M e^{M |\!|\!|w_{-}|\!|\!|^\zeta} .
\]
Furthermore, it follows from (\ref{erepr}) that
$ \tE\|(X,\bbX)\|^p_{\g} \le M e^{M \|\CG w_{-}\|_\g^\zeta}$,
thus proving claim (\ref{eqninvmcond}).
The second claim then follows from Fernique's theorem.\vadjust{\goodbreak}
\end{pf}
As an immediate corollary, we obtain that the Malliavin matrix has all moments:

%
%co2 #&#
\begin{corollary} \label{cormalinvallmom}
Under the assumptions of Proposition~\ref{propmallmatestprob}, the
matrix $\mathcal{M}_T$ is almost surely invertible and, for any $R>0$
and any $p \geq1$,
\[
\td{\mathbb{E}} \bigl( \bigl\|\mathcal{M}_T^{-p}\bigr\| \bigr) \le M
e^{M |\!|\!|w_{-}|\!|\!|
^\zeta} ,
\]
uniformly over all initial conditions $z$ of (\ref{eqnSDErpi}) such
that $|z| \le R$.
\end{corollary}

As a consequence of Proposition~\ref{propmallmatestprob}, we obtain
the smoothness of the laws of $Z_t$ conditioned on an instance of the
past $w_{-}$:
%
%
%th5.2 #&#
\begin{theorem}\label{theosmooth}
Let (\ref{eqnSDErpi}) be such that Assumptions~\ref{assgrowth},
\ref
{assHormander} and~\ref{assjcfmom} are satisfied.

Then, for every realization of the past $w_{-}$ with $|\!|\!|w_-|\!|\!|
< \infty$,
every initial condition $z$ and every $t>0$, the conditional
distribution of $Z_t^z$
has a smooth density $p(x;z,w_-)$ with respect to Lebesgue measure.

Furthermore, for every multiindex $\alpha$, the derivative $\partial
_x^\alpha p(x;z,w_-)$
has finite moments of all orders, so that the unconditioned
distribution $p(x;z) = \E p(x;z,w_-)$ of $Z_t^z$
also has moments of all orders.
\end{theorem}

%
%re13 #&#
\begin{remark}
The norm $|\!|\!|\cdot|\!|\!|$ appearing in the statement is the same
as the one
appearing in Assumption~\ref{assjcfmom}.
\end{remark}

\begin{pf*}{Proof of Theorem~\ref{theosmooth}}
Combining Theorem~\ref{lemMallbdrSDE} with Assumption~\ref
{assjcfmom}, we see that
the random variable $Z_t^z$ belongs to the space $\D^\infty$. The
claim then immediately follows from the fact that
the Malliavin matrix has inverse moments of all orders~\cite{Nual06}.

The claim about the moments of the density follows from the fact that,
by (\ref{eqncondJbd}),
$\|\CM_T^{-1}\|$ and all Malliavin derivatives of $Z_t^z$ also have
unconditional
moments of all orders.
\end{pf*}

%s5.3 #&#
\subsection{A cutoff argument}\label{secapproxDens}

While~\cite{CLL11} provides a large collection of examples for which
Assumption~\ref{assjcfmom} holds,
this condition is not always easy to check. In this section, we
therefore provide a cutoff argument that allows us
to still show the existence of a density for the law of the solutions
to (\ref{eqnSDErpi}) under H\"ormander's condition,
without assuming that Assumption~\ref{assjcfmom} holds. Actually,
we show slightly more than the mere existence of a density; namely, we
show that the density can be
approximated from below by a sequence of smooth densities. More
precisely, the main result of this section is the following:

%
%th5.3 #&#
\begin{theorem} \label{thmmallcutoffden}
Assume that Assumptions~\ref{assgrowth} and~\ref{assHormander}
hold, and
denote by $\mu_t$ the conditional law of the solution to (\ref
{eqnSDErpi}) at time $t>0$, with fixed initial
condition $z \in\R^n$.\vadjust{\goodbreak}

Then,
there exists a sequence of increasing positive measures $\mu_t^n$ with
$\CC^\infty$ densities $\rho_t^n$ such that
$\lim_{n \rightarrow\infty} \mu_t^n(A) = \mu_t(A)$ for every Borel
set $A$.
In particular, $\mu_t$ has a density $\rho_t$ with respect to
Lebesgue measure and $\lim_{n \rightarrow\infty} \rho_t^n(x) = \rho_t(x)$
for Lebesgue-almost every $x$.
\end{theorem}

%
%re14 #&#
\begin{remark}
The statement that $\rho_t$ can be approximated from below by smooth
functions is strictly stronger than
just $\rho_t \in L^1$, which was already obtained in~\cite{CassFriz}.
An example of a density function that cannot be approximated in this
way would
be the characteristic function of a Cantor set with positive Lebesgue measure.
\end{remark}

\begin{pf*}{Proof of Theorem~\ref{thmmallcutoffden}}
The idea is to perform the following cutoff argument. For $\beta> 0$
real, $T \geq t$ and $q \ge2$ an even integer,
we define the function
%
%
%e5.29 #&#
\begin{equation}
\label{eqnCaplam} \Lambda_{\beta,q,T}(X,\bbX) = \int_0^T
\int_0^t \frac{|\delta
X_{s,t}|^{2q} + \|\tilde\bbX_{st}\|^{q} }{ |t-s|^{2\beta q}} \,ds \,
dt ,
\end{equation}
where we denote by $\tilde\bbX$ the antisymmetric part of $\bbX$.
This function has the following desirable properties:
\begin{longlist}[(1)]
\item[(1)] From the scaling of the covariance function for fractional
Brownian motion and the equivalence of moments for
Gaussian measures, we conclude that if $(X,\bbX)$ is fractional
Brownian motion with Hurst parameter $H$, then
$\Lambda_{\beta,q,T}(X,\bbX)$ has finite (conditional) moments of
all orders, provided that $\beta< H$.
\item[(2)] For every $\gamma\in(0,H)$ and every $\beta\in(\gamma
,H)$, there exists $q>0$ and $M>0$ such that
%
%
%e5.30 #&#
\begin{equation}
\label{eboundNorm} \|X\|_\gamma^2 + \|\bbX\|_{2\gamma}
\le M \Lambda_{\beta
,q,T}^{1/q}(X,\bbX) .
\end{equation}
A proof of this fact can be found in~\cite{FrizVict10}, page 149.
Note that since we assume $(X,\bbX)$ to be
geometric, the symmetric part of $\bbX_{s,t}$ is given by $\delta
X_{s,t} \otimes\delta X_{s,t}$, so that
it is indeed sufficient to control the increments of $X$ and the
antisymmetric part of $\bbX$.
\item[(3)] For every $(X,\bbX) \in\CD^\gamma$, the map
%
%
%e5.31 #&#
\[
\CH_{H,+} \ni h \mapsto\Lambda_{\beta,q,T} \bigl(
\tau_h(X,\bbX) \bigr) ,
\]
is Fr\'echet differentiable to all orders~\cite{FrizVict10}, where
$\tau_h$ is the ``translation map'' as defined in
(\ref{eqnmaptauh}) below. In particular, the map $w_+ \mapsto
\Lambda_{\beta,q,T} (X(w_+),\bbX(w_+) )$
belongs to the space $\D^\infty$ of random variables that are
Malliavin differentiable of all orders with all Malliavin derivatives
having moments of all orders. The precise statement of this fact is
given in Proposition~\ref{propderLambda} in
the \hyperref[app]{Appendix}.
\end{longlist}
The proof is now straightforward. First of all, we let $\gamma< H$ be
as in the previous sections,
let $\beta\in(\gamma, H)$ and fix $q$ large enough so that (\ref
{eboundNorm}) holds. We also let $\chi\colon\R_+ \to\R_+$
be a $\CC^\infty$ nonincreasing cut-off function so\vadjust{\goodbreak} that $\chi
(\lambda) = 1$ for $\lambda\le1$ and $\chi(\lambda) = 0$ for
$\lambda\ge2$. With these definitions at hand, we set
%
%
%e5.32 #&#
\begin{equation}
\label{eqncutofffn} \Psi_n(w_+) \eqdef\chi\bigl(n^{-1}
\Lambda_{\beta,q,T} \bigl(X(w_+),\bbX(w_+) \bigr) \bigr) .
\end{equation}
Fix furthermore $z \in\R^n$, and as before denote by $\Phi_t(z,w_+)$
the It\^o map, so that
$\mu_t = \Phi_t^* \P$. We then set $\mu_t^n = \Phi_t^*(\Psi_n \P
)$. In other words, we have the identity
%
%
%e5.33 #&#
\[
\mu_t^n(A) = \int_{\Phi_t^{-1}(A)}
\Psi_n(w_+) \P(dw_+) ,
\]
valid for every measurable set $A \subset\R^n$. Since $\Lambda
_{\beta,q,T}$ is almost surely finite,
we clearly have $\mu_t^n(A) \nearrow\mu_t(A)$ for every measurable
set $A$, so that the claim follows if we can show that
every $\mu_t^n$ has a smooth density. This in turn follows by
Malliavin's lemma~\cite{Nual06} if we are able to show that,
for every bounded open set $K \subset\R^n$ and every multiindex
$\alpha$, there exists a constant $M$ such that
the bound
%
%
%e5.34 #&#
\[
\tE\bigl(D^\alpha G \bigl(\Phi_t(z,w_+) \bigr)
\Psi_n(w_+) \bigr) \le M(w_-) \sup_{x \in K} \bigl|G(x)\bigr|
\]
holds uniformly over all test functions $G\colon\R^n \to\R$ that
are $\CC^\infty$ and supported in~$K$.

Let $\alpha= (\alpha_1,\alpha_2,\ldots, \alpha_k), \alpha_i \in\{
1,2,\ldots, n\}$. Using the chain rule
and the integration by parts formula from Malliavin calculus, we have
the identity
%
%
%e5.35 #&#
\begin{eqnarray}
\label{eqnskibp1} &&\tE\bigl(D^\alpha G \bigl(\Phi_t(z,w_+)
\bigr) \Psi_n(w_+) \bigr)
\nonumber
\\[-8pt]
\\[-8pt]
\nonumber
&&\qquad= \tE\bigl(G \bigl(\Phi_t(z,w_+)
\bigr) H_\alpha\bigl(\Phi_t(z,w_+),\Psi_n(w_+)
\bigr) \bigr) ,
\end{eqnarray}
where the random variables $H_\alpha$ are defined as follows. For
$\alpha= \varnothing$, the empty multiindex,
we set $H_{\varnothing} = \Psi_n$.
Furthermore, given a random variable $G$ and an index $\alpha_1$, we set
%
%
%e5.36 #&#
\begin{equation}
\CH_{\alpha_1}(G) = \D^* \Biggl(G \sum_{j=1}^n
\bigl(\mathcal{M}_t^{-1}\bigr)_{\alpha_1j}
\D^j_\cdot\Phi_t(z,w_+) \Biggr) \label
{eqnHfuns}
\end{equation}
with these definitions at hand, and it is straightforward to see that,
for $\alpha= (\alpha_1,\ldots,\alpha_k)$, we have
%
%
%e5.37 #&#
\[
H_\alpha=\CH_{\alpha_1} (H_{(\alpha_2,\ldots,\alpha
_{k})} ) .
\]
Fortunately, all of these expressions can be controlled in the
following way.
Define the set
%
%
%e5.38 #&#
\begin{equation}
\label{eqnsetsn} \mathcal{S}_n \eqdef\bigl\{w\dvtx
\Lambda_{\beta,q,T} \bigl(X(w_+),\bbX(w_+)\bigr) \le2n\bigr\} .
\end{equation}
It then follows from the local property of the Skorokhod integral \cite
{Nual06}, Proposition~1.3.15, that
$H_\alpha(w_+) = 0$ for $w_+ \notin\CS_n$. As a consequence, we
also have the identity
%
%
%e5.39 #&#
\[
H_\alpha=\tilde\CH_{\alpha_1} (H_{(\alpha_2,\ldots,\alpha
_{k})} ) ,\vadjust{\goodbreak}
\]
where
%
%
%e5.40 #&#
\[
\tilde\CH_{\alpha_1}(G) = \D^* \Biggl( G \sum_{j=1}^n
\bigl(\Psi_{2n}\bigl(\mathcal{M}_t^{-1}
\bigr)_{\alpha_1j} \bigr) \D^j_\cdot\bigl(
\Psi_{2n}\Phi_t(z,w_+) \bigr) \Biggr) .
\]
Note now that, by Corollary~\ref{corFublem2}, Theorems~\ref
{ThmorigMallmat} and~\ref{lemMallbdrSDE},
both $\Psi_{2n}(\mathcal{M}_t^{-1})_{\alpha_1j}$ and $\Psi_{2n}\Phi
_t(z,w_+)$ belong to the stochastic Sobolev
space $\D^{\ell,p}$ for
every $\ell,p>1$, uniformly over every $w_-$ such that $|\!|\!|w_-|\!
|\!|_\gamma
\le R$ for any $R>0$, where $|\!|\!|w_-|\!|\!|_\gamma$
was defined in (\ref{noisespace}).

As a consequence, for $\alpha= (\alpha_1,\ldots,\alpha_k)$ and
$\ell> 0$, we have the bound
%
%
%e5.41 #&#
\[
\tilde\E\bigl\|\D^{(\ell)} H_\alpha\bigr\|^p \le\mathrm{K}\bigl(
|\!|\!|w_-|\!|\!|\bigr) \sum_{m \le\ell+1} \bigl(\tilde\E\bigl\|
\D^{(m)} H_{(\alpha_2,\ldots
,\alpha_k)}\bigr\|^{2p} \bigr)^{{1}/{ 2}} ,
\]
where we denote by $\D^{(k)}$ the $k$th iterated Malliavin derivative
and by $\|\cdot\|$ the $L^2$-norm.
Since $H_\varnothing$ also belongs to $\D^{\ell,p}$ for
every $\ell,p>1$ by Corollary~\ref{corFublem2}, the claim then follows.
\end{pf*}

%s6 #&#
\section{Ergodicity of SDEs driven by fBm}
\label{secergodic}

The aim of this section is to use the preceding results in order to
obtain ergodicity results for
stochastic differential equations driven by fractional Brownian motion.
In order to do this, we
make use of the abstract framework introduced in~\cite{Hair05} and
further refined in~\cite{HairOhas07,Hair09}.
This allows us to introduce a notion of a ``strong Feller property''
for a large class of equations driven
by nonwhite noise, together with a corresponding version of the
Doob--Khasminskii theorem, stating that the strong Feller
property, combined with a form of topological irreducibility and a
quasi-Markovian property,
is sufficient to deduce the uniqueness of an ``invariant measure''
in a suitable sense.

In order to use this framework, we view solutions to (\ref{eqnSDErpi})
as a \textit{discrete-time}
Markov process on a space of the type $\CW\times\R^n$, where $\CW$
contains all the information
about the driving noise $X$ required to solve (\ref{eqnSDErpi}) over a
(fixed) time interval~$1$ and to
predict the law of its future evolution. In our case, it is natural to
choose $\CW$ to be of the form
\[
\CW= \CW_- \oplus\CW_+ ,
\]
where $\CW_-$ contains the ``past'' of the driving noise up to time
$0$, and $\CW_+$ contains the noise
between times $0$ and $1$. The reason for splitting our space
explicitly into two parts is that in order to
be able to give a meaning to solutions to (\ref{eqnSDErpi}),
we consider the driving noise as a rough path; that is, we choose $\CW
_+ = \CD_g^\gamma([0,1],\R^d)$
for some $\gamma\in(\frac{1}{ 3},H)$. (Recall that $\CD_g^\gamma$
is the closure of the set of lifts of smooth functions
in $\CD^\gamma$.)

On the other hand, in order to recover the conditional law of
fractional Brownian motion
given its past, iterated integrals are not needed, and it is sufficient
to retain information about the path itself. Therefore,
it makes sense to choose $\CW_-$ in a way \vadjust{\goodbreak}similar to~\cite{Hair05};
namely,
we choose $\CW_- = \CW_\gamma$ for some $\gamma< H$, where $\CW
_\gamma$
was defined in (\ref{noisespace}).
Denote as before by $\mathbb{P}_-$ the measure on $\CW_\gamma$ such
that the canonical process is a
fractional Brownian motion with Hurst parameter $H$ under~$\bbP_-$.

For any given $w_{-} \in\CW_-$, we now construct a measure $\hat\CP
(w_{-},\cdot)$ on $\CW_+$ as
the law of a two-sided fractional Brownian motion, conditioned on its
past $w_{-}$,
and enhanced with the corresponding ``area process.'' To this end, let
us first denote by
$\tilde\P_+$ the law of the stochastic process $\{X_t\}_{t \in
[0,1]}$, given by
%
%
%e6.1 #&#
\begin{equation}
\label{ereprFBM} X_t = \alpha_H \int
_{0}^t (t-r)^{H-{1}/{ 2}} \,dW_r ,
\end{equation}
where $W$ is a standard Wiener process, and $\alpha_H$ is the constant
appearing in (\ref{erepr}). It can be checked that the covariance
of $\tilde\P_+$ satisfies the assumptions of \cite
{CoutQian,FVGauss}, so that it can be lifted in a canonical
way to a measure $\P_+$ on~$\CW_+$.

With this definition at hand, we define a Markov transition kernel
$\hat\CP$ from
$\CW_-$ to $\CW_+$ by
%
%
%e6.2 #&#
\[
\hat\CP(w_{-},\cdot) = \tau_{\CG w_{-}}^* \P_+ ,
\]
with the shift operator $\tau_{\CG w_{-}}$ as in (\ref{eqnmaptauh}).
It follows from (\ref{eboundCG}), (\ref{eboundTR1}) and (\ref
{eboundTR2}) that $\hat\CP$ is
Feller. Furthermore, it determines a measure $\P$ on $\CW= \CW_-
\times\CW_+$ in a natural way
by
%
%
%e6.3 #&#
\[
\P(dw_- \times dw_+) = \P_-(dw_-) \hat\CP(w_-,dw_+) .
\]
It follows from our construction that if we denote by $\Pi\colon\CW
\to\CC((-\infty,1],\R^d)$, the natural
map that concatenates $w_-$ with the ``path'' component of $w_+$, then
the image of $\P$ under
$\Pi$ is precisely the law of a two-sided fractional Brownian motion
with Hurst parameter $H$.
Similarly, we have a natural
shift map $\Theta\colon\CW\to\CW_-$ that consists of composing
$\Pi$ with the usual time-$1$ shift map
that maps $\CC((-\infty,1],\R^d)$ into $\CC((-\infty,0],\R
^d)\supset\CW_-$.
It follows from the definitions of $\CW_-$ and $\CW_+$ that the map
$\Theta$ is actually continuous. This construction
allows us to lift $\hat\CP$ to a Feller Markov transition kernel on
$\CW$ by
%
%
%e6.4 #&#
\[
\CP(w,\cdot) = \delta_{\Theta(w)}\otimes\hat\CP\bigl(\Theta
(w),\cdot\bigr)
.
\]
It also follows from our construction that $\P$ is invariant (and
ergodic) for $\CP$.
Indeed, the action of $\CP(w,\cdot)$ is to shift the ``path''
component of $w$ backwards by one time
unit and to then concatenate it with the canonical lift to $\CW_+$ of
a piece of two-sided fractional
Brownian motion, conditional on its past being given by $\Theta(w)$.

We now combine the noise process $(\CW,\P,\CP)$ with the solution
map for (\ref{eqnSDErpi}) in the following
way. As before let $\Phi_\cdot(z,w_+)$ denote the map that solves
(\ref{eqnSDErpi}) for a given initial
condition $z$ and a given realization $w = (w_{-},w_+)$ of the driving noise.
%Note that even though the noise $w$ belongs
%to $\CW= \CW_- \times\CW_+$, we only make use of the component in $
Since the It\^o map
is continuous on the space of rough paths with fixed H\"older
regularity~\cite{FrizVict10},\vadjust{\goodbreak} the map $\Phi$ is continuous.
We can then view the solutions to (\ref{eqnSDErpi}) as a Markov
process on $\R^n \times\CW$ with transition
probabilities given by
%
%
%e6.5 #&#
\[
\CQ(z,w;\cdot) = \Phi_z^* \CP(w,\cdot) ,
\]
where we define $\Phi_z \colon\CW\to\R^n \times\CW$ by $\Phi
_z(w) = (\Phi_1(z,w_+), w)$. In other words, we first
shift back the noise by a time interval $1$, then draw a sample from
the conditional realization of an enhanced fractional
Brownian motion on $[0,1]$, and then use this sample to solve (\ref
{eqnSDErpi}) between $0$ and $1$.

The aim of this section is to show that the Markov operator $\CQ$
admits a
unique invariant measure, modulo
a natural equivalence relation described in Section~\ref{secunique} below.
Note that while $\CQ$ is Feller (since $\Phi$ is continuous and $\CP
$ is Feller), it is certainly not strong Feller
in the usual sense. We will, however, show in Section~\ref{secunique}
that there is a natural generalization
of the strong Feller property in this context that, in a way, only
considers the part of $\CQ$ in $\R^n$.
In this generalized sense, it turns out that the invertibility of the
Malliavin matrix shown in the preceding sections
allows us to prove that $\CQ$ satisfies the strong Feller property in
this generalized sense. Combined with a
form of topological irreducibility and a ``quasi-Markov'' property,
this is then sufficient to deduce the uniqueness
of the invariant measure for $\CQ$ modulo equivalence of the induced
laws on the space of trajectories on $\R^n$.

%
%The aim of this section is to show that there exist invariant measures
%for $\CQ$, provided that
%result is the following:
%
%Assume that ......
%
%Since $\CQ$ is Feller, we can use the usual Krylov-Bogoliubov
%criterion for the existence of
%an invariant measure. We furthermore know that $\P$ is the unique
%invariant measure for $\CP$ on
%$\CW$, so that every invariant measure for $\CQ$ must project down to $

%s6.1 #&#
\subsection{General uniqueness criterion for the invariant measure}
\label{secunique}

From now on, we use the notation $\CX= \R^n$ in order to simplify
notation and to emphasize the
fact that the results do not depend on the linear structure of the space.

The aim of this section is to study the uniqueness of ``invariant
measures'' for~(\ref{eqnSDErpi}).
The question of uniqueness of the invariant measure for the SDE (\ref
{eqnSDErpi}) should not be interpreted as the question of uniqueness
of the invariant measure for the Markov operator $\CQ$ constructed in
the previous section.
This is because one might imagine that the augmented phase space $\CX
\times\CW$ contains some
``redundant'' randomness that is not necessary to describe the
stationary solutions to (\ref{eqnSDErpi}). (This would be the case,
e.g.,
if the $V_i$'s are not always linearly independent.)
One would like therefore to have a concept of uniqueness for the
invariant measure that is independent of the
particular description of the driving noise.

To this end, we introduce the Markov transition kernel $\bar\CQ$ from
$\CX\times\CW$ to $\CX^\N$ constructed in the following way.
Denote by $(z_n, w_n)$ a sample of the Markov chain with transition
probabilities $\CQ$ starting at $(z_0,w_0)$.
We then denote by $\bar\CQ(z_0,w_0; \cdot)$ the law of
$(z_1,z_2,\ldots)$.
(We do not include the
starting point, consistent with the convention that $0 \notin\mathbb{N}$.)

With this notation, we have a natural equivalence relation between
measures on $\CX\times\CW$ given by
%
%
%e6.6 #&#
\begin{equation}
\label{eequivIM} \mu\sim\nu\quad\Leftrightarrow\quad\bar\CQ\mu= \bar\CQ
\nu.\vadjust{\goodbreak}
\end{equation}
In other terms, two measures on $\CX\times\CW$ are equivalent if
they generate the same dynamics in $\CX$.
In the particular case when the process in $\CX$ is Markov, $\bar\CQ
$ is independent of $w$, and the equivalence relation
simply states that the marginals on $\CX$ should agree.
Denoting by $\|\cdot\|_\TV$ the total variation norm, this suggests
that the following is a good generalization of the
strong Feller property to our setting:

%
%de4 #&#
\begin{definition} \label{sfellerdefinition}
The solutions to (\ref{eqnSDErpi}) are said to be \textit{strong
Feller} if there
exists a jointly
continuous function $\ell\colon\CX^2 \times\CW\rightarrow\R_+$
such that
%
%
%e6.7 #&#
\begin{equation}
\label{eqnSF} \bigl\|\Q(z,w;\cdot) - \Q(y,w;\cdot) \bigr\|_\TV\le\ell
(z,y,w) ,
\end{equation}
and $\ell(z,z,w)=0$ for every $z\in\CX$ and every $w \in\CW$.
\end{definition}

We stress again that the
definition given here has essentially \textit{nothing} to do with the
strong Feller property of $\CQ$. It rather
generalizes the notion of the strong Feller property for the Markov
process associated to (\ref{eqnSDErpi})
in the case where the driving noise is white in time. See, for example,
the review article~\cite{Hair09} for more details.
%
%
%de5 #&#
\begin{definition} \label{irreducibility}
The solutions to (\ref{eqnSDErpi}) are said to be \textit
{topologically irreducible}
if, for every $z \in\CX$, $w \in\CW$ and every nonempty open set
$U\subset\CX$, one has $\CQ(z,w; U \times\CW) > 0$.
\end{definition}

%
%re15 #&#
\begin{remark}
In order to prove topological irreducibility, one usually uses some
form of the Stroock--Varadhan support theorem~\cite{SVSupport}.
A version of this theorem was shown in the present context to hold in
\cite{FrizVict10}, Theorem~15.63. This shows that,
in order to verify that (\ref{eqnSDErpi}) is topologically
irreducible, it suffices to show that, for every $x_0 \in\R^n$, the set
of points that are obtained as the solution at time $t=1$ to
%
%
%e6.8 #&#
\[
\dot x(t) = V_0\bigl(x(t)\bigr) + \sum
_{i=1}^d V_i\bigl(x(t)\bigr)
u_i(t) ,\qquad x(0) = x_0 ,
\]
with $u \in\CC^\infty([0,1], \R^d)$ is dense in $\R^n$.
\end{remark}

The following result, which is a consequence of~\cite{HairOhas07},
Theorem 3.10, is a generalization of the well-known
Doeblin--Doob--Khasminskii
criterion for the uniqueness of the invariant measure of a general
Markov chain:

%
%th6.1 #&#
\begin{theorem}\label{doob}
If the solutions to (\ref{eqnSDErpi}) are strong Feller and and
topologically irreducible,
then (\ref{eqnSDErpi}) can have at most one invariant measure, modulo
the equivalence relation (\ref{eequivIM}).
\end{theorem}

\begin{pf}
The only missing ingredient to be able to apply~\cite{HairOhas07},
Theorem~3.10, is the ``quasi-Markov'' property of the
solutions to a stochastic differential equation driven by fractional\vadjust{\goodbreak}
Brownian motion with $H \in(\frac{1}{ 3}, \frac{1}{ 2})$.
For the case $H > \frac{1}{ 2}$, this property was shown to hold in
\cite{HairOhas07}, Proposition~5.11, and in the case $H = \frac{1}{
2}$, solutions are Markovian anyway. The proof in the
case $H < \frac{1}{ 2}$ is virtually identical, so we only sketch it.
It only uses the fact that the set
$\CX= \{h \in\CC^\infty([0,1],\R^d) \dvtx h'(0) = 0\}$ has the
following properties:
\begin{longlist}[(1)]
\item[(1)] The canonical injection $\CX\hookrightarrow\CW_+$ has
dense image in $\CW_+$.
\item[(2)] The set $\CX$ belongs to the Cameron--Martin space of
$\tilde\P_+$, viewed as a measure on $\CC([0,1], \R^d)$.
\item[(3)] The set $\{\hat\CG h \dvtx h \in\CX\}$, where $\hat\CG$
is defined as
%
%
%e6.9 #&#
\[
\hat\CG h(t) = \gamma_H \int_0^1
\frac{1}{ r} g \biggl(\frac{t }{
r} \biggr) \bigl(h(1-r) - h(0) \bigr)
\,dr - \gamma_H h(0) \int_1^\infty
\frac{1}{ r} g \biggl(\frac{t }{ r} \biggr) \,dr ,
\]
belongs to the Cameron--Martin space of $\tilde\P_+$, viewed as a
measure on $\CC(\R_+,\break  \R^d)$. Indeed, because of the representation
(\ref{erepr}) and the properties of fractional integrals, it suffices
to check that $\CD^{H + {1}/{ 2}} \hat\CG h \in L^2(\R_+, \R^d)$
for $h \in\CX$. Using~(\ref{ebehaveg}), an explicit calculation
shows that $\CD^{H + {1}/{ 2}} \hat\CG h \sim t^{{1}/{2}-H}$ for
$t \ll1$ and $\CD^{H + {1}/{ 2}} \hat\CG h \sim t^{-{1}/{2}-H}$ for $t \gg1$ (see also~\cite{Hair05}, Lemma~4.3), so that
this is indeed the case.
\end{longlist}
Indeed, the first property ensures that, given any two open sets $U,V
\in\CW_+$, we can find two smaller open sets
$\bar U, \bar V \in\CW_+$, such that $\bar U \subset U$, $\bar V
\subset V$ and $\bar V = \tau_h \bar U$, with $h \in\CX$.
Furthermore, $\P_+(\bar U) > 0$, and so is $\P_+(\bar V)$ since the
topological support of $\P_+$ is all of $\CW_+$
by~\cite{FrizVict10}, Theorem~15.63.

Since $\CX$ belongs to the Cameron--Martin space of $\P_+$, this
guarantees that, for every $w_- \in\CW_-$, we
can construct a subcoupling $\hat\CP_{U,V}$ on $\CW_+^2$ between
$\hat\CP(w_-, \cdot)|_{\bar U}$ and $\hat\CP(w_-, \cdot)|_{\bar
V}$ such that
$\hat\CP_{U,V}$ charges the set of pairs $(w_+, \bar w_+)$ such that
$\bar w_+ = \tau_h w_+$.
In order to check the quasi-Markov property, it now suffices to check
that the measures $\bar\CQ(x,w;\cdot)$ and $\bar\CQ(x, \bar
w;\cdot)$
are mutually equivalent if $(w,\bar w)$ are such that their components
in $\CW_-$ are identical, and their components in $\CW_+$
satisfy $\bar w_+ = \tau_h w_+$. This in turn is precisely the content
of the third property above.
\end{pf}

The aim of the next section is to show that (\ref{eqnSDErpi}) does
indeed possess the strong Feller property, provided
that the vector fields $\{V_i\}$ satisfy H\"ormander's bracket condition.

%s6.2 #&#
\subsection{Verification of the strong Feller property}
\label{secSF}

The main result of this section is that the strong Feller property is a
consequence of H\"ormander's bracket condition.

%
%th6.2 #&#
\begin{theorem}\label{thmmainSF}
Under Assumptions~\ref{assgrowth} and~\ref{assHormander}, (\ref
{eqnSDErpi}) is strong Feller in the sense of Definition~\ref
{sfellerdefinition}.
\end{theorem}

%
%re16 #&#
\begin{remark}
The main feature that distinguishes this situation from the usual one
is that
the process is not Markov. As a consequence, our definition of the
strong Feller property
implies that, as in~\cite{HairOhas07}, we need to construct a
coupling between solutions starting from nearby points
such that, with high probability, solutions agree not only after some
fixed time (say $1$),
but also for all subsequent times. Furthermore, we will circumvent the
fact that we do not assume
a priori that the Jacobian of our solution has moments. This will be
done by a cutoff procedure
similar to~\cite{HairOhas07}.
\end{remark}

\begin{pf*}{Proof of Theorem~\ref{thmmainSF}}
Fix some arbitrary value $N>1$ and a Fr\'echet differentiable map $\psi
\colon\CX^N\to\R$, which is bounded with bounded
derivative. Denote furthermore by $R_{N}\colon\CX^\N\to\CX^N$ the
projection onto the first $N$ components, and set as before
\[
\bar\CQ\psi(z,w) \eqdef\int_{\CX^\N} \psi(R_{N} x)
\Q(z,w; dx) ,
\]
so that $\bar\CQ\psi\colon\CX\times\CW\to\R$.

The strong Feller property will follow if we show the existence of a
jointly continuous function $\ell\dvtx \CX^2 \times\CW\mapsto
\mathbb{R}_+$
such that
%
%
%e6.10 #&#
\begin{equation}
\label{eqnsffinest} \bigl|\bar\CQ\psi(z,w) - \bar\CQ\psi(y,w)\bigr| \leq
\ell(z,y,w)
\end{equation}
for all Fr\'{e}chet differentiable functions $\psi$ with bounded
derivatives such that
\[
\sup_{x \in\CX^N} \bigl|\psi(x)\bigr| \leq1,
\]
uniformly for all $N > 1$.

To this end, set $z_s = z s + y (1-s)$ for $s \in[0,1]$ and $\xi= z- y$.
Let $\Phi_{[1,T]}(z,w_+)$ denote the solution to (\ref{eqnSDErpi}),
restricted to
the interval $[1,T]$. Since
\[
\CQ\psi(z,w) = \tE\psi_T \bigl(\Phi_{[1,T]} (z, w_+) \bigr)
,
\]
where $\psi_T$ is just $\psi$, composed with the evaluation map at
integer times,
we have the identity
%
%
%e6.11 #&#
\begin{equation}
\label{eqnsfdqfi} \bar\CQ\psi(z,w) - \bar\CQ\psi(y,w) = \tE\int
_0^1 D\psi_T \bigl(
\Phi_{[1,T]} (z_s, w_+) \bigr) J_{0,\cdot}^s
\xi\,ds .
\end{equation}
Here, $J_{0,\cdot}^s$ denotes the linearization of (\ref{eqnSDErpi})
with the initial condition $z_s$.

If moment bounds for the Jacobian $J_{0,t}^s$ are available, as in
\cite{HairPill10}, we can proceed
via a stochastic control argument using a Bismut--Elworthy--Li type
formula~\cite{Xuemei} to show that
that $\CQ\psi(z,w)$ is actually differentiable in $z$. Since we do
not assume this, we will combine this with a cutoff
argument adapted from~\cite{HairOhas07}.

Recall the function $\Lambda_{\beta,q} (X,\bbX) \eqdef\Lambda
_{\beta,q,1} (X,\bbX)$ from (\ref{eqnCaplam}) with $\beta> \g
$, and set $q$ to be\vadjust{\goodbreak} an even integer such that (\ref{eboundNorm}) holds.
Similar to (\ref{eqncutofffn}), define the cutoff function
%
%
%e6.12 #&#
\begin{equation}
\label{eqnpsifn} \Psi_{R}(w_+) \eqdef\chi\biggl(\frac{1 }{ R}
\Lambda_{\beta
,q} \bigl(X(w_+),\bbX(w_+) \bigr) \biggr),\qquad R > 0 ,
\end{equation}
where $\chi\colon\R_+ \to\R_+$
is a $\CC^\infty$ decreasing function with $\chi(\lambda) = 1$ for
$\lambda\le1$ and $\chi(\lambda) = 0$ for
$\lambda\ge2$.
%It is clear that the function $\Psi_{R}$ belongs to $\mathbf{D}^{2,1}$.

From (\ref{eqnsfdqfi}) we obtain that
\begin{eqnarray*}
\bigl|\bar\CQ\psi_T (z,w) - \bar\CQ\psi_T (y,w)\bigr| &\leq&\biggl
\llvert\tE\int_0^1 \Psi_{R}(w_+)
D\psi_T \bigl(\Phi_{[1,T]} (z_s, w_+) \bigr)
J_{0,\cdot}^s\xi\,ds\biggr\rrvert
\\
&&{}+ \bigl\llvert\tE\bigl(1 - \Psi_{R}(w_+)\bigr) \psi_T
\bigl(\Phi_{[1,T]} (z, w_+) \bigr) \bigr\rrvert
\\
&&{}+ \bigl\llvert\tE\bigl(1 - \Psi_{R}(w_+)\bigr) \psi_T
\bigl(\Phi_{[1,T]} (y, w_+) \bigr) \bigr\rrvert
\\
&\eqdef & T_1 + T_2 + T_3.
\end{eqnarray*}
Since $\psi_T$ is bounded by $1$, we have the bound
%
%
%e6.13 #&#
\begin{equation}
\label{eqnT2T3} T_2 + T_3 \leq2\CP\bigl(w, \bigl\{w_+ |
\Lambda_{\beta
,q} \bigl(X(w_+),\bbX(w_+) \bigr)> 2R \bigr\} \bigr) ,
\end{equation}
which can be made arbitrarily small by choosing $R$ sufficiently large.

For tackling the term $T_1$, we now outline the stochastic control argument.
Recall the operator $\A$ from (\ref{eqnopA}).
As explained in (\ref{eqncontbas}), the Fr\'echet derivative of the
flow map with respect to the driving noise $w_+$
in the direction of $\int_0^\cdot v(s) \,ds$ is given by
\[
J_{0,T}^s \A_T v .
\]
The key idea underlying Bismut-type formulas
is to use relation (\ref{eqncontbas}) to convert the derivative of
$\Phi_T(z_s,w_+)$
with respect to its initial condition $z_s$ into a derivative with
respect to the driving noise
and to use the integration by parts formula from Malliavin calculus.

To this end, given an initial displacement $\xi\in\R^n$,
we seek for a ``control'' $v$ on the time interval $[0,1]$ that solves
the equation $\A_1 v = \xi$.
If this can be achieved,
then we extend $v$ to all of $\mathbb{R}_+$ by setting $v(s) = 0$ for
$s \ge1$ and define $\tilde v = \CI^{{1}/{ 2} - H} v$.
Note that since $v(s) = 0$ for $s > 1$, it follows from the definition
of $\CA_T$ that we have $\CA_T v = \CA_1 v = \xi$ for $T \ge1$.
If $v$ is sufficiently regular in time so that $\tilde v \in L^2(\R
_+,\R^d)$, we have the identity
%
%
%e6.14 #&#
\begin{equation}
\D_{\tilde{v}} Z^{z_s}_T = J_{0,T}^s
\xi\label{eqncontMalli}
\end{equation}
for every $T \ge1$ and therefore, by the chain rule,
%
%
%e6.15 #&#
\begin{equation}
\label{eidenMal} D\psi_T (\Phi_{[1,T]} ) J_{0,\cdot}
\xi=D\psi_T (\Phi_{[1,T]} ) \D_{\tilde{v}}
Z^{z_s}_\cdot= \D_{\tilde{v}} \bigl(\psi_T (
\Phi_{[1,T]}) \bigr) .
\end{equation}

It remains to find a control $v$ which
solves $\A_1 v = \xi$.
From Proposition~\ref{propmallmatest}, $C_1 = \A_1\A^*_1$ is
invertible, and therefore
one possible solution to the equation $\A_1 v = \xi$ is given by the
``least squares'' formula,
%
%
%e6.16 #&#
\begin{equation}\qquad
v(r) \eqdef\A^*_1 \bigl(\bigl(\A_1 \A^*_1
\bigr)^{-1}\xi\bigr) (r) = V(Z_r)^* \bigl(J^{-1}_{0,r}
\bigr)^* C_1^{-1}\xi,\qquad  r \in(0,1) . \label{eqnstochcont}
\end{equation}
Note that the control $v$ also depends on the initial condition $z_s$,
but since $|z_s| \leq|z| \vee|y|$,
all our estimates for the rest of the proof will be uniform in the
initial condition by Remark~\ref{remunifbound}.

Inserting identity (\ref{eidenMal}) into the definition of $T_1$, we obtain
%
%
%e6.17 #&#
\begin{equation}
\label{eqnfubbound} |T_1| = \biggl|\int_0^1
\tE\bigl( \Psi_{R}(w_+) \D_{\tilde
{v}} \psi_T \bigl(
\Phi_{[1,T]} (z_s, w_+) \bigr) \bigr) \,ds \biggr| .
\end{equation}
Applying the integration by parts formula from Malliavin calculus, we obtain
%
%
%e6.18 #&#
\begin{eqnarray}\label{eqnmalltwoterms}
\tE\bigl( \Psi_{R}(w_+) \D_{\tilde{v}} \psi_T (
\Phi_{[1,T]} ) \bigr) &=& \tE\bigl(\psi_T (\Phi_{[1,T]}
) \D^* \bigl(\Psi_{R}(w_+) \tilde{v} \bigr) \bigr)
\nonumber
\\[-8pt]
\\[-8pt]
\nonumber
&\leq&\bigl(\tE|\D^* \bigl(\Psi_{R}(w_+) \tilde{v} \bigr)
|^2 \bigr)^{{1}/{ 2}} ,
\end{eqnarray}
where the second inequality follows from the fact that $\psi_T$ is
bounded by $1$.
To conclude the proof, it thus suffices to show that
%
%
%e6.19 #&#
\begin{equation}
\label{ewantedbound} \tE\bigl( \bigl|\D^* (\Psi_{R}\tilde{v}
)\bigr|^2 \bigr) \le C(R,w_-, z) |\xi|^2 ,
\end{equation}
where $C$ is uniformly bounded on $|\!|\!|w_-|\!|\!|_\gamma\le M$ and
$|z| \le M$.

Since the stochastic process $\tilde{v}$ is in general not adapted to
the filtration
generated by the underlying Wiener process,
we use the following extension of It\^o's isometry~\cite{Nual06}:
%
%
%e6.20 #&#
\begin{eqnarray}
\label{eqnintegsko} \tE\bigl(\bigl |\D^* (\Psi_{R}\tilde{v}
)\bigr|^2 \bigr) &=& \tE\biggl( \int_{0}^{\infty}
\bigl|\Psi_{R} \tilde v(s)\bigr|^2 \,ds \biggr)
\nonumber\\
&&{} + \tE\int_{0}^{\infty} \int_{0}^{\infty}
\tr\bigl(\D_t \bigl(\Psi_{R} \tilde v(s)
\bigr)^T \D_s \bigl(\Psi_{R}\tilde v(t)
\bigr) \bigr) \,ds \,dt
\nonumber
\\[-8pt]
\\[-8pt]
\nonumber
&\le&\tE\|\Psi_{R} \tilde v\|^2 + \tE\bigl\|\D(
\Psi_{R}\tilde v )\bigr\|^2 \le c \bigl(\tE\|
\Psi_{R} v\|^2 + \tE\|\D\Psi_{R}v
\|^2 \bigr)
\\
&\eqdef& I_1 + I_2 .\nonumber
\end{eqnarray}
Here, $\|v\|$ denotes the $L^2$-norm of $v$ and similarly for $\D v$.
Since $\tilde v = \CI^{H - {1}/{ 2}} v$,
the second inequality is a consequence of the fact that $\CI^{H-{1}/{ 2}}$ is a bounded
operator from $L^2([0,1])$ into $L^2(\R_+)$; see Corollary~\ref
{corboundInt} below.

Bound (\ref{ewantedbound}) on $I_1$ now follows immediately from
Proposition~\ref{propmallmatest}
and Remark~\ref{remunifbound}. The bound on $I_2$ follows similarly
by also using Theorem~\ref{lemMallbdrSDE}.
The proof of Theorem~\ref{thmmainSF}
is complete.
\end{pf*}

We now show that $\CI^{{1}/{ 2}-H}$ is indeed a bounded operator
from $L^2([0,1])$ to
$L^2(\R_+)$. For this, define the operator $\tilde\CI_\alpha$ by
%
%
%e6.21 #&#
\[
(\tilde\CI_\alpha v ) (s) = \int_0^1
|s-r|^{\alpha-1} v(r) \,dr .
\]
We then have:

%
%le7 #&#
\begin{lemma}\label{lemboundInt}
For $\alpha\in(0,\frac{1}{ 2})$, there exists a constant $c$ such
that, for positive~$v$,
%
%
%e6.22 #&#
\[
\|\tilde\CI_\alpha v \|^2 \le c \|v\| \|\tilde
\CI_{2\alpha} v \| .
\]
\end{lemma}

\begin{pf}
We have the bound
\begin{eqnarray*}
\|\tilde\CI_\alpha v \|^2 &=& \int_0^1
\int_0^1 \int_0^\infty|s-r|^{\alpha-1}
|s-t|^{\alpha-1} \,ds\, v(r) v(t) \,dr \,dt
\\
&\le&\int_0^1 \int_0^1
\int_{-\infty}^\infty|s-r|^{\alpha-1}
|s-t|^{\alpha-1} \,ds\, v(r) v(t) \,dr \,dt
\\
&=& c\int_0^1 \int_0^1
|r-t|^{2\alpha-1} \,ds\, v(r) v(t) \,dr \,dt ,
\end{eqnarray*}
where the first step follows from the positivity of $v$, and the second
step follows from a simple scaling argument.
Since this is nothing but $c \langle v, \tilde\CI_{2\alpha} v\rangle
$, the
requested bound follows from the Cauchy--Schwarz
inequality.
\end{pf}

%
%co3 #&#
\begin{corollary}\label{corboundInt}
For every $\alpha\in(0,1)$, the operator $\CI^{\alpha}$ is bounded
from $L^2([0,1])$ to
$L^2(\R_+)$.
\end{corollary}

\begin{pf}
Note that
\[
\bigl|\CI^{\alpha}v(s)\bigr | \le\CI^{\alpha} |v |(s) \le\tilde
\CI^{\alpha} |v |(s) .
\]
Since $|s-r|^{\alpha-1}$ is square integrable if $\alpha> \frac{1}{
2}$, the claim follows for that
range of $\alpha$. For smaller values of $\alpha$, it is always
possible to reduce oneself to the range $(\frac{1}{ 2},1)$
by Lemma~\ref{lemboundInt}, noting also that $ \|\tilde\CI_\alpha v
\| \le c \|\tilde\CI_\beta v \|$
for $\alpha> \beta$.
\end{pf}

%s7 #&#
\section{Examples} \label{secexamples}

In this section, we collect a few examples to which our main results apply.

%s7.1 #&#
\subsection{Hypoelliptic Ornstein--Uhlenbeck process}

Consider the process $x_t$ given by
%
%
%e7.1 #&#
\begin{equation}
\label{eOU} dx = Ax \,dt + C \,dB_H(t) ,
\end{equation}
where $x \in\R^n$, $B_H$ is an $m$-dimensional
fractional Brownian motion with Hurst parameter $H > \frac{1}{ 3}$,
$A$ is an $n \times n$ matrix with $A + A^T < 0$ and $C$ is an
$n\times m$ matrix. It is well known that (\ref{eOU}) satisfies H\"ormander's
condition if and only if there exists $k>0$ such that the matrix
$(C,AC,\ldots,A^k C)$ has rank $n$.

Since the Jacobian is given by $J_{s,t} = \exp(A(t-s))$ and therefore
has moments of all orders, we conclude that, for any initial condition
$x_0$ and
for any sequence of times $t_1,\ldots,t_k$,
the joint distribution of $(x_{t_1},\ldots,x_{t_k})$ has a smooth
density with respect to Lebesgue measure.
Since this distribution is Gaussian, one could have verified directly
that its covariance is nondegenerate,
but this would have been a rather lengthy calculation.

%s7.2 #&#
\subsection{Linear equations/L\'evy area}

Let $B$ be a $d$-dimensional fractional Brownian motion and consider
equations in $\R^m$ of the type
%
%
%e7.2 #&#
\begin{equation}
\label{elinear} dX_i = (A_{ijk} X_j +
C_{ik} ) \circ\,dB_k(t) ,
\end{equation}
where we use Einstein's convention of summation over repeated indices.
In this case, the derivative of the solution with respect to its
initial condition
in a direction $\eta\in\R^m$ is nothing but the solution to
%
%
%e7.3 #&#
\begin{equation}
\label{elinearised} dJ_i = A_{ijk} J_j \circ
dB_k(t) ,
\end{equation}
with initial condition $J(0) = \eta$. Similar formulas hold for higher
order derivatives,
so that it follows from the results recently obtained in \cite
{PeterExp} that Assumption~\ref{assjcfmom} is satisfied and our result
on the smoothness of the densities applies, provided that H\"ormander's
condition holds.

As an immediate consequence, we have the smoothness of the L\'evy area,
which was recently obtained independently in~\cite{Driscoll}:

%
%pr5 #&#
\begin{proposition}
Let $B$ be a $d$-dimensional fractional Brownian motion with Hurst
parameter $H > \frac{1}{ 3}$,
and let $W_{ij}(t) = \int_0^t B_i(s)\circ dB_j(s) - \int_0^t
B_j(s)\circ dB_i(s)$ for $i < j$. Then,
for any fixed $t>0$, the vector $(B_k(t), W_{ij}(t))$ with $k =
1,\ldots,d$ and $i < j$ has a smooth
density with respect to Lebesgue measure.
\end{proposition}

\begin{pf}
The verification of H\"ormander's condition boils down to a simple
problem in linear algebra.
Writing $e_j$ for the basis vector in the direction $B_j$ and $f_{ij}$
for the basis vector in the
direction $W_{ij}$, we can rewrite $x = B \oplus W$ as the solution to
the SDE
\[
dx = \sum_j \biggl(e_j + \sum
_{i<j} f_{ij}\langle x,e_i
\rangle- \sum_{i>j}f_{ji} \langle
x,e_i\rangle\biggr)\circ dB_j = \sum
_j V_j(x)\circ dB_j .
\]
An explicit calculation shows that, for $j < k$, we have
\[
[V_j, V_k](x) = 2f_{jk} ,
\]
so that H\"ormander's condition holds after one step.
\end{pf}

%
%re17 #&#
\begin{remark}
Higher order totally antisymmetric iterated integrals can be treated in
exactly the same way
with the $k$th iterated Lie brackets recovering precisely the basis
vectors of the elements in the
$k$th antisymmetric tensor.
\end{remark}

%s7.3 #&#
\subsection{Simplified fractional Langevin equation} \label{secsfle}

Consider the process $(q_t,p_t)$ on $\R^{2n}$ given by
%
%
%e7.4 #&#
\begin{equation}
\label{eLangevin} dq = p \,dt , \qquad dp = -\nabla V(q) \,dt - p \,dt +
dB_H(t) ,\vadjust{\goodbreak}
\end{equation}
where, for the sake of simplicity, we assume that $V\colon\R^n\to\R
_+$ has bounded second derivative,
and there exist $C>0$ and $\kappa> 0$ such that
%
%
%e7.5 #&#
\begin{equation}
\label{ecoercV} \bigl\langle q,\nabla V(q)\bigr\rangle\ge\kappa|q|^2
- C ,\qquad  V(q) \ge\kappa|q|^2 - C .
\end{equation}
This equation is a simplified version of the fractional Langevin
equation. (The equation
satisfying the correct physical detailed balance condition would have a more
complicated memory kernel instead of the simple friction term $-p \,dt$
appearing above.)

Because we assume $V$ to have a bounded second derivative, the Jacobian
of~(\ref{eLangevin})
is bounded by a deterministic constant over any finite time interval.
Furthermore, H\"ormander's condition
is easy to verify, so that we can apply Theorem~\ref{theosmooth} to
infer the existence of smooth densities for the
joint distribution of the solution at any time.

Regarding the existence of a unique invariant measure for (\ref{eOU}),
it only remains to
obtain a Lyapunov function for the solution to (\ref{eLangevin}). For
this, similar to~\cite{Hair05}, we proceed as follows.
We consider the process $(\tilde p, \tilde q)$ solution to
%
%
%e7.6 #&#
\[
d\tilde q = -\tilde q \,dt ,\qquad  d\tilde p = -\tilde p \,dt + dB_H(t) .
\]
It is, of course, trivial to bound solutions to this equation. Then we
set $P = p - \tilde p$ and $Q = q-\tilde q$.
The equation for $(P,Q)$ can be written as
%
%
%e7.7 #&#
\[
\dot Q = P + R_Q ,\qquad \dot P = - \nabla V(Q) - P + R_P
,
\]
where
%
%
%e7.8 #&#
\[
R_Q = \tilde p - \tilde q ,\qquad R_P = \nabla V(Q) -
\nabla V(Q + \tilde q) .
\]
Note that since we assumed that $V$ has bounded second derivative, both
$R_P$ and $R_Q$ are
bounded by a multiple of $|\tilde p| + |\tilde q|$, independently of
$P$ and $Q$.
We now set $\bar H(P,Q) = \frac{1}{ 2} P^2 + V(Q) + \gamma PQ$ for a
constant $\gamma$ to be determined later.
An explicit calculation yields the bound
%
%
%e7.9 #&#
\begin{eqnarray*}
\frac{d}{ dt} \bar H(P,Q) &=& - (1-\gamma) |P|^2 - \gamma\bigl
\langle Q,\nabla V(Q)\bigr\rangle+ \bigl\langle\nabla V(Q) + \gamma P,
R_Q\bigr\rangle\\
&&{}+ \langle P + \gamma Q, R_P\rangle.
\end{eqnarray*}
Making use of (\ref{ecoercV}) and the bounds on $R_Q$ and $R_P$, we
see that there exists constant
$\alpha> 0$ and $C>0$ such that
%
%
%e7.10 #&#
\[
\frac{d}{ dt} \bar H(P,Q) \le- \alpha\bar H(P,Q) + C \bigl(1 +
|\tilde
p|^2 + |\tilde q|^2 \bigr) .
\]
Since, by (\ref{ecoercV}), $\bar H$ grows quadratically at infinity
for $\gamma$ small enough,
it follows in the same way as in~\cite{Hair05}, Proposition~3.12, that
$|p|^2 + |q|^2$ is a Lyapunov function for
(\ref{eLangevin}). We therefore have:

%
%th7.1 #&#
\begin{theorem}
If $V$ has bounded second derivative and (\ref{ecoercV}) holds, then
there exists a unique invariant measure for
(\ref{eLangevin}).
\end{theorem}

\begin{pf}
The existence of an invariant measure follows from the fact that $|p|^2
+ |q|^2$ is a Lyapunov function.
The uniqueness then follows from Theorem~\ref{doob}.
 \end{pf}

%
%re18 #&#
\begin{remark}
Our results also apply to more degenerate situations. For example, if
we consider the fractional Langevin
equation associated to systems of anharmonic oscillators in contact
with thermal baths at their boundary,
as studied in~\cite{MR1685893,MR1764365}, our results imply the
uniqueness of a steady state.
Existence of a steady state, however, is a much harder problem in such
systems, which is
partially unsolved even in the Markovian case.
\end{remark}

%
%re19 #&#
\begin{remark}
The noise in the above example is additive, and thus it might seem that
we are not using the rough path nature of the fBm here. This is true in
this particular case, but in more complicated situations with additive
noise, like, for example,
the one in~\cite{MR1764365}, both rough path analysis and our version
of Norris's lemma are
still needed in order to analyze expressions such as (\ref
{eqnLiebrackcal}), when $U$ is given by a higher-order Lie bracket.
\end{remark}

%
%sA #&#
\begin{appendix}

%sA #&#
\section*{Appendix: Bounds on the cutoff function}\label{app}

In this section, we show that the function $\Lambda_{\beta,q}$
appearing in Sections~\ref{secapproxDens} and~\ref{secSF}
does indeed have the requested smoothness properties. Our main result
is the following:

%
%pr6 #&#
\begin{proposition}\label{propderLambda}
Let $\Lambda_{\beta,q}$ be as in (\ref{eqnCaplam}), and assume that
$\beta$ and $q$ are such that
(\ref{eboundNorm}) holds
and such that $2\beta\le\gamma+ H$. [This is always possible by
first setting $\beta= (\gamma+H)/2$ and
then choosing $q$ large enough.]

Then, for every $k>0$
and every $R>0$, there exists a constant $M$ such that the bound
%
%
%eA.1 #&#
\[
\bigl\|\D^{(k)} \Lambda_{\beta,q} (X,\bbX)\bigr\| \le M ,
\]
holds for all $(X,\bbX)$ such that $\Lambda_{\beta,q} (X,\bbX)
\le2R$. Here,
$\D^{(k)}$ denotes the $k$th iterated Malliavin derivative, and $\|
\cdot\|$ is the $L^2$-norm
on $[0,T]^k$.
\end{proposition}

\begin{pf}
Note first that, by definition,
%
%
%eA.2 #&#
\[
\DD_r^i \delta X^j_{s,t} =
\delta_{ij}\one_{r \in[s, t]} ,
\]
where $\delta_{ij}$ is the Kronecker delta.
It thus follows from (\ref{erelDD}) that
%
%
%eA.3 #&#
\setcounter{equation}{0}
\begin{equation}
\label{ederdeltaX} \qquad\D_r^i \delta X^j_{s,t}
= c \delta_{ij} \bigl((t-r)^{H-{1}/{
2}} \one_{r < t} -
(s-r)^{H-{1}/{ 2}} \one_{r < s} \bigr) \eqdef\delta_{ij}
f_{s,t}(r) .
\end{equation}
The $L^2$-norm of $f_{s,t}$ is given by
\[
\|f_{s,t}\|^2 = c^2\int_{s}^t
(t-r)^{2H-1} \,dr + c^2\int_0^s
\bigl( (t-r)^{2H-1} - (s-r)^{2H-1} \bigr) \,dr .
\]
Since $H < \frac{1}{ 2}$ by assumption, one has the inequality
\[
\bigl|t^{2H} - s^{2H} \bigr| \le|t-s|^{2H} ,
\]
so that a straightforward calculation yields the bound
\[
\|f_{s,t}\| \le\kappa|t-s|^H
\]
for some constant $\kappa> 0$.

Concerning $\tilde\bbX^{k\ell}_{s,t} = \bbX^{k\ell}_{s,t} - \bbX
^{\ell k}_{s,t}$, an explicit calculation yields the identity
\[
\DD_r^i \tilde\bbX^{k\ell}_{s,t} =
\one_{r \in[s, t]} \bigl(\delta_{ik} \bigl(\delta
X^\ell_{r,t} - \delta X^\ell_{s,r}
\bigr) - \delta_{i\ell} \bigl(\delta X^k_{r,t} -
\delta X^k_{s,r} \bigr) \bigr) .
\]
Applying again (\ref{erelDD}), we obtain
%
%
%eA.4 #&#
\[
\D_{r}^j \bbX^{k\ell}_{s,t} =
\delta_{ik} G^\ell_{s,t}(r) +
\delta_{i\ell} G^k_{s,t}(r) ,
\]
where we set
\begin{eqnarray*}
G^k_{s,t}(r) &=& 2\one_{r \in[s, t]} \int
_r^t (u-r)^{H-{3}/{ 2}} \delta
X^k_{r,u} \,du
\\
&&{} + \one_{r \in[s, t]} \bigl(\delta X^k_{r,t} - \delta
X^k_{s,r} \bigr)\int_t^\infty(u-r)^{H-{3}/{ 2}}
\,du
\\
&&{} + \one_{r < s} \int_s^t
(u-r)^{H-{3}/{ 2}} \bigl(\delta X^k_{s,u} - \delta
X^k_{u,t} \bigr) \,du .
\end{eqnarray*}
The important fact about $G^k_{s,t}(r)$ is that we can estimate it by
\begin{eqnarray*}
\bigl|G^k_{s,t}(r)\bigr| &\le& M\|X\|_\gamma
\one_{r\in[s,t]} (t-r)^{\gamma
+H-{1}/{ 2}}
\\
&&{} + M\|X\|_\gamma\one_{r<s} (t-s)^\gamma
\bigl((t-r)^{H -{1}/{ 2}}-(s-r)^{H - {1}/{ 2}} \bigr) ,
\end{eqnarray*}
so that its $L^2$-norm is controlled by
%
%
%eA.5 #&#
\begin{equation}
\label{eboundG} \bigl\|G^k_{s,t}\bigr\| \le M \|X
\|_\gamma|t-s|^{H+\gamma} .
\end{equation}

We finally compute the second Malliavin derivative of $\bbX^{k\ell
}_{s,t}$. It follows in a rather straightforward way from (\ref{erelDD})
that, for $r_1 \in(s,t)$ and $v \in(r_1,t)$, one has
%
%
%eA.6 #&#
\[
\DD_{v}^i \D_{r_1}^j
\bbX^{k\ell}_{s,t} = (\delta_{ik}\delta_{j\ell}
+ \delta_{i\ell}\delta_{jk} ) \biggl(2 \int
_v^t (u-r_1)^{H-{3}/{ 2}} \,du +
\int_t^\infty(u-r_1)^{H-{3}/{
2}}
\,du \biggr)
\]
for $r_1 < s$ and $v \in(s,t)$, one has
%
%
%eA.7 #&#
\[
\DD_{v}^i \D_{r_1}^j
\bbX^{k\ell}_{s,t} = (\delta_{ik}\delta_{j\ell}
+ \delta_{i\ell}\delta_{jk} ) \biggl(2 \int
_v^t (u-r_1)^{H-{3}/{ 2}} \,du -
\int_s^t(u-r_1)^{H-{3}/{ 2}}
\,du \biggr) ,
\]
and for all other combinations with $v > r_1$ one has $\DD_{v}^i \D
_{r_1}^j \bbX^{k\ell}_{s,t} = 0$.\vadjust{\goodbreak}

A lengthy but straightforward calculation then yields
%
%
%eA.8 #&#
\begin{equation}
\label{edefg} \D_{r_2}^i \D_{r_1}^j
\bbX^{k\ell}_{s,t} = (\delta_{ik}\delta_{j\ell}
+ \delta_{i\ell}\delta_{jk} ) g_{s,t}(r_1,r_2)
,
\end{equation}
where, for $s < r_1 < r_2 < t$, the function $g_{s,t}$ is given by
\begin{eqnarray*}
g_{s,t}(r_1,r_2) &=& 2\int
_{r_2}^t (v-r_2)^{H-{3}/{ 2}} \int
_{r_2}^v (u-r_1)^{H-{3}/{ 2}} \,du
\,dv
\\
& &{}+ c_g (t-r_2)^{H-{1}/{ 2}} \bigl(2(r_2-r_1)^{H-{1}/{
2}}
- (t-r_1)^{H-{1}/{ 2}} \bigr)
\\
&\eqdef& g_{s,t}^{(1)}(r_1,r_2)+
g_{s,t}^{(2)}(r_1,r_2)
\end{eqnarray*}
for some constant $c_g$. For $r_1 < s < r_2 < t$ on the other hand, one has
\begin{eqnarray*}
&&g_{s,t}(r_1,r_2)\\
 &&\qquad= 2\int
_{r_2}^t (v-r_2)^{H-{3}/{ 2}} \int
_{r_2}^v (u-r_1)^{H-{3}/{ 2}} \,du
\,dv
\\
&&\qquad\quad{} + c_g (t-r_2)^{H-{1}/{ 2}} \bigl(2(r_2-r_1)^{H-{1}/{
2}}
- (s-r_1)^{H-{1}/{ 2}} - (t-r_1)^{H-{1}/{ 2}}
\bigr)
\\
&&\qquad\eqdef g_{s,t}^{(3)}(r_1,r_2)+
g_{s,t}^{(4)}(r_1,r_2) .
\end{eqnarray*}
Finally, for $r_1 < r_2 < s < t$, one has
\begin{eqnarray*}
g_{s,t}(r_1,r_2) = 2\int_{s}^t
(v-r_2)^{H-{3}/{ 2}} \int_{r_2}^v
(u-r_1)^{H-{3}/{ 2}} \,du \,dv \eqdef g_{s,t}^{(5)}(r_1,r_2)
.
\end{eqnarray*}
It is possible to check that
%
%
%eA.9 #&#
\begin{equation}
\label{eboundg} \|g_{s,t}\| \le M |t-s|^{2H} ,
\end{equation}
where $\|\cdot\|$ denotes again the $L^2$-norm. (We postpone the proof
of this to Lemma~\ref{lemboundg} below.)

We now have all the necessary tools to conclude. Write
\[
\Lambda_{\beta,q}(X,\bbX) = \Lambda^{(1)}_{\beta,q}(X) +
\Lambda^{(2)}_{\beta,q}(\bbX) ,
\]
where $\Lambda^{(1)}$ is as (\ref{eqnCaplam}), but keeping only the
term proportional to $|\delta X_{s,t}|^{2q}$ in the
integral, and similarly for $\Lambda^{(2)}$. It follows from (\ref
{ederdeltaX}) that, for
$\ell\le2q$, the multiple Malliavin derivative of $\Lambda_{\beta,q}^{(1)}$
satisfies the bound
%
%
%eA.10 #&#
\[
\bigl|\D_{s_1}\cdots\D_{s_\ell}\Lambda_{\beta,q}^{(1)}(X)\bigr|
\le M\int_0^T \int_0^t
\frac{|\delta X_{s,t}|^{2q-\ell} }{ |t-s|^{\beta
(2q-\ell)}} \prod_{j=1}^\ell
\frac{|f_{s,t}(r_j)| }{ |t-s|^\beta} \,ds \,dt .
\]
Since the $L^2$-norm of $f_{s,t}$ is bounded by $M |t-s|^H$, it follows
immediately that there exits
a constant $M$ such that the $L^2$-norm of $\D_{s_1}\cdots\D_{s_\ell
}\Lambda_{\beta,q}^{(1)}$ is bounded
by $M \Lambda_{\beta,q}^{(1)}$. Its Malliavin derivative of order
$\ell> 2q$ on the other hand vanishes identically.

Similarly, we only need to consider Malliavin derivatives of order
$\ell\le2q$ for $\Lambda_{\beta,q}^{(2)}$. A reasoning
similar to the above shows that its Malliavin derivative can be written as
%
%
%eA.11 #&#
\[
\D_{s_1}\cdots\D_{s_\ell}\Lambda_{\beta,q}^{(2)}
= \int_0^T \int_0^t
\frac{\PP(\bbX_{s,t}, \D_{s_\cdot} \bbX_{s,t},
g_{s,t}(s_\cdot,s_\cdot) ) }{ |t-s|^{2\beta q}} \,ds \,dt ,
\]
where $\PP$ is a homogeneous polynomial of degree $q$ and ``$s_\cdot
$'' is a generic placeholder for any of the times $s_1,\ldots,s_\ell$.
It now follows from (\ref{eboundG}), (\ref{eboundg}) and the
assumption $2\beta\le\gamma+ H$ that
the $L^2$-norm of $\D_{s_1}\cdots\D_{s_\ell}\Lambda_{\beta
,q}^{(2)}$ is bounded by
$M (\Lambda_{\beta,q}^{(2)} + \|X\|_\gamma^q )$. Since on
the other hand, $\|X\|_\gamma^q$ is bounded by
$M \Lambda_{\beta,q}$, by assumption, this completes the proof.
\end{pf}

%
%co4 #&#
\begin{corollary}\label{corFublem2}
Let $\Psi_R(w_+)$ be as defined in
(\ref{eqnpsifn}) with $\Lambda_{\beta,q}$ as in (\ref{eqnCaplam}),
and let $\beta$ and $q$
be as in Proposition~\ref{propderLambda}.

Then, for every $R>0$, $\Psi_R \in\D^\infty$.
Furthermore, every multiple Malliavin derivative of $\Psi_R$ vanishes
outside of
the set $\{\Lambda_{\beta,q} (X,\bbX) \le2R\}$.
\end{corollary}

\begin{pf}
By the chain rule,
\[
\D_{s}\Psi_{R}(w_+) = \frac{1 }{ R}
\chi' \bigl(R^{-1}\Lambda_{\beta
,q} (X,\bbX)\bigr)
\D_{s} \Lambda_{\beta,q} (X,\bbX) ,
\]
and similarly for higher order derivatives. Since all derivatives of
$\chi$ vanish
when the argument is larger than $2$, the claim follows from
Proposition~\ref{propderLambda}.
\end{pf}

%
%le8 #&#
\begin{lemma}\label{lemboundg}
For every $T>0$, there exists a constant $M$ such that the function
$g_{s,t}$ from
(\ref{edefg}) satisfies $\|g_{s,t}\| \le M |t-s|^{2H}$.
\end{lemma}

\begin{pf}
We show the bound separately for $g_{s,t}^{(j)}$ with $j=1,\ldots,5$.
For $g_{s,t}^{(1)}$, we use the bound
%
%
%eA.12 #&#
\[
(u-r_1)^{H-{3}/{ 2}} \le(u-r_1)^{H-{3}/{
2}}(r_2-r_1)^{-\beta}
,
\]
in order to conclude that, provided that $1- 2H < \beta$,
one has the pointwise bound
%
%
%eA.13 #&#
\[
\bigl|g_{s,t}^{(1)}(r_1,r_2)\bigr|
\le(t-r_2)^{2H-1+\beta} (r_2 - r_1)^{-\beta
}
\le|t-s|^{2H-1+\beta}(r_2 - r_1)^{-\beta} .
\]
If furthermore $\beta< \frac{1}{ 2}$ (which is always possible if $H
> \frac{1}{ 4}$), the integral of $(r_2 - r_1)^{-2\beta}$
over $s < r_1 < r_2 < t$ is proportional to $|t-s|^{1-2\beta}$, thus
yielding the required bound.

Similarly, $g_{s,t}^{(2)}$ satisfies
\[
\bigl|g_{s,t}^{(2)}(r_1,r_2)\bigr| \le
C(t-r_2)^{H-{1}/{2}} (r_2 - r_1)^{H -
{1}/{ 2}}
,
\]
and a straightforward calculation shows that
\[
\int_s^t (t-r_2)^{2H-1}
\int_{s}^{r_2} (r_2 -
r_1)^{2H - 1} \,dr_1 \,dr_2 = C
|t-s|^{4H}
\]
for some constant $C$ as required.

For $g_{s,t}^{(3)}$ we have, as for $g_{s,t}^{(1)}$,
%
%
%eA.14 #&#
\[
\bigl|g_{s,t}^{(3)}(r_1,r_2)\bigr|
\le(t-r_2)^{2H-1+\beta} (r_2 - r_1)^{-\beta
}
.
\]
This time, however, we choose $\beta\in(\frac{1}{ 2}, 1)$, so that
%
%
%eA.15 #&#
\[
\int_s^t \int_0^s
\bigl|g_{s,t}^{(3)}(r_1,r_2)\bigr|^2
\,dr_1 \,dr_2 \le M \int_s^t
(t-r_2)^{4H-2+2\beta} (r_2-s)^{1-2\beta}
\,dr_2 ,
\]
which is indeed proportional to $|t-s|^{4H}$.

To bound $g_{s,t}^{(4)}$ we perform the change of variables $r_1
\mapsto s-r_1$ and $r_2 \mapsto r_2 + s$,
so that
\begin{eqnarray*}
&&\int_s^t \int_0^s
\bigl|g_{s,t}^{(4)}(r_1,r_2)\bigr|^2
\,dr_1 \,dr_2\\
&&\qquad= M \int_0^{t-s}
(t-s-r_2)^{2H-1}
\\
&&\hspace*{34pt}\qquad\quad{} \times\int_0^s \bigl(2(r_2+r_1)^{H-{1}/{ 2}}
- r_1^{H-{1}/{ 2}} - (t-s+r_1)^{H-{1}/{ 2}}
\bigr)^2 \,dr_1 \,dr_2 .
\end{eqnarray*}
We then dilate the expression by $t-s$, showing that it is proportional to
%
%
%eA.16 #&#
\begin{eqnarray*}
&&|t-s|^{4H} \int_0^1\int
_0^{{s} /{ (t-s)}} (1-r_2)^{2H-1}\\
&&\hspace*{72pt}\qquad{}\times
\bigl(2(r_2+r_1)^{H-{1}/{ 2}} - r_1^{H-{1}/{ 2}}
- (1+r_1)^{H-{1}/{ 2}} \bigr)^2 \,dr_1
\,dr_2 .
\end{eqnarray*}
It is straightforward to check that this integral converges for all $H
\in(0,1)$, which shows the requested bound on $g_{s,t}^{(4)}$.

Finally, to bound $g_{s,t}^{(5)}$, we perform the change of variables
$r_1 \mapsto s-r_1$ and $r_2 \mapsto s - r_2$,
followed by a dilatation of $t-s$, so that
\[
\int_0^s \int_0^{r_2}
\bigl|g_{s,t}^{(5)}(r_1,r_2)\bigr|^2
\,dr_1 \,dr_2 = |t-s|^{4H} \int
_0^{{s} /{ (t-s)}} \int_{r_2}^{{s}/{( t-s)}} \bigl|\tilde g^{(5)}(r_1,r_2)\bigr|^2
\,dr_1 \,dr_2 ,
\]
where
\[
\tilde g^{(5)}(r_1,r_2) = \int
_{-1}^0 (r_2-v)^{H-{3}/{ 2}} \int
_{v}^{r_2} (r_1-u)^{H-{3}/{ 2}}
\,du \,dv .
\]
Note now that, for every $\beta\in(0, \frac{3}{ 2}-H)$, there exists
a constant $M$ such that, for $r_1 > r_2$, one has the bound
%
%
%eA.17 #&#
\begin{eqnarray*}
&&\bigl|\tilde g^{(5)}(r_1,r_2)\bigr| \\
&&\qquad\le
M(r_1-r_2)^{-\beta} \int_{-1}^0
(r_2-v)^{2H-2+\beta} \,dv \le M \frac{(r_1-r_2)^{-\beta}
r_2^{2H-1+\beta} }{ 1+r_2} \,dv .
\end{eqnarray*}
Choosing $\beta\approx\frac{1}{ 2}$ (but slightly larger than $\frac
{1}{ 2}$) for $r_1 > r_2 + 1$ and $\beta= 0$ for $r_1 \le r_2$,
we can check that
%
%
%eA.18 #&#
\[
\int_0^{\infty} \int_{r_2}^{\infty}
\bigl|\tilde g^{(5)}(r_1,r_2)\bigr|^2
\,dr_1 \,dr_2 < \infty,
\]
so that the claim follows.
\end{pf}
\end{appendix}

\section*{Acknowledgments}
We are grateful to Tom Cass, Peter Friz, Massimilliano Gubinelli, Terry
Lyons and Samy Tindel for many interesting discussions
on the subject treated in this work. We also thank an anonymous referee
for detailed comments which helped us improve the exposition.

\def\cprime{$'$}
%
% imsref loaded by akundreckaite, 2012-10-24 14:47:00
%
% imsref loaded by akundreckaite, 2012-10-25 15:18:16

%

%suskaldyti doi

\printaddresses


\begin{thebibliography}{43}
% BibTex style file: ims.bst, 2012-08-21
% Default style options (sort=0,type=number).
% Used options (sort=1,type=number).

\bibitem{Ande55}
\begin{barticle}[mr]
\bauthor{\bsnm{Anderson},~\bfnm{T.~W.}\binits{T.~W.}}
(\byear{1955}).
\btitle{The integral of a symmetric unimodal function over a symmetric convex
  set and some probability inequalities}.
\bjournal{Proc. Amer. Math. Soc.}
\bvolume{6}
\bpages{170--176}.
\bid{issn={0002-9939}, mr={0069229}}
\bptok{imsref}%
\end{barticle}
\endbibitem

\bibitem{Keith}
\begin{barticle}[mr]
\bauthor{\bsnm{Ball},~\bfnm{Keith}\binits{K.}}
(\byear{1992}).
\btitle{Ellipsoids of maximal volume in convex bodies}.
\bjournal{Geom. Dedicata}
\bvolume{41}
\bpages{241--250}.
\bid{doi={10.1007/BF00182424}, issn={0046-5755}, mr={1153987}}
\bptok{imsref}%
\end{barticle}
\endbibitem

\bibitem{BaudHair07}
\begin{barticle}[mr]
\bauthor{\bsnm{Baudoin},~\bfnm{Fabrice}\binits{F.}} \AND
  \bauthor{\bsnm{Hairer},~\bfnm{Martin}\binits{M.}}
(\byear{2007}).
\btitle{A version of {H}\"ormander's theorem for the fractional {B}rownian
  motion}.
\bjournal{Probab. Theory Related Fields}
\bvolume{139}
\bpages{373--395}.
\bid{doi={10.1007/s00440-006-0035-0}, issn={0178-8051}, mr={2322701}}
\bptok{imsref}%
\end{barticle}
\endbibitem

\bibitem{Bismut81}
\begin{barticle}[mr]
\bauthor{\bsnm{Bismut},~\bfnm{Jean-Michel}\binits{J.-M.}}
(\byear{1981}).
\btitle{Martingales, the {M}alliavin calculus and hypoellipticity under general
  {H}\"ormander's conditions}.
\bjournal{Z. Wahrsch. Verw. Gebiete}
\bvolume{56}
\bpages{469--505}.
\bid{doi={10.1007/BF00531428}, issn={0044-3719}, mr={0621660}}
\bptok{imsref}%
\end{barticle}
\endbibitem

\bibitem{MR0305608}
\begin{barticle}[mr]
\bauthor{\bsnm{Butcher},~\bfnm{J.~C.}\binits{J.~C.}}
(\byear{1972}).
\btitle{An algebraic theory of integration methods}.
\bjournal{Math. Comp.}
\bvolume{26}
\bpages{79--106}.
\bid{issn={0025-5718}, mr={0305608}}
\bptok{imsref}%
\end{barticle}
\endbibitem

\bibitem{CassFriz}
\begin{barticle}[mr]
\bauthor{\bsnm{Cass},~\bfnm{Thomas}\binits{T.}} \AND
  \bauthor{\bsnm{Friz},~\bfnm{Peter}\binits{P.}}
(\byear{2010}).
\btitle{Densities for rough differential equations under {H}\"ormander's
  condition}.
\bjournal{Ann. of Math. (2)}
\bvolume{171}
\bpages{2115--2141}.
\bid{doi={10.4007/annals.2010.171.2115}, issn={0003-486X}, mr={2680405}}
\bptok{imsref}%
\end{barticle}
\endbibitem

\bibitem{MallRough}
\begin{barticle}[mr]
\bauthor{\bsnm{Cass},~\bfnm{Thomas}\binits{T.}},
  \bauthor{\bsnm{Friz},~\bfnm{Peter}\binits{P.}} \AND
  \bauthor{\bsnm{Victoir},~\bfnm{Nicolas}\binits{N.}}
(\byear{2009}).
\btitle{Non-degeneracy of {W}iener functionals arising from rough differential
  equations}.
\bjournal{Trans. Amer. Math. Soc.}
\bvolume{361}
\bpages{3359--3371}.
\bid{doi={10.1090/S0002-9947-09-04677-7}, issn={0002-9947}, mr={2485431}}
\bptok{imsref}%
\end{barticle}
\endbibitem

\bibitem{CLL11}
\begin{bmisc}[author]
\bauthor{\bsnm{{Cass}},~\bfnm{T.}\binits{T.}},
  \bauthor{\bsnm{{Litterer}},~\bfnm{C.}\binits{C.}} \AND
  \bauthor{\bsnm{{Lyons}},~\bfnm{T.}\binits{T.}}
(\byear{2011}).
\bhowpublished{Integrability estimates for {G}aussian rough differential
  equations. Available at arXiv:\arxivurl{1104.1813}.}
\bptok{imsref}%
\end{bmisc}
\endbibitem

\bibitem{CoutQian}
\begin{barticle}[mr]
\bauthor{\bsnm{Coutin},~\bfnm{Laure}\binits{L.}} \AND
  \bauthor{\bsnm{Qian},~\bfnm{Zhongmin}\binits{Z.}}
(\byear{2002}).
\btitle{Stochastic analysis, rough path analysis and fractional {B}rownian
  motions}.
\bjournal{Probab. Theory Related Fields}
\bvolume{122}
\bpages{108--140}.
\bid{doi={10.1007/s004400100158}, issn={0178-8051}, mr={1883719}}
\bptok{imsref}%
\end{barticle}
\endbibitem

\bibitem{Driscoll}
\begin{bmisc}[author]
\bauthor{\bsnm{Driscoll},~\bfnm{P.}\binits{P.}}
(\byear{2010}).
\bhowpublished{Smoothness of density for the area process of fractional
  {B}rownian motion. Available at arXiv:\arxivurl{1010.3047}.}
\bptok{imsref}%
\end{bmisc}
\endbibitem

\bibitem{MR1764365}
\begin{barticle}[mr]
\bauthor{\bsnm{Eckmann},~\bfnm{J.~P.}\binits{J.~P.}} \AND
  \bauthor{\bsnm{Hairer},~\bfnm{M.}\binits{M.}}
(\byear{2000}).
\btitle{Non-equilibrium statistical mechanics of strongly anharmonic chains of
  oscillators}.
\bjournal{Comm. Math. Phys.}
\bvolume{212}
\bpages{105--164}.
\bid{doi={10.1007/s002200000216}, issn={0010-3616}, mr={1764365}}
\bptok{imsref}%
\end{barticle}
\endbibitem

\bibitem{MR1685893}
\begin{barticle}[mr]
\bauthor{\bsnm{Eckmann},~\bfnm{J.~P.}\binits{J.~P.}},
  \bauthor{\bsnm{Pillet},~\bfnm{C.~A.}\binits{C.~A.}} \AND
  \bauthor{\bsnm{Rey-Bellet},~\bfnm{L.}\binits{L.}}
(\byear{1999}).
\btitle{Non-equilibrium statistical mechanics of anharmonic chains coupled to
  two heat baths at different temperatures}.
\bjournal{Comm. Math. Phys.}
\bvolume{201}
\bpages{657--697}.
\bid{doi={10.1007/s002200050572}, issn={0010-3616}, mr={1685893}}
\bptok{imsref}%
\end{barticle}
\endbibitem

\bibitem{Xuemei}
\begin{barticle}[mr]
\bauthor{\bsnm{Elworthy},~\bfnm{K.~D.}\binits{K.~D.}} \AND
  \bauthor{\bsnm{Li},~\bfnm{X.~M.}\binits{X.~M.}}
(\byear{1994}).
\btitle{Formulae for the derivatives of heat semigroups}.
\bjournal{J.~Funct. Anal.}
\bvolume{125}
\bpages{252--286}.
\bid{doi={10.1006/jfan.1994.1124}, issn={0022-1236}, mr={1297021}}
\bptok{imsref}%
\end{barticle}
\endbibitem

\bibitem{PeterExp}
\begin{bmisc}[author]
\bauthor{\bsnm{Friz},~\bfnm{P.}\binits{P.}} \AND
  \bauthor{\bsnm{Riedel},~\bfnm{S.}\binits{S.}}
(\byear{2011}).
\bhowpublished{Integrability of linear rough differential equations. Available
  at arXiv:\arxivurl{1104.0577}.}
\bptok{imsref}%
\end{bmisc}
\endbibitem

\bibitem{FVGauss}
\begin{barticle}[mr]
\bauthor{\bsnm{Friz},~\bfnm{Peter}\binits{P.}} \AND
  \bauthor{\bsnm{Victoir},~\bfnm{Nicolas}\binits{N.}}
(\byear{2010}).
\btitle{Differential equations driven by {G}aussian signals}.
\bjournal{Ann. Inst. Henri Poincar\'e Probab. Stat.}
\bvolume{46}
\bpages{369--413}.
\bid{doi={10.1214/09-AIHP202}, issn={0246-0203}, mr={2667703}}
\bptok{imsref}%
\end{barticle}
\endbibitem

\bibitem{FrizVict10}
\begin{bbook}[mr]
\bauthor{\bsnm{Friz},~\bfnm{Peter~K.}\binits{P.~K.}} \AND
  \bauthor{\bsnm{Victoir},~\bfnm{Nicolas~B.}\binits{N.~B.}}
(\byear{2010}).
\btitle{Multidimensional Stochastic Processes as Rough Paths: Theory and
  Applications}.
\bseries{Cambridge Studies in Advanced Mathematics}
\bvolume{120}.
\bpublisher{Cambridge Univ. Press}, \blocation{Cambridge}.
\bid{mr={2604669}}
\bptok{imsref}%
\end{bbook}
\endbibitem

\bibitem{Gubi04}
\begin{barticle}[mr]
\bauthor{\bsnm{Gubinelli},~\bfnm{M.}\binits{M.}}
(\byear{2004}).
\btitle{Controlling rough paths}.
\bjournal{J. Funct. Anal.}
\bvolume{216}
\bpages{86--140}.
\bid{doi={10.1016/j.jfa.2004.01.002}, issn={0022-1236}, mr={2091358}}
\bptok{imsref}%
\end{barticle}
\endbibitem

\bibitem{Hair05}
\begin{barticle}[mr]
\bauthor{\bsnm{Hairer},~\bfnm{Martin}\binits{M.}}
(\byear{2005}).
\btitle{Ergodicity of stochastic differential equations driven by fractional
  {B}rownian motion}.
\bjournal{Ann. Probab.}
\bvolume{33}
\bpages{703--758}.
\bid{doi={10.1214/009117904000000892}, issn={0091-1798}, mr={2123208}}
\bptok{imsref}%
\end{barticle}
\endbibitem

\bibitem{Hair09}
\begin{bincollection}[mr]
\bauthor{\bsnm{Hairer},~\bfnm{Martin}\binits{M.}}
(\byear{2009}).
\btitle{Ergodic properties of a class of non-{M}arkovian processes}.
In \bbooktitle{Trends in Stochastic Analysis}.
\bseries{London Mathematical Society Lecture Note Series}
\bvolume{353}
\bpages{65--98}.
\bpublisher{Cambridge Univ. Press}, \blocation{Cambridge}.
\bid{mr={2562151}}
\bptok{imsref}%
\end{bincollection}
\endbibitem

\bibitem{Hair10}
\begin{barticle}[author]
\bauthor{\bsnm{Hairer},~\bfnm{M.}\binits{M.}}
(\byear{2011}).
\btitle{Rough stochastic PDEs}.
\bjournal{Comm. Pure Appl. Math.}
\bvolume{64}
\bpages{1547--1585}.
\bid{mr={2832168}}
\bptok{imsref}%
\end{barticle}
\endbibitem

\bibitem{HairMatt09}
\begin{barticle}[author]
\bauthor{\bsnm{Hairer},~\bfnm{Martin}\binits{M.}} \AND
  \bauthor{\bsnm{Mattingly},~\bfnm{Jonathan~C.}\binits{J.~C.}}
(\byear{2011}).
\btitle{A theory of hypoellipticity and unique ergodicity for semilinear
stochastic {PDE}s}.
\bjournal{Electron. J. Probab.}
\bvolume{16}
\bpages{658--738}.
\bid{mr={2786645}}
\bptok{imsref}%
\end{barticle}
\endbibitem

\bibitem{HairOhas07}
\begin{barticle}[mr]
\bauthor{\bsnm{Hairer},~\bfnm{M.}\binits{M.}} \AND
  \bauthor{\bsnm{Ohashi},~\bfnm{A.}\binits{A.}}
(\byear{2007}).
\btitle{Ergodic theory for {SDE}s with extrinsic memory}.
\bjournal{Ann. Probab.}
\bvolume{35}
\bpages{1950--1977}.
\bid{doi={10.1214/009117906000001141}, issn={0091-1798}, mr={2349580}}
\bptok{imsref}%
\end{barticle}
\endbibitem

\bibitem{HairPill10}
\begin{barticle}[mr]
\bauthor{\bsnm{Hairer},~\bfnm{M.}\binits{M.}} \AND
  \bauthor{\bsnm{Pillai},~\bfnm{N.~S.}\binits{N.~S.}}
(\byear{2011}).
\btitle{Ergodicity of hypoelliptic {SDE}s driven by fractional {B}rownian
  motion}.
\bjournal{Ann. Inst. Henri Poincar\'e Probab. Stat.}
\bvolume{47}
\bpages{601--628}.
\bid{doi={10.1214/10-AIHP377}, issn={0246-0203}, mr={2814425}}
\bptok{imsref}%
\end{barticle}
\endbibitem

\bibitem{HorActa}
\begin{barticle}[mr]
\bauthor{\bsnm{H{\"o}rmander},~\bfnm{Lars}\binits{L.}}
(\byear{1967}).
\btitle{Hypoelliptic second order differential equations}.
\bjournal{Acta Math.}
\bvolume{119}
\bpages{147--171}.
\bid{issn={0001-5962}, mr={0222474}}
\bptok{imsref}%
\end{barticle}
\endbibitem

\bibitem{HuTind11}
\begin{bmisc}[author]
\bauthor{\bsnm{{Hu}},~\bfnm{Y.}\binits{Y.}} \AND
  \bauthor{\bsnm{{Tindel}},~\bfnm{S.}\binits{S.}}
(\byear{2011}).
\bhowpublished{Smooth density for some nilpotent rough differential equations.
  Available at arXiv:\arxivurl{1104.1972}.}
\bptok{imsref}%
\end{bmisc}
\endbibitem

\bibitem{John}
\begin{bincollection}[mr]
\bauthor{\bsnm{John},~\bfnm{Fritz}\binits{F.}}
(\byear{1948}).
\btitle{Extremum problems with inequalities as subsidiary conditions}.
In \bbooktitle{Studies and {E}ssays {P}resented to {R}. {C}ourant on His 60th
  {B}irthday, {J}anuary 8, 1948}
\bpages{187--204}.
\bpublisher{Interscience}, \blocation{New York, NY}.
\bid{mr={0030135}}
\bptok{imsref}%
\end{bincollection}
\endbibitem

\bibitem{shifted}
\begin{barticle}[mr]
\bauthor{\bsnm{Kuelbs},~\bfnm{James}\binits{J.}},
  \bauthor{\bsnm{Li},~\bfnm{Wenbo~V.}\binits{W.~V.}} \AND
  \bauthor{\bsnm{Linde},~\bfnm{Werner}\binits{W.}}
(\byear{1994}).
\btitle{The {G}aussian measure of shifted balls}.
\bjournal{Probab. Theory Related Fields}
\bvolume{98}
\bpages{143--162}.
\bid{doi={10.1007/BF01192511}, issn={0178-8051}, mr={1258983}}
\bptok{imsref}%
\end{barticle}
\endbibitem

\bibitem{KSAMI}
\begin{bincollection}[mr]
\bauthor{\bsnm{Kusuoka},~\bfnm{Shigeo}\binits{S.}} \AND
  \bauthor{\bsnm{Stroock},~\bfnm{Daniel}\binits{D.}}
(\byear{1984}).
\btitle{Applications of the {M}alliavin calculus. {I}}.
In \bbooktitle{Stochastic Analysis ({K}atata/{K}yoto, 1982)}.
\bseries{North-Holland Math. Library}
\bvolume{32}
\bpages{271--306}.
\bpublisher{North-Holland}, \blocation{Amsterdam}.
\bid{doi={10.1016/S0924-6509(08)70397-0}, mr={0780762}}
\bptok{imsref}%
\end{bincollection}
\endbibitem

\bibitem{KSAMII}
\begin{barticle}[mr]
\bauthor{\bsnm{Kusuoka},~\bfnm{S.}\binits{S.}} \AND
  \bauthor{\bsnm{Stroock},~\bfnm{D.}\binits{D.}}
(\byear{1985}).
\btitle{Applications of the {M}alliavin calculus. {II}}.
\bjournal{J.~Fac. Sci. Univ. Tokyo Sect. IA Math.}
\bvolume{32}
\bpages{1--76}.
\bid{issn={0040-8980}, mr={0783181}}
\bptok{imsref}%
\end{barticle}
\endbibitem

\bibitem{KSAMIII}
\begin{barticle}[mr]
\bauthor{\bsnm{Kusuoka},~\bfnm{S.}\binits{S.}} \AND
  \bauthor{\bsnm{Stroock},~\bfnm{D.}\binits{D.}}
(\byear{1987}).
\btitle{Applications of the {M}alliavin calculus. {III}}.
\bjournal{J. Fac. Sci. Univ. Tokyo Sect. IA Math.}
\bvolume{34}
\bpages{391--442}.
\bid{issn={0040-8980}, mr={0914028}}
\bptok{imsref}%
\end{barticle}
\endbibitem

\bibitem{LiShao01}
\begin{bincollection}[mr]
\bauthor{\bsnm{Li},~\bfnm{W.~V.}\binits{W.~V.}} \AND
  \bauthor{\bsnm{Shao},~\bfnm{Q.~M.}\binits{Q.~M.}}
(\byear{2001}).
\btitle{Gaussian processes: Inequalities, small ball probabilities and
  applications}.
In \bbooktitle{Stochastic Processes: Theory and Methods}.
\bseries{Handbook of Statist.}
\bvolume{19}
\bpages{533--597}.
\bpublisher{North-Holland}, \blocation{Amsterdam}.
\bid{doi={10.1016/S0169-7161(01)19019-X}, mr={1861734}}
\bptok{imsref}%
\end{bincollection}
\endbibitem

\bibitem{MR2036784}
\begin{bbook}[mr]
\bauthor{\bsnm{Lyons},~\bfnm{Terry}\binits{T.}} \AND
  \bauthor{\bsnm{Qian},~\bfnm{Zhongmin}\binits{Z.}}
(\byear{2002}).
\btitle{System Control and Rough Paths}.
\bpublisher{Oxford Univ. Press}, \blocation{Oxford}.
\bid{doi={10.1093/acprof:oso/9780198506485.001.0001}, mr={2036784}}
\bptok{imsref}%
\end{bbook}
\endbibitem

\bibitem{TerryRough}
\begin{barticle}[mr]
\bauthor{\bsnm{Lyons},~\bfnm{Terry~J.}\binits{T.~J.}}
(\byear{1998}).
\btitle{Differential equations driven by rough signals}.
\bjournal{Rev. Mat. Iberoam.}
\bvolume{14}
\bpages{215--310}.
\bid{doi={10.4171/RMI/240}, issn={0213-2230}, mr={1654527}}
\bptok{imsref}%
\end{barticle}
\endbibitem

\bibitem{TerryFlour}
\begin{bbook}[mr]
\bauthor{\bsnm{Lyons},~\bfnm{Terry~J.}\binits{T.~J.}},
  \bauthor{\bsnm{Caruana},~\bfnm{Michael}\binits{M.}} \AND
  \bauthor{\bsnm{L{\'e}vy},~\bfnm{Thierry}\binits{T.}}
(\byear{2007}).
\btitle{Differential Equations Driven by Rough Paths}.
\bseries{Lecture Notes in Math.}
\bvolume{1908}.
\bpublisher{Springer}, \blocation{Berlin}.
\bid{mr={2314753}}
\bptok{imsref}%
\end{bbook}
\endbibitem

\bibitem{Mal76}
\begin{binproceedings}[auto]
\bauthor{\bsnm{Malliavin},~\bfnm{P.}\binits{P.}}
(\byear{1978}).
\btitle{Stochastic calculus of variations and hypoelliptic operators}.
In \bbooktitle{Symp. on Stoch. Diff. Equations, Kyoto 1976}
\bpages{147--171}.
\bpublisher{Wiley}, \blocation{New York}.
\bptok{imsref}%
\end{binproceedings}
\endbibitem

\bibitem{Mal97SA}
\begin{bbook}[mr]
\bauthor{\bsnm{Malliavin},~\bfnm{Paul}\binits{P.}}
(\byear{1997}).
\btitle{Stochastic Analysis}.
\bseries{Grundlehren der Mathematischen Wissenschaften [Fundamental Principles
  of Mathematical Sciences]}
\bvolume{313}.
\bpublisher{Springer}, \blocation{Berlin}.
\bid{mr={1450093}}
\bptok{imsref}%
\end{bbook}
\endbibitem

\bibitem{MandVann68}
\begin{barticle}[mr]
\bauthor{\bsnm{Mandelbrot},~\bfnm{Benoit~B.}\binits{B.~B.}} \AND
  \bauthor{\bsnm{Van~Ness},~\bfnm{John~W.}\binits{J.~W.}}
(\byear{1968}).
\btitle{Fractional {B}rownian motions, fractional noises and applications}.
\bjournal{SIAM Rev.}
\bvolume{10}
\bpages{422--437}.
\bid{issn={0036-1445}, mr={0242239}}
\bptok{imsref}%
\end{barticle}
\endbibitem

\bibitem{Norr86}
\begin{bincollection}[mr]
\bauthor{\bsnm{Norris},~\bfnm{James}\binits{J.}}
(\byear{1986}).
\btitle{Simplified {M}alliavin calculus}.
In \bbooktitle{S\'eminaire de {P}robabilit\'es, {XX}, 1984/85}.
\bseries{Lecture Notes in Math.}
\bvolume{1204}
\bpages{101--130}.
\bpublisher{Springer}, \blocation{Berlin}.
\bid{doi={10.1007/BFb0075716}, mr={0942019}}
\bptok{imsref}%
\end{bincollection}
\endbibitem

\bibitem{Nual06}
\begin{bbook}[mr]
\bauthor{\bsnm{Nualart},~\bfnm{David}\binits{D.}}
(\byear{2006}).
\btitle{The {M}alliavin Calculus and Related Topics},
\bedition{2nd} ed.
\bpublisher{Springer}, \blocation{Berlin}.
\bid{mr={2200233}}
\bptok{imsref}%
\end{bbook}
\endbibitem

\bibitem{NuaSau09}
\begin{barticle}[mr]
\bauthor{\bsnm{Nualart},~\bfnm{David}\binits{D.}} \AND
  \bauthor{\bsnm{Saussereau},~\bfnm{Bruno}\binits{B.}}
(\byear{2009}).
\btitle{Malliavin calculus for stochastic differential equations driven by a
  fractional {B}rownian motion}.
\bjournal{Stochastic Process. Appl.}
\bvolume{119}
\bpages{391--409}.
\bid{doi={10.1016/j.spa.2008.02.016}, issn={0304-4149}, mr={2493996}}
\bptok{imsref}%
\end{barticle}
\endbibitem

\bibitem{SamKilMar93}
\begin{bbook}[mr]
\bauthor{\bsnm{Samko},~\bfnm{Stefan~G.}\binits{S.~G.}},
  \bauthor{\bsnm{Kilbas},~\bfnm{Anatoly~A.}\binits{A.~A.}} \AND
  \bauthor{\bsnm{Marichev},~\bfnm{Oleg~I.}\binits{O.~I.}}
(\byear{1993}).
\btitle{Fractional Integrals and Derivatives: Theory and Applications}.
\bpublisher{Gordon and Breach}, \blocation{Yverdon}.
\bid{mr={1347689}}
\bptok{imsref}%
\end{bbook}
\endbibitem

\bibitem{SVSupport}
\begin{binproceedings}[mr]
\bauthor{\bsnm{Stroock},~\bfnm{Daniel~W.}\binits{D.~W.}} \AND
  \bauthor{\bsnm{Varadhan},~\bfnm{S.~R.~S.}\binits{S.~R.~S.}}
(\byear{1972}).
\btitle{On the support of diffusion processes with applications to the strong
  maximum principle}.
In \bbooktitle{Proceedings of the {S}ixth {B}erkeley {S}ymposium on
  {M}athematical {S}tatistics and {P}robability ({U}niv. {C}alifornia,
  {B}erkeley, {C}alif., 1970/1971), {V}ol. {III}: {P}robability Theory}
\bpages{333--359}.
\bpublisher{Univ. California Press}, \blocation{Berkeley, CA}.
\bid{mr={0400425}}
\bptok{imsref}%
\end{binproceedings}
\endbibitem

\end{thebibliography}
\end{document}